%% file: BreitkopfUlbrich2026_GRP_v2.tex
\documentclass[onefignum,onetabnum]{siamart171218}

\newcommand{\R}{\mathbb{R}} %
\newcommand{\N}{\mathbb{N}} %
\newcommand{\D}{\mathcal{D}} %
\newcommand{\U}{\mathcal{U}} %
\newcommand{\V}{\mathcal{V}} %
\newcommand{\It}{\mathcal{I}_t} %
\newcommand{\St}{\mathcal{S}_t} %
\newcommand{\Uad}{\mathcal{U}_{\mathrm{ad}}} 
\newcommand{\Vad}{\mathcal{V}_{\mathrm{ad}}} 
\newcommand{\Vref}{\bar{\mathcal{V}}} 
\newcommand{\dom}{\bar{D}} 
\newcommand{\intd}{\mathrm{d}} %
\newcommand{\dx}{\intd x}
\newcommand{\dt}{\intd t}
\newcommand{\ds}{\intd s}
\newcommand{\xs}{x_0} 
\newcommand{\el}{\ell} 
\newcommand{\x}{\bar{x}} 
\newcommand{\y}{\bar{y}}
\newcommand{\urefinit}{\bar{u}} 
\newcommand{\w}{\bar{w}}
\newcommand{\lambdamax}{\lambda_{\mathrm{max}}}
\newcommand{\Sref}{\bar{\mathcal{S}}} 
\newcommand{\pcont}{PC^0(\dom)} 
\newcommand{\pdiff}{PC^1(\dom)}
\newcommand{\Yval}{\mathcal Y_0} 

\usepackage{stmaryrd} 

\input{ex_shared}

\ifpdf
\hypersetup{
  pdftitle={A Sensitivity Calculus for 
  	Optimal Control of the Generalized Riemann Problem for 
Hyperbolic Systems of Conservation Laws},
  pdfauthor={Jannik Breitkopf, and Stefan Ulbrich}
}
\fi

\begin{document}

\maketitle

\begin{abstract}
	We develop a sensitivity calculus for 
	entropy solutions of the Generalized Riemann Problem (GRP)
	for strictly hyperbolic systems of conservation laws
	where the control is the initial state.
	The GRP has a discontinuous initial state with exactly one
	discontinuity and continuously differentiable ($C^1$) states left and right of it.
	The control consists of the $C^1$ parts of the initial state and
	the position of the discontinuity. 
	Solutions of the problem are generally discontinuous since they contain shock curves.
	We assume the time horizon $T>0$ to be sufficiently small such that no
	shocks interact and no new shocks are generated.
	Moreover, we assume that no rarefaction waves occur and that
	the jump of the initial state is sufficiently small.
	
	Since the shock positions depend on the control, a transformation
	to a reference space is used
	to fix the shock positions.
	In the reference space, we prove that the solution of the GRP
	between the shocks is continuously Fr\'echet differentiable from the control space to $C^0$.
	In physical coordinates, this implies that the shock curves in $C^1$
	and the states between the shocks in the topology of $C^0$ 
	depend continuously Fr\'echet differentiable on the control.
	As a consequence, we obtain the Fr\'echet differentiability of 
	tracking type objective functionals. 

        The results provide an analytical foundation for derivative-based
        methods for the optimal control of piecewise $C^1$ entropy solutions
        of systems of conservation laws, the study of the adjoint equation
        and the design of numerical schemes in this context.
\end{abstract}



\section{Introduction}
We consider optimal initial-value control problems for
entropy solutions of the Generalized Riemann Problem (GRP)
for strictly hyperbolic systems of balance laws of the form
\begin{equation} \label{GRP} \tag{GRP}
\begin{aligned}
y_t + f(y)_x &= g(\cdot ,y) && \text{on } \D \subset [0,T] \times \R , \\
y (0, \cdot) &= u_0 && \text{on } [-\el,\el].
\end{aligned}
\end{equation}
We assume that $g \in C^{1,1}(\D \times \Yval ; \R^n)$ and $f \in C^{2,1}(\Yval; \R^n)$
with a set $\Yval \subset \R^n$ specified below.
The initial state is piecewise smooth
with one discontinuity located at $x = x_0$, i.e.,
\begin{equation} \label{initial_data_GRP}
u_0(x) = \left\{
    \begin{array}{ll}
        u_l(x) & \text{if } x \in [-\el, x_0), \\
        u_r(x) & \text{if } x \in (x_0, \el].
    \end{array} \right. 
\end{equation}
For a suitable open subset $\Vad \subset \V \coloneqq C^1 ( [-\el, \varepsilon]; \R^n ) 
\times C^1 ( [-\varepsilon, \el]; \R^n )$, the control is
\begin{equation*}
u = (u_l, u_r, \xs) \in \Uad \coloneqq \Vad \times (-\varepsilon, \varepsilon) .
\end{equation*}
We consider for suitable $a<b$ with $\{ T \} \times [a,b] \subset \D$ the objective functional 
\begin{equation} \label{objective_fcntal}
u \in \Uad \mapsto J(u) \coloneqq \int_a^b \Phi \left( y(T, x), y_d(x) \right) \dx
\quad \text{subject to } y \text{ solves \eqref{GRP}}
\end{equation}
with $\Phi \in C^{1,1}_{\mathrm{loc}}(\R^n \times \R^n)$ and a desired state $y_d \in BV([a,b];\R^n)$.

The aim of this paper is to develop a sensitivity calculus for the dependence of solutions
of \eqref{GRP} on the control $u_0$. In particular, this will imply
the differentiability of the objective functional \eqref{objective_fcntal}.

\textbf{Assumptions.} The problem \eqref{GRP} is assumed to be strictly hyperbolic
on a non-empty, open and convex set $\Yval \subset \R^n$, i.e.,
the Jacobian matrix $A(y) = f'(y)$ has real and distinct eigenvalues
$\lambda_1(y) < \ldots < \lambda_n(y)$ for all $y \in \Yval$.
All eigenvalues and eigenvectors of $A$ are assumed to be bounded in $C^{1,1}(\Yval)$.
We assume that there are $\lambdamax > 0$ 
and disjoint intervals $\Lambda_i = [\lambda_i^-, \lambda_i^+]$ with
\begin{equation} \label{lambda_ass}
	 | \lambda_i(y) | \leq \lambdamax, \quad
	 \lambda_i(y) \in \Lambda_i \qquad
	 \forall y \in \Yval, \ \forall i \in \{1, \ldots, n \} .
\end{equation}
We denote the minimal difference of the eigenvalues by 
\begin{equation} \label{min_speeddiff}
	\eta_{\min} = \min_{i=1,\ldots,n-1} (\lambda_{i+1}^- - \lambda_i^+) > 0 .
\end{equation}
With $r_i(y), l_i(y) \in \R^n$ we denote the right and left 
eigenvectors of $A(y)$, i.e.,
\begin{equation} \label{evec_normalization}
A(y)r_i(y) = \lambda_i(y) r_i(y), \quad l_i(y)^\top A(y) = \lambda_i(y) l_i(y)^\top,
\qquad \forall y \in \Yval.
\end{equation}
Moreover, we assume $|r_i(y)| \equiv 1$ and
$l_i(y) ^\top r_j(y) = \delta_{ij}$ with the Kronecker delta $\delta_{ij}$.

The underlying domain of \eqref{GRP} is assumed to be
\begin{equation} \label{domain_physical}
\D = \left\{ (t,x) \in [0,T] \times \R \ | \
-\el + t \lambdamax \leq x \leq \el - t \lambdamax \right\}
\end{equation}
for a small time horizon $T>0$.
Note that this ensures that any solution of \eqref{GRP} on $\D$
is fully determined by the initial values in $t=0$ and no
boundary data is required.

\textbf{Related literature.} For scalar conservation laws, 
differentiability properties of entropy solutions w.r.t.\ the initial state
are investigated, e.g., in \cite{BressanGuerra1997, BJ1999, BardosPironneau2002, BardosPironneau2002b}.
In \cite{Ulbrich2002}, general variations of the smooth parts
and positions of the discontinuities of the initial state are considered,
in particular implying the differentiability of the objective functional \eqref{objective_fcntal}.
An adjoint gradient representation is derived in \cite{Ulbrich2003}.
The results are extended to initial boundary value problems in \cite{PfaffUlbrich2015, SchmittUlbrich2021}.
Optimal control problems of scalar conservation laws are studied
in \cite{Jacquet2006a, Zuazua2008, Zuazua2016}.

For systems of conservation laws, 
existence results for entropy solutions with discontinuous initial states 
exist within the class of $BV$-functions \cite{Bressan2005} and
for piecewise smooth solutions \cite{LiYu1985}. 
A variational calculus for piecewise Lipschitz continuous solutions
w.r.t.\ directional variations of the initial data is given in \cite{BressanMarson1995}
and corresponding optimality conditions are derived in \cite{BressanMarson1995b, BressanShen2007}.
Another approach by horizontal shifts of initial controls is proposed in \cite{Bianchini2000}.
A formal sensitivity and adjoint computation for systems is carried out 
in \cite{Jacquet2006b} and in \cite{BardosPironneau2005}
for the case of piecewise smooth solutions with exactly one shock curve.
Here, it is assumed that the shock and the solution on both sides of the shock
depends differentiably on a parameter.
The linearized stability of solutions of hyperbolic systems of conservation laws 
is discussed in \cite{GodlweskiRaviart1999}.

\textbf{Notation.}
If a map $F : U \to V$ is Fr\'echet differentiable, 
we denote the Fr\'echet derivative in $\bar u \in U$ by $d_u F(\bar u) \in \mathcal{L}(U, V)$.
With $\mathbf{1}_{[a,b]}$ we denote the indicator function on the interval $[a,b]$, i.e.,
$\mathbf{1}_{[a,b]}(x) = 1$ if $x \in [a,b]$ and $\mathbf{1}_{[a,b]}(x) = 0$ otherwise.
We denote $[n] \coloneqq \{1, \ldots, n\}$ and $[n]_0 \coloneqq \{0, \ldots, n\}$.
For brevity, we drop indices for norm bounds of eigenvalues and eigenvectors of $f'$, i.e.,
\begin{equation} \label{notation_norm_rllambda}
	\| X \| = \max_{i \in [n]} \sup_{y \in \Yval} | X_i(y) |
	\qquad \text{for } X \in \{ r, \, l, \, \nabla l, \, \nabla \lambda \} .
\end{equation}
We write $h(x\pm) = \lim_{s \to x^\pm} h(s)$ 
for $h \in BV(\R; \R^n)$ and denote the jump at $x$ by
\begin{equation} \label{notation_jump}
	\llbracket h(x) \rrbracket = h(x-) - h(x+).
\end{equation}

\textbf{Organization.} In \cref{sec:sol_existence}, we derive existence and stability
results for solutions of the GRP. We define the physically relevant
type of solutions in \cref{sec:entropysolution}.
Based on the Riemann Problem, we define the set of admissible controls 
in \cref{sec:riemannproblem} and construct a fixed space-time domain 
in \cref{sec:refspace} which serves as a reference space.
\Cref{sec:slproblems} is concerned with existence and stability properties
of the semilinear problem and \cref{sec:qlproblems} with the quasilinear problem.

The main results of this paper are in \cref{sec:diff}.
We prove the differentiability of the solution operator in the reference space
in \cref{sec:diffref} and deduce further differentiability properties in physical coordinates
in \cref{sec:diffphys}.

\section{Existence and Stability Properties of Discontinuous Solutions}
\label{sec:sol_existence}
In this section, we prove existence and stability properties
of discontinuous solutions of semi- and quasilinear hyperbolic
systems of conservation laws.

Existence results and a priori estimates for
hyperbolic conservation laws are obtained
by the method of characteristic curves.
Consider a domain $\Omega \subset [0,T] \times \R$ and suppose there are
$y, z \in C^1(\Omega ; \R^n)$ such that it holds in the classical sense that
\begin{equation} \label{standard_pde}
y_t + A(z) y_x = g(\cdot, y) \quad \text{on } \Omega .
\end{equation}
Multiplying with $l_i(z)^\top$ yields for 
the characteristic variables $y_i \coloneqq l_i(z)^\top y$ that
\begin{equation} \label{characteristic_form}
(y_i)_t + \lambda_i(z) (y_i)_x = l_i(z)^\top g(\cdot, y) 
+ \left( l_i(z)_t + \lambda_i(z) l_i(z)_x \right) ^\top y .
\end{equation}
Let $\xi_i = \xi(\cdot; t,x,z)$ be the solution of the ordinary differential equation
\begin{equation} \label{char_curve}
\xi_i'(s) = \lambda_i \left( z(s, \xi_i(s) \right), \quad
\xi_i(t) = x. 
\end{equation}
Then, \eqref{characteristic_form} shows that a solution of \eqref{standard_pde} satisfies
\begin{equation*}
\frac{\mathrm{d}}{\dt} y_i(t, \xi_i(t)) =
\left[ l_i(z)^\top g(\cdot, y) + \left( l_i(z)_t + \lambda_i(z) l_i(z)_x \right) ^\top y \right] (t, \xi_i(t)).
\end{equation*}

\subsection{Entropy Solutions and Jump Condition} \label{sec:entropysolution}
In general, weak solutions of \eqref{GRP} are not unique,
see, e.g., \cite{Bressan2005, Smooller1983}. 
To select the physically relevant solution,
the \emph{Lax Entropy Condition} \cite{Lax1957} is imposed along shock curves.

\begin{definition} \label{def:entropysolution}
A function $y \in BV(\D ; \R^n)$ with values in $\Yval$ is an \emph{entropy solution} of \eqref{GRP} if
\begin{enumerate}
\item $y$ is a weak solution, i.e., $\mathrm{ess \ lim}_{t \searrow 0} 
\| y(t, \cdot) - u_0 \|_{L^1([-\ell, \ell]; \R^n)} = 0 $ and 
\begin{equation*}
\int_{\D} ( y^\top \phi_t + f(y)^\top \phi_x 
+ g(\cdot, y)^\top \phi ) \ \mathrm{d} (t,x)
= 0 \quad \forall \phi \in C^{\infty}(\D;\R^n)
\text{ with } \phi | _{\partial \D} = 0.
\end{equation*}    
\item Let $J \subset \D$ be the set of 
approximate jump discontinuities of $y$.
For almost all $(t,x) \in J$
w.r.t.\ the one dimensional Hausdorff measure
one can associate
a normal vector $\hat n = (-s, 1)^\top \in \R^2$ 
and one-sided limits $y^{\pm}$, see, e.g., \cite{Dafermos2016, Volpert1967}.
The Lax entropy condition
requires that for some $j \in [n]$ it holds
\begin{equation} \label{entropy_cond_speeds}
	\lambda_j(y^- ) \geq s \geq \lambda_j(y^+ ), 
	\quad \lambda_k(y^{\pm} ) > s > \lambda_i(y^{\pm} )
	\quad \forall i < j < k.
\end{equation}
\end{enumerate}    
\end{definition}

A weak solution of \eqref{GRP} satisfies the Rankine--Hugoniot jump condition.
\begin{lemma}
With the notation of \cref{def:entropysolution},
let $y \in BV(\D ; \R^n)$ with values in $\Yval$ be a weak solution of \eqref{GRP}
and assume it is discontinuous along $t \mapsto (t, \xi(t))$.
Then it holds 
\begin{equation} \label{RH_cond}
\dot{\xi}(t) \left( y^+(t) - y^-(t) \right) 
= f \left( y^+(t) \right) - f \left( y^-(t) \right) \quad \forall t \in [0,T].
\end{equation}
\end{lemma}
A proof can be found in \cite[Theorem 4.2]{Bressan2005}.
The condition \eqref{RH_cond} can be equivalently
reformulated in terms of the jump of components of $y$
w.r.t.\ a fixed basis. Let 
\begin{equation} \label{Aavg}
	A(y_1, y_2) = \int_0^1 A \left( sy_1 + (1-s) y_2 \right) \ds
\end{equation}
with eigenvalues $\lambda_i(y_1, y_2)$
and eigenvectors $r_i(y_1, y_2)$ and $l_i(y_1, y_2)$,
normalized as in \eqref{evec_normalization}.
Note that $A(y_1, y_2)$ is strictly hyperbolic for $y_1, y_2 \in \Yval$
with $|y_1 - y_2|$ small, using that $\Yval$ is convex.

\begin{lemma} \label{lem:GRP_jump_equivalent}
	Assume a solution $y \in BV(\D ; \R^n)$ of \eqref{GRP} with values in $\Yval$
	is discontinuous along $(t, \xi(t))$ for $t \in [0,T]$ with a curve $\xi \in C^1([0,T])$.
	Let $y^{\pm}(t) = y(t, \xi(t) \pm)$ and assume that $| y^+(t) - y^-(t) |$ is sufficiently small.
	Let $l^*_k, r^*_k \in \R^n$ for $k \in [n]$ be biorthogonal bases of $\R^n$,
	i.e., $(l^*_i)^\top r^*_k = \delta_{ik}$ for all $i,k \in [n]$, and assume
	\begin{equation} \label{fixed_basis_nondeg}
		\big| l_i ( y^+(t) , y^-(t) ) ^\top r^*_i \big| \geq \tfrac12
		\qquad \forall i \in [n], \ \forall t \in [0,T].
	\end{equation}
	Then \eqref{RH_cond} holds if and only if there is $j \in [n]$ such that
	\begin{equation} \label{GRP_jump_speed}
		\dot{\xi}(t) = \lambda_j \big( y^+(t) , y^-(t) \big) 
		\quad \text{for all } t \in [0,T],
	\end{equation}
	and it holds for all $i \neq j$ that
	\begin{equation} \label{GRP_jump}
		(l^*_i)^\top y^+ = 
		(l^*_i)^\top y^-	
		- \sum_{k \neq i} \frac{ l_i ( y^+ , y^- ) ^\top r^*_k}
		{l_i ( y^+ , y^- ) ^\top r^*_i} (l^*_k)^\top (y^+ -y^-)
		\quad \text{on } [0,T].
	\end{equation}
\end{lemma}
\begin{proof}
	The matrix $A(y^-, y^+)$ from \eqref{Aavg} is also strictly hyperbolic
	for sufficiently small $| y^+ - y^- |$, 
	and \eqref{RH_cond} is equivalent to \eqref{GRP_jump_speed} together with
	\begin{equation} \label{jump_revec}
		y^+(t) - y^-(t) = c \, r_j ( y^+(t) , y^-(t) ) 
	\end{equation}
	for some $j \in [n]$ and $c \in \R$.
	Therefore, \eqref{jump_revec} is equivalent to 
	\begin{equation} \label{shock_orthogonal}
		l_i ( y^+ , y^- ) ^\top ( y^+ - y^- ) = 0
		\qquad \forall i \neq j.
	\end{equation}
	Inserting $y^{\pm} = \sum_{k=1}^n ( (l^*_k)^\top y^{\pm} ) r^*_k$ 
	in \eqref{shock_orthogonal} and dividing by $l_i ( y^+ , y^- ) ^\top r^*_i \neq 0$
	yields \eqref{GRP_jump}	for all $i \neq j$.
	Conversely, \eqref{GRP_jump} for all $i \neq j$ implies \eqref{shock_orthogonal}
	and hence \eqref{jump_revec}, 
	which together with \eqref{GRP_jump_speed} yields \eqref{RH_cond}.
\end{proof}

\subsection{Admissible Controls} \label{sec:riemannproblem}
In this section, we define the set of admissible controls $\Uad$.
This depends on the structure of solutions of the Riemann Problem
\begin{equation} \label{RP} 
\begin{aligned}
y_t + f(y)_x = 0 \quad \text{on } [0,T] \times \R,  \qquad
y (0, \cdot) = u_0 = \left\{
    \begin{array}{ll}
        u_L & \text{if } x<0, \\
        u_R & \text{if } x>0,
    \end{array} \right.  
\end{aligned} 
\end{equation}
where $u_L \neq u_R \in \Yval$.
The existence of entropy solutions for \eqref{RP} can be shown
under the following standard assumptions on the flux function $f$, see, e.g., \cite{Bressan2005}. 

\begin{definition}
If $\nabla \lambda_i(y) ^\top r_i(y) \neq 0$ for all $y \in \Yval$,
the $i$-th characteristic field is \emph{genuinely nonlinear}.
If $\nabla \lambda_i(y) ^\top r_i(y) = 0$ for all $y \in \Yval$,
it is \emph{linearly degenerate}.
\end{definition}

\begin{theorem} \label{thm:thm_sol_RP}
Assume that all characteristic fields $i \in [n]$ are either
genuinely nonlinear or linearly degenerate. 
If $u_L, u_R \in \Yval$ and $|u_L - u_R|$ is sufficiently small, 
there is a unique entropy solution $\hat{y}$ of \eqref{RP}.
There are $\hat{y}_1, ..., \hat{y}_{n-1} \in \Yval$, 
$\mathcal{R} \dot{\cup} \mathcal{S} = [n]$, 
speeds $s_1^- \leq s_1^+ < \ldots < s_n^- \leq s_n^+$ 
with $s_i^- = \lambda_i(\hat{y}_{i-1}) < s_i^+ = \lambda_i(\hat{y}_{i})$ 
for $i \in \mathcal{R}$ and 
$s_i \coloneqq s_i^- = s_i^+ = \lambda_i(\hat{y}_{i-1}, \hat{y}_{i})$ 
for $i \in \mathcal{S}$ and $W_i \in C^1([s_i^-, s_i^+];\R^n)$
for $i \in \mathcal{R}$, such that
\begin{equation} \label{sol_RP}
\hat{y} (t, x) = \left\{
    \begin{array}{lll}
        \hat{y}_0 = u_L & \text{if } x/t < s_1^-,\\
        \hat{y}_i & \text{if } s_{i}^+ < x/t < s_{i+1}^-, & \forall i \in [n-1], \\
        W_i(x/t) & \text{if } s_{i}^- < x/t < s_{i}^+, & \forall i \in \mathcal{R}, \\
        \hat{y}_n = u_R & \text{if } x/t > s_n^+ . \\
    \end{array} \right. 
\end{equation}
\end{theorem}
If $j \in \mathcal{S}$ and the $j$-th characteristic field
is genuinely nonlinear, the solution $\hat{y}$ contains
a $j$-shock curve. If the field is linearly degenerate,
the solution contains a contact discontinuity.
If $j \in \mathcal{R}$, it contains a $j$-rarefaction wave.
\begin{proof}
The detailed proof can be found in \cite[Theorem 5.3]{Bressan2005}.
We only give a short summary, since the notation and construction
of $\hat{y}$ is necessary for our further considerations.
It can be shown that for all characteristic fields $i \in [n]$
and any left state $u_0 \in \Yval$ there exists $\sigma_0 > 0$ and a function
$\Psi_i(\cdot)(u_0) \in C^2([-\sigma_0, \sigma_0]; \Yval)$ with $\Psi_i(0)(u_0) = u_0$
and the following properties.
If the $i$-th field is genuinely
nonlinear, any states $\Psi_i(\sigma)(u_0)$ for $\sigma \in [0, \sigma_0]$ can be
connected to $u_0$ by an $i$-rarefaction wave. 
All states $\Psi_i(\sigma)(u_0)$ for $\sigma \in [-\sigma_0, 0)$ can be
connected to $u_0$ by an $i$-shock curve satisfying the 
Rankine-Hugoniot jump condition \eqref{RH_cond}.
If the $i$-th field is linearly degenerate,
any state $\Psi_i(\sigma)(u_0)$ for $\sigma \in [-\sigma_0, \sigma_0]$ can be connected
to $u_0$ by a contact discontinuity.
The concatenation of the maps $\Psi_i$ for $i \in [n]$ 
defines for some $\sigma_1 > 0$ sufficiently small
the map $\Lambda(\cdot)(u_L) : [-\sigma_1, \sigma_1]^n \to \Yval$ by
\begin{equation} \label{implicit_map_RP}
\omega_0 = u_L, \quad \omega_i = \Psi_i(\sigma_{i})(\omega_{i-1}) \ \forall i \in [n],
\quad \Lambda(\sigma)(u_L) = \omega_n.
\end{equation}
The unique existence of an entropy solution $\hat{y}$ of \eqref{RP}
then follows by proving the unique existence of $\sigma \in \R^n$ such that
$u_R = \Lambda(u_L)(\sigma)$, which follows from the Implicit Function Theorem.
If $\sigma \leq 0$, the solution of \eqref{sol_RP}
is fully determined by $\hat{y}_i = \omega_i$ for all $i \in [n]_0$.
If $\sigma_i > 0$ for some $i \in [n]$, the intermediate states
$\hat{y}_{i-1}, \hat{y}_i$ must be connected by an $i$-rarefaction wave.
\end{proof}

We can now define the admissible set $\Vad$ of initial states for \eqref{GRP}.

\begin{assumption} \label{ass:Uad1}
Let $u_L \neq u_R \in \Yval$ with $|u_L - u_R|$ sufficiently small
and assume that the associated Riemann Problem
\eqref{RP} has an entropy solution only consisting of shock curves or
contact discontinuities, i.e., with the notation \eqref{implicit_map_RP}, 
there is a unique $\bar{\sigma} \in \R^n$ such that
$u_R = \Lambda(u_L)(\bar \sigma)$. 
Furthermore, assume that 
$\bar{\sigma}_i < 0 $ for all genuinely nonlinear fields $i \in [n]$.
For constants $M_0, M_1>0$, we define
\begin{equation} \label{def_Uad}
	\begin{aligned}
		\Vad = \big\{ (u_l, u_r) \in \V \, : \, 
		& |u_l(0)-u_L|, |u_r(0)-u_R| < M_0, \
		\| u_l'\|, \| u_r'\| < M_1, \\
		& \mathrm{img}(u_l), \mathrm{img}(u_r) \subset \Yval \big\},
	\end{aligned}
\end{equation}
where $\| \cdot \|$ denotes the $C^0$-norm on the 
respective intervals $[-\el, \varepsilon]$ and $[-\varepsilon, \el]$.
It holds for all $(u_l, u_r) \in \Vad$ and $x \in [-\varepsilon, \varepsilon]$
that $ |u_l(x) - u_L | \leq M_0 + M_1 \varepsilon$ and
$ |u_r(x) - u_R | \leq M_0 + M_1 \varepsilon$.
For fixed $M_1 > 0$, choosing $M_0, \, \varepsilon > 0$ sufficiently small
(depending on $M_1$), it follows for any left and right states
\begin{equation} \label{jump_val_neighborhood}
(v_L, v_R) \in \mathcal{J} \coloneqq 
\big\{ \left(u_l(x), u_r(x) \right) \in \Yval \times \Yval \ | \ 
(u_l, u_r) \in \Vad, \, 
- \varepsilon \leq x \leq \varepsilon \big\} 
\end{equation}
that the Riemann Problem \eqref{RP} with
initial state $(v_L,v_R)$ has a unique entropy solution
only consisting of shock curves or contact discontinuities, i.e., it holds that 
$v_R = \Lambda(v_L)(\sigma)$ satisfies $\sigma_i < 0$
for all genuinely nonlinear fields $i \in [n]$. 
\end{assumption}
With the above choice of $M_0, \, \varepsilon > 0$
and $\Vad$ from \eqref{def_Uad}, we define
\begin{equation} \label{Uad_physicalcoord}
	\Uad \coloneqq \Vad \times (-\varepsilon, \varepsilon) .
\end{equation}
Note that $M_1$ is not assumed to be small.
Moreover, $\Vad$ is open in $\V$, since $\mathrm{img}(u_l)$, $\mathrm{img}(u_r)$
are compact and $\Yval$ is open.

\begin{remark} \label{rem:RP_soloperator_cont}
	By the Implicit Function Theorem, there exists a neighborhood of $(u_L,u_R)$ on
	which $v_R = \Lambda(v_L)(\sigma)$ is uniquely solvable and $\sigma \in \R^n$
	depends continuously on $(v_L, v_R)$. For $M_0, \varepsilon > 0$ sufficiently small,
	$\mathcal{J}$ from \eqref{jump_val_neighborhood} is contained in this neighborhood.
	To avoid rarefaction waves, it is necessary to assume that $\bar{\sigma}_i < 0$
	for genuinely nonlinear fields. Otherwise, any perturbation of the initial
	state of the Riemann Problem could produce a solution containing a
	rarefaction wave.
\end{remark}

\begin{remark}
	The exclusion of rarefaction waves is a structural restriction.
	Inside a rarefaction wave, the solution is by \eqref{sol_RP} of the
	form $W_i(x/t)$, i.e., the spatial derivative is of order $\mathcal{O}(t^{-1})$.
	Extending the calculus to solutions containing
	rarefaction waves therefore requires exploiting the self-similar structure
	of the rarefaction wave. This is left to future work.
\end{remark}

To express coupling conditions of solutions to \eqref{GRP}
across shock curves, we introduce the following reference basis.
\begin{definition} \label{def:fixed_shock_basis}
	Let $\hat{y}_0, \ldots, \hat{y}_n \in \Yval$ be the states \eqref{sol_RP} of the
	entropy solution of the Riemann Problem \eqref{RP} associated with $(u_L, u_R)$
	from \cref{ass:Uad1}.
	For $j \in [n]$, the matrix $A ( \hat{y}_{j-1} , \hat{y}_j )$ from \eqref{Aavg}
	is strictly hyperbolic, since $| \hat{y}_{j-1} - \hat{y}_j |$ is small
	for $|u_L - u_R|$ sufficiently small by \cref{rem:RP_soloperator_cont}.
	We denote its eigenvectors by
	\begin{equation*} 
		l^{*,j}_k \coloneqq l_k ( \hat{y}_{j-1} , \hat{y}_j ) , \qquad
		r^{*,j}_k \coloneqq r_k ( \hat{y}_{j-1} , \hat{y}_j ) ,
		\qquad \forall k \in [n],
	\end{equation*}
	normalized as in \eqref{evec_normalization}, i.e.,
	$( l^{*,j}_i )^\top r^{*,j}_k = \delta_{ik}$ for all $i,j,k \in [n]$.
	Since $l_i ( \hat{y}_{j-1} , \hat{y}_j ) = l^{*,j}_i$, the Lipschitz continuity of
	$l_i$ shows that \eqref{fixed_basis_nondeg} holds
	for $l^*_k = l^{*,j}_k$, $r^*_k = r^{*,j}_k$ and $y^-, y^+ \in \Yval$
	with $|y^- - \hat{y}_{j-1}| + |y^+ - \hat{y}_j|$ sufficiently small.
\end{definition}

\subsection{The Reference Space} \label{sec:refspace}
To cope with the varying positions of shock curves, 
we define a fixed space-time domain serving as a
reference space and equivalently reformulate \eqref{GRP} as done in \cite{LiYu1985}. 
The reference space is obtained by a transformation of the space variable.
As a consequence, the positions of shock curves are fixed.

All quantities in the reference space will be denoted with a bar,
e.g., the solution $y$ to \eqref{GRP} will be associated with the
transformed solution $\y$ in the reference space.
The original coordinates will also be called the physical coordinates.
The reference domain is motivated by the structure
of an entropy solution of the Riemann Problem, see \eqref{sol_RP}.
We denote by $s_1 < \ldots < s_n$ the shock speeds
of the solution of the Riemann Problem with initial state $(u_L, u_R)$,
see \cref{ass:Uad1}.
Define the sectors 
\begin{equation} \label{def_D_i}
\dom_i = \{ (t, \x) \in [0,T] \times \R \ | \ 
s_i t \leq \x \leq s_{i+1} t \}, \quad \forall i \in [n-1].
\end{equation}
For the sectors left of the $1$-shock and right of the $n$-shock curves,
we define the reference speeds $s_{(\ell)} = \lambdamax$ and
$s_{(r)} = - \lambdamax$ and define
\begin{equation} \label{speeds_sr_sl}
\begin{aligned} 
\dom_0 &= \{(t, \x) \in [0,T] \times \R \ | \ 
- \el + s_{(\ell)} t \leq \x \leq s_1 t \}, \\ 
\dom_n &= \{ (t, \x) \in [0,T] \times \R \ | \ 
s_n t \leq \x \leq \el + s_{(r)} t \},
\end{aligned}
\end{equation}
and $\dom = \cup_{i=0}^n \dom_i$.
The interior boundaries are denoted by
\begin{equation} \label{Sigma_i}
	\Sigma_i = \{ (t, s_i t) \ | \ t \in [0,T] \} \quad \forall i \in [n].
\end{equation}
Based on the definition of the sets $\dom_0,\ldots,\dom_n$, we define
\begin{equation} \label{def_X}
PC^{k}(\dom) \coloneqq \bigtimes_{i=0}^n C^k(\dom_i), \quad k \in \{0, 1 \} .
\end{equation}
We identify $\bar{y} \in PC^k(\dom)$ by 
$\bar{y} = (\bar{y}^j)_{j \in [n]_0}$ with $\bar{y}^j \in C^k(\dom_j)$ 
and the induced norm
$\| \bar{y} \|_{PC^k(\dom)} = \max_j  \| \bar{y}^j \|_{C^k(\dom_j)}$.
Clearly, $(PC^{k}(\dom), \| \cdot \|_{PC^k(\dom)} )$ is a Banach space.

With $\bar{y} \in \pcont$ taking values in $\Yval$ and $\xs \in \R$ we associate the curves
\begin{align} 
\xi_{(\ell)}^{(\bar{y}, \xs)}(t) &= \int_0^t 
\lambda_n \big( \bar{y}^{0} (\tau, -\el + s_{(\ell)} \tau ) \big) \ \mathrm{d} \tau -\el , \label{xil_curve} \\ 
\xi_{(r)}^{(\bar{y}, \xs)}(t) &= \int_0^t \lambda_1 \big( \bar{y}^{n} (\tau, \el + s_{(r)} \tau) \big) 
\ \mathrm{d} \tau + \el , \label{xir_curve} \\
\xi_j^{(\bar{y}, \xs)}(t) & = \int_0^t \lambda_j 
\big( \bar{y}^{j-1} (\tau, s_j \tau), \bar{y}^{j} (\tau, s_j \tau) \big)
\ \mathrm{d} \tau + \xs .  \label{xij_curve_shock}
\end{align}
We can now define the state-dependent coordinate transformation.
\begin{definition}[Space transformation on $\dom_j$] \label{def:def_spacetrafo}
Let $\bar{y} \in \pcont$ 
and $\xs \in \R$. For $j \in [n-1]$ and $(t, \x) \in \dom_j$, 
the space variable $\bar{x}$ is associated with the location
\begin{equation} \label{x_trafo_middle}
x = x^{(\bar{y}, \xs),j}(t,\x) \coloneqq
\frac{\bar{x} - s_jt}{(s_{j+1}-s_j)t} \xi_{j+1}^{(\bar{y}, \xs)}(t) 
+ \frac{s_{j+1} t - \bar{x} }{(s_{j+1}-s_j)t} \xi_{j}^{(\bar{y}, \xs)}(t) 
\end{equation}
in physical coordinates. 
For $(t, \bar{x}) \in \dom_0$, we set with the speed $s_{(\ell)}$ from \eqref{speeds_sr_sl}
\begin{equation} \label{x_trafo_left}
x^{(\bar{y}, \xs),0}(t,\bar{x}) \coloneqq 
    \frac{\bar{x} - (s_{(\ell)} t - \el)}{s_1t- (s_{(\ell)} t - \el)} \xi_{1}^{(\bar{y}, \xs)}(t) 
    + \frac{s_1 t - \bar{x} }{s_1t- (s_{(\ell)} t - \el)} \xi_{(\ell)}^{(\bar{y}, \xs)}(t) .
\end{equation}
The definition of $x^{(\bar{y}, \xs),n}(t,\bar{x})$ for $(t, \bar{x}) \in \dom_n$ is analogous
to \eqref{x_trafo_left}.
\end{definition}
For the well-definedness of the space transformation, $T > 0$ must be sufficiently small.

\begin{lemma} \label{lem:welldefinedness_spacetrafo}
Let $\bar{y} \in \pdiff$ with values in $\Yval$ and $\xs \in [-\varepsilon, \varepsilon]$
satisfy the following properties.
\begin{enumerate}
\item The jump in $t=0$ satisfies $(\bar{y}^0(0,0), \bar{y}^n(0,0)) \in \mathcal{J}$
with $\mathcal{J}$ as in \eqref{jump_val_neighborhood}.
\item Let $c_{\max} = \max_{j \in [n]_0} |\bar{y}^j(0,0)-\hat{y}_j|$
with the piecewise constant states $\hat{y}_j$ of the solution of the
Riemann Problem \eqref{RP} for the initial state $(u_L, u_R)$. 
\end{enumerate}  
Now assume that 
$ c_{\max}, \varepsilon, T > 0$ are sufficiently small such that
with $\lambdamax$ from \eqref{lambda_ass}, $\eta_{\min}$ from \eqref{min_speeddiff}
and the notation \eqref{notation_norm_rllambda}, it holds that
\begin{equation} \label{smallness_ass_welldef_spacetrafo}
    \| \nabla \lambda \| \big( (1+\lambdamax) \, T \| \bar{y} \|_{PC^1(\dom)} 
    + c_{\max} \big) \leq \eta_{\min} /8,
    \qquad \varepsilon  \leq \tfrac16 \el, 
    \qquad T \leq \frac{\el}{6 \lambdamax}.
\end{equation}
Then \eqref{xil_curve}, \eqref{xir_curve}, \eqref{xij_curve_shock}
do not intersect and satisfy for all $t \in [0,T]$ and $j \in [n-1]$
\begin{equation} \label{nondegcondition_shockcurves}
	\tfrac{1}{2} (s_{j+1} -s_j)t 
	\leq \xi_{j+1}^{(\bar{y}, \xs)}(t) - \xi_{j}^{(\bar{y}, \xs)}(t)  
	\leq 2 (s_{j+1} -s_j)t . 
\end{equation}
Moreover, we have for the curves $\xi_{(\ell)}^{(\bar{y}, \xs)}$, 
$\xi_{(r)}^{(\bar{y}, \xs)}$ from \eqref{xil_curve}, \eqref{xir_curve} that
\begin{equation} \label{nondegcondition_boundarycurves}
	\tfrac{1}{2} \el \leq \xi_{1}^{(\bar{y}, \xs)}(t) - \xi_{(\ell)}^{(\bar{y}, \xs)}(t)  \leq 2 \el, \qquad
	\tfrac{1}{2} \el \leq \xi_{(r)}^{(\bar{y}, \xs)}(t) - \xi_{n}^{(\bar{y}, \xs)}(t)  \leq 2 \el.
\end{equation}
\end{lemma}

\begin{proof}
By definition \eqref{xij_curve_shock} of $\xi_j^{(\bar{y}, \xs)}(t)$ it holds 
for $t \in [0,T]$ that
\begin{equation*}
\begin{aligned}
    | \dot{\xi} _j^{(\bar{y}, \xs)}(t) - s_j |
    & \leq | \dot{\xi} _j^{(\bar{y}, \xs)}(t) - \dot{\xi} _j^{(\bar{y}, \xs)}(0) | 
    + | \dot{\xi} _j^{(\bar{y}, \xs)}(0) - s_j | \\
    & \leq 2 \| \nabla \lambda \| \big( (1+\lambdamax) \, T \| \bar{y} \|_{PC^1(\dom)}
    + c_{\max} \big)
    \leq \eta_{\min} /4 
    \end{aligned}
\end{equation*}
and thus $| \xi _j^{(\bar{y}, \xs)}(t) - s_j t - \xs| \leq t \, \eta_{\min} /4$,
which proves \eqref{nondegcondition_shockcurves}.
Assertion \eqref{nondegcondition_boundarycurves} follows from the observation 
that $|\xi_{1}^{(\bar{y}, \xs)}(0) - \xi_{(\ell)}^{(\bar{y}, \xs)}(0) | \in [\el - \varepsilon, \el + \varepsilon]$
and both curves have a maximal speed difference of $2 \lambdamax$, analogously for the right part.
\end{proof}

For the rest of this section, we assume $\y \in \pdiff$ to satisfy the 
smallness assumptions \eqref{smallness_ass_welldef_spacetrafo}.
This ensures the well-definedness of the control-dependent sectors 
\begin{equation*} 
\begin{aligned}
D^{(\bar{y}, \xs)}_0 & = \lbrace (t,x) \ | \
0 \leq t \leq T, \ \xi_{(\ell)}^{(\bar{y}, \xs)}(t) \leq x \leq \xi_{1}^{(\bar{y}, \xs)}(t) \rbrace,  \\ 
D^{(\bar{y}, \xs)}_j & = \lbrace (t,x) \ | \ 
0 \leq t \leq T, \ \xi_{j}^{(\bar{y}, \xs)}(t) \leq x
\leq \xi_{j+1}^{(\bar{y}, \xs)}(t) \rbrace \quad \forall j \in [n-1], \\ 
D^{(\bar{y}, \xs)}_n & = \lbrace (t,x) \ | \  
0 \leq t \leq T, \ \xi_{n}^{(\bar{y}, \xs)}(t) \leq x \leq \xi_{(r)}^{(\bar{y}, \xs)}(t) \rbrace ,
\end{aligned}
\end{equation*}
and the inverse transformations
$(t,x) \in D^{(\bar{y}, \xs)}_j \mapsto (t, \bar{x}^{(\bar{y}, \xs),j}(t,x)) \in \dom_j$
for all $j \in [n]_0$.
It can be easily verified that for $(t,x) \in D^{(\bar{y}, \xs)}_j$ it holds that
\begin{equation} \label{inv_trafo_xbar}
\bar{x}^{(\bar{y}, \xs),j}(t,x) = 
\frac{x - \xi_{j}^{(\bar{y}, \xs)}(t)}{\xi_{j+1}^{(\bar{y}, \xs)}(t) - \xi_{j}^{(\bar{y}, \xs)}(t)} s_{j+1} t + \frac{\xi_{j+1}^{(\bar{y}, \xs)}(t) - x}{\xi_{j+1}^{(\bar{y}, \xs)}(t) - \xi_{j}^{(\bar{y}, \xs)}(t)} s_{j} t
\end{equation}
for $j \in [n-1]$. The cases $j \in \{ 0, n\}$ are analogous and skipped for brevity.

\begin{remark}
    The crucial point of the reference space above is that
    the sectors $D^{(\bar{y}, \xs)}_j$ in physical coordinates
    depend on the state $\bar{y}$, but the fixed sectors $\dom_j$ do not.
    This allows to handle the a-priori unknown location of the shock curves.
\end{remark}

\begin{lemma} \label{lema:spacetrafo_properties}
Let $\bar{y} \in \pdiff$, $\xs \in \R$
satisfy the assumptions
of \cref{lem:welldefinedness_spacetrafo}.
The transformations $x^{(\bar{y}, \xs),j}$
and $\bar{x}^{(\bar{y}, \xs),j}$ from \eqref{x_trafo_middle}, \eqref{x_trafo_left}
and \eqref{inv_trafo_xbar}
are twice continuously differentiable w.r.t.\ $(t, \x)$, and $(t,x)$, respectively.
Moreover,
\begin{equation} \label{spacetrafo_xderivative_bound}
	\partial_{\bar{x}} x^{(\bar{y}, \xs),j} , \,
	\partial_{x} \bar{x}^{(\bar{y}, \xs),j}  
	\in \left[ 1/2,2 \right] \quad \text{on } \dom_j 
	\text{ for all } j \in [n]_0. 
\end{equation}	
The transformations \eqref{x_trafo_middle}, \eqref{x_trafo_left}
are continuously Fr\'echet differentiable as maps
\begin{equation} \label{trafo_Fderivative}
	(\bar{y}, \xs) \in \pcont \times \R
	\mapsto x^{(\bar{y}, \xs),j}	
	\in C^1(\dom_j) \quad \forall j \in [n]_0.
\end{equation}	
\end{lemma}
\begin{proof}
	The first property follows from the smoothness of $\bar{y}$ and the definitions
	\eqref{xil_curve}-\eqref{xij_curve_shock} of the curves $\xi^{(\bar{y}, \xs)}$.
	Moreover, differentiating \eqref{x_trafo_middle} w.r.t.\ $\x$ yields
	\begin{equation} \label{x_trafo_middle_dx}
		x_{\x}^{(\bar{y}, \xs),j}(t, \x) = \frac{\xi_{j+1}^{(\bar{y}, \xs)}(t)
			- \xi_{j}^{(\bar{y}, \xs)}(t)}{(s_{j+1} - s_j) t}
		\qquad \forall j \in [n-1], \ (t, \x) \in \dom_j, \ t > 0,
	\end{equation}
	hence \eqref{spacetrafo_xderivative_bound}
	follows for $x^{(\bar{y}, \xs),j}$ from \eqref{nondegcondition_shockcurves} and,
	analogously for $j \in \{0, n\}$, from \eqref{nondegcondition_boundarycurves}.
	For the inverse transformation, differentiating
	$x^{(\bar{y}, \xs),j} \big( t, \bar{x}^{(\bar{y}, \xs),j}(t,x) \big) = x$
	w.r.t.\ $x$ gives $\partial_x \bar{x}^{(\bar{y}, \xs),j}(t,x)
	= \big( \partial_{\x} x^{(\bar{y}, \xs),j}(t,\x) \big)^{-1}$
	with $\x = \bar{x}^{(\bar{y}, \xs),j}(t,x)$.
	
	To prove \eqref{trafo_Fderivative}, it follows from
	\eqref{xil_curve}, \eqref{xir_curve}, \eqref{xij_curve_shock}
	the continuous Fr\'echet differentiability of the maps
	\begin{equation} \label{xij_Fdiff}
		(\bar{y}, \xs) \in \pcont \times \R \mapsto  
		\xi_{(\ell)}^{(\bar{y}, \xs)}, 
		\xi_{(r)}^{(\bar{y}, \xs)}, 
		\xi_j^{(\bar{y}, \xs)} \in C^1([0,T])
	\end{equation}
	with the Fr\'echet derivative for all 
	$(\delta \bar{y}, \delta \xs) \in \pcont \times \R$ and $j \in [n]$ given by
	\begin{equation} \label{Fdiff_xij}
		\big( d_{(\bar{y}, \xs)} \xi_j^{(\bar{y}, \xs)} \cdot 
		(\delta \bar{y}, \delta \xs) \big) (t) =
		\int_0^{t} 
		\nabla \lambda_j ( \bar{y}^{j-1} , \bar{y}^{j}  )
		\cdot ( \delta \bar{y}^{j-1} , \delta \bar{y}^{j} ) \ \intd \tau
		+ \delta \xs, 
	\end{equation}
	where $\bar{y}^{j-1}, \, \bar{y}^{j}, \, \delta \bar{y}^{j-1}, \, \delta \bar{y}^{j}$
	are all evaluated in $(\tau, s_j \tau)$.
	The argument is analogous for $\xi_{(\ell)}^{(\bar{y}, \xs)}$ 
	and $\xi_{(r)}^{(\bar{y}, \xs)}$.
	We now prove \eqref{trafo_Fderivative} for $j \in [n-1]$.
	First, observe that $x^{(\bar{y}, \xs),j}$ depends linearly on
	$\xi_j^{(\bar{y}, \xs)}, \xi_{j+1}^{(\bar{y}, \xs)}$ since 
	$x^{(\bar{y}, \xs),j} = P_j ( \xi_j^{(\bar{y}, \xs)}, \xi_{j+1}^{(\bar{y}, \xs)} )$
	with the linear operator $P_j$ defined on the closed subspace
	\begin{equation} \label{def_Xj}
		\mathcal{X}_j \coloneqq \{ \zeta = (\zeta_j, \zeta_{j+1}) 
		\in C^1([0,T])^2 : \zeta_j(0) = \zeta_{j+1}(0) \}
	\end{equation}
	with $\Delta s_j \coloneqq s_{j+1} - s_j > 0$
	and $\theta(t,\x) \coloneqq (\x - s_j t)/(\Delta s_j t)$ for all $(t, \x) \in \dom_j$ by
	\begin{equation*}
		P_j \zeta \coloneqq \zeta_j + \theta (\zeta_{j+1} - \zeta_j).
	\end{equation*}
	We show $P_j \in \mathcal{L}(\mathcal{X}_j, C^1(\dom_j))$.
	Since $\theta \in [0,1]$ on $\dom_j$,
	$\| P_j \zeta \|_{C^0(\dom_j)} \leq \| \zeta \|_{C^0([0,T])^2}$ follows.
	Moreover, $\eta \coloneqq \zeta_{j+1} - \zeta_j$ satisfies $\eta(0) = 0$
	by \eqref{def_Xj}, 
	so $| \eta(t)/t | \leq \| \dot \eta \|_{C^0([0,T])}$. Since
	$\partial_{\x} (P_j \zeta)(t,\x) = \eta(t) (\Delta s_j t)^{-1}$
	for all $(t, \x) \in \dom_j$ with $t > 0$,		
	this implies $\| \partial_{\x} (P_j \zeta) \|_{C^0(\dom_j)}
	\leq \| \dot \eta \|_{C^0([0,T])} / \Delta s_j$.
	Further, $\partial_t \theta = - \x/(\Delta s_j t^2)$
	and $\x/t \in [s_j, s_{j+1}]$ on $\dom_j$ give
	$\partial_t (P_j \zeta) = \dot \zeta_j + \big( \frac{\x}{t}
	( \dot \eta - \frac{\eta}{t} ) - s_j \dot \eta \big) / \Delta s_j$.
	With $| \dot \eta(t) - \eta(t)/t | \leq 2 \| \dot \eta \|_{C^0([0,T])}$, this yields
	\begin{equation*}
		\| \partial_t (P_j \zeta) \|_{C^0(\dom_j)}
		\leq C_{P_j} \| \zeta \|_{C^1([0,T])^2}
		\quad \text{with } C_{P_j} = C_{P_j}(s_j, s_{j+1}) > 0 .
	\end{equation*}
	Since $\eta(t)/t \to \dot \eta(0)$ and $\dot \eta(t) - \eta(t)/t \to 0$ for $t \searrow 0$,
	both $\partial_{\x} (P_j \zeta)$, $\partial_t (P_j \zeta)$
	extend continuously to $(0,0)$.
	This shows $P_j \in \mathcal{L}(\mathcal{X}_j, C^1(\dom_j))$ and
	$\| P_j \|_{\mathcal{L}(\mathcal{X}_j, C^1(\dom_j))} \leq C_{P_j}$
	by possibly increasing $C_{P_j}$.	
	By \eqref{xij_Fdiff}, the map $\Xi_j : (\bar{y}, \xs) \mapsto
	( \xi_j^{(\bar{y}, \xs)}, \xi_{j+1}^{(\bar{y}, \xs)} ) \in C^1([0,T])^2$
	is continuously Fr\'echet differentiable and takes values in $\mathcal{X}_j$.
	The same applies for its derivative $d \Xi_j \cdot (\delta \bar{y}, \delta \xs)
	= (\delta \xi_j, \delta \xi_{j+1})$, since $\xi_j^{(\bar{y}, \xs)}(0)
	= \xi_{j+1}^{(\bar{y}, \xs)}(0) = \xs$ and $\delta \xi_j(0) = \delta \xi_{j+1}(0)
	= \delta \xs$ by \eqref{Fdiff_xij}.
	Hence, $\Xi_j$ is continuously Fr\'echet differentiable as a map to
	$\mathcal{X}_j$ and $x^{(\bar{y}, \xs),j} = P_j \circ \Xi_j$.
	The chain rule proves \eqref{trafo_Fderivative}, since
	\begin{equation*} 
		d_{(\bar{y}, \xs)} x^{(\bar{y}, \xs),j} \cdot (\delta \bar{y}, \delta \xs)
		= P_j \big( d \Xi_j \cdot (\delta \bar{y}, \delta \xs) \big)
		= P_j ( \delta \xi_j, \delta \xi_{j+1} ).
	\end{equation*}
	Since $\| \delta \xi_j \|_{C^1([0,T])}, \| \delta \xi_{j+1} \|_{C^1([0,T])}
	\leq (T+1) \| \nabla \lambda \| \| \delta \y \|_{\pcont} + |\delta \xs|$
	by \eqref{Fdiff_xij}, the boundedness of $P_j$ implies
	with $C = C_{P_j} \max \{ (T+1) \| \nabla \lambda \| , 1 \}$ that 
	\begin{equation} \label{x_trafo_middle_Fdiff_normbound}
		\big\| d_{(\bar{y}, \xs)} x^{(\bar{y}, \xs),j} \cdot (\delta \bar{y}, \delta \xs)
		\big\|_{C^1(\dom_j)} \leq C \big( \| \delta \y \|_{\pcont} + |\delta \xs| \big).
	\end{equation}		
	The cases $j \in \{0, n \}$ are straightforward, since 
	\eqref{smallness_ass_welldef_spacetrafo}
	implies that the denominators in \eqref{x_trafo_left} are
	bounded away from zero independently of $t$.
\end{proof}

An analogous result to \cref{lema:spacetrafo_properties} also holds
for the inverse transformations. 
\begin{lemma} \label{lema:spacetrafo_inverse_properties}
	Let $\bar{y} \in \pdiff$ and $\xs \in (- \varepsilon, \varepsilon)$
	satisfy the assumptions	of \cref{lem:welldefinedness_spacetrafo}.	
	With the domain $\D$ from \eqref{domain_physical},
	the inverse transformations $\x^{(\bar{y}, \xs),j}$ from \eqref{inv_trafo_xbar}	
	are continuously Fr\'echet differentiable as maps
	\begin{equation} \label{trafo_inverse_Fderivative}
		(\bar{y}, \xs) \in \pcont \times \R
		\mapsto \x^{(\bar{y}, \xs),j}
		\in C^0( \D )
		\quad \forall j \in [n]_0.
	\end{equation}	
	Moreover, there is $L_{\mathrm{trafo}} > 0$, independent of $\bar{y}$ and $\xs$,
	such that for all $\bar{y}^1, \bar{y}^2, \xs^1, \xs^2$ satisfying these assumptions,
	it holds that
	\begin{equation} \label{xbar_trafo_lipschitz}
		\| \x^{(\bar{y}^1, \xs^1),j} - \x^{(\bar{y}^2, \xs^2),j} \|_{C^0(\D)}
		\leq L_{\mathrm{trafo}} \big( \| \bar{y}^1 - \bar{y}^2 \|_{\pcont}
		+ |\xs^1 - \xs^2| \big)
		\quad \forall j \in [n]_0 .
	\end{equation}	
\end{lemma}
\begin{proof}
	We only prove the case $j \in [n-1]$, since $j \in \{ 0, n \}$ is similar.
	First note that \eqref{inv_trafo_xbar} is in fact
	well-defined for $t>0$ and extends continuously to $t=0$, since we have with
	$| x - \xi_{j}^{(\bar{y}, \xs)}(t) | \leq 2 \ell$ and
	\eqref{nondegcondition_shockcurves} that 
	$| \x^{(\bar{y}, \xs),j}(t,x) | \leq  2 \ell C$
	for some constant $C = C(s_j, s_{j+1}, \lambdamax) > 0 $.
	With this, the continuous Fr\'echet differentiability of
	\eqref{trafo_inverse_Fderivative} follows from 
	the continuous Fr\'echet differentiability of \eqref{xij_Fdiff}.
	
	The estimate \eqref{xbar_trafo_lipschitz} follows from \eqref{inv_trafo_xbar},
	the Lipschitz continuity of $\xi_j^{(\bar{y}, \xs)}$ w.r.t.\ $(\bar{y}, \xs)$
	by \eqref{xil_curve}--\eqref{xij_curve_shock}, and the lower bound
	\eqref{nondegcondition_shockcurves} on the denominator.
\end{proof}

The semilinear counterpart of \eqref{GRP} can equivalently be formulated in the reference space
by using the above space transformations.

\begin{lemma} \label{lem:lemma_rescaledsolution_Dj}
Let $\bar{z} \in \pdiff$ with values in $\Yval$ and $\xs \in \R$. 
For $j \in [n]_0$ define 
\begin{align}
\bar{A}^{(\bar{z}, \xs),j}(t,\bar{x}) 
&= \big( A(\bar{z}(t, \bar{x})) - x_t^{(\bar{z}, \xs),j}(t,\bar{x}) I \big) 
\big( x_{\x}^{(\bar{z}, \xs),j}(t,\bar{x}) \big)^{-1}, \label{A_bar_z} \\ 
\bar{g}^{(\bar{z}, \xs),j}(t,\bar{x}, \bar{y}) 
&= g(t, x^{(\bar{z}, \xs),j}(t,\bar{x}), \bar{y}) \label{h_bar_z},
\end{align}
on $\dom_j$.
If $\bar{y} \in C^1(\bar{D}_j)$ is a classical solution of the semilinear problem
\begin{equation} \label{problem_ybar_sl_component} 
\bar{y}_t + \bar{A}^{(\bar{z}, \xs),j} \bar{y}_{\x} 
= \bar{g}^{(\bar{z}, \xs),j}(\cdot, \bar{y}) 
\quad \text{on } \dom_j,
\end{equation}
then the transformed state $y(t, x) = \bar{y} (t, \bar{x}^{(\bar{z}, \xs),j}(t,x))$
is a classical solution $y \in C^1(D^{(\bar{z}, \xs)}_j; \R^n)$ of the semilinear problem 
\begin{equation*} 
y_t + A(z) y_x = g(\cdot,y) \quad \text{on } D^{(\bar{z}, \xs)}_j,
\end{equation*}
where $z(t, x) = \bar{z} (t, \bar{x}^{(\bar{z}, \xs),j}(t,x))$.
The converse is also true.
If the smallness assumptions \eqref{smallness_ass_welldef_spacetrafo}
are satisfied, the matrices $\bar{A}^{(\bar{z}, \xs),j}$ are also strictly hyperbolic.
\end{lemma}
\begin{proof}
Differentiating $x^{(\bar{z}, \xs),j} \big(t, \bar{x}^{(\bar{z}, \xs),j}(t,x) \big) = x$ 
w.r.t.\ $t$ and $x$ implies
\begin{equation} \label{spacetrafo_identities}
	\bar{x}_x^{(\bar{z}, \xs),j}(t,x) 
	= \big( x^{(\bar{z}, \xs),j}_{\x} (t, \x ) \big)^{-1}, \quad
	\bar{x}_t^{(\bar{z}, \xs),j}(t,x) = - x^{(\bar{z}, \xs),j}_{t} (t, \x )
	\big( x^{(\bar{z}, \xs),j}_{\x} (t, \x ) \big)^{-1}
\end{equation} 
evaluated in $\x = \bar{x}^{(\bar{z}, \xs),j}(t,x)$. 
Furthermore, $y(t, x) = \bar{y} (t, \bar{x}^{(\bar{z}, \xs),j}(t,x))$
satisfies $y_t = \bar{y}_t + \bar{y}_{\bar{x}} \bar{x}^{(\bar{z}, \xs),j}_t$
and $y_x = \bar{y}_{\bar{x}} \bar{x}^{(\bar{z}, \xs),j}_x$.
Inserting \eqref{spacetrafo_identities} proves \eqref{A_bar_z}. 
The strict hyperbolicity of $\bar{A}^{(\bar{z}, \xs),j}$ 
follows from the lower bound $\partial_{\bar{x}} x^{(\bar{z}, \xs),j} \geq 1/2$, 
see \cref{lema:spacetrafo_properties}.
The converse follows in the same way, since $x^{(\bar{z}, \xs),j}$ is a
bijection.
\end{proof}

\begin{remark} \label{rem:notation_dropj}
The above lemma justifies to equivalently solve the transformed problems
\eqref{problem_ybar_sl_component} in the reference space. 
For convenience, we use the notation 	
\begin{equation*} 
\bar{y}_t + \bar{A}^{(\bar{z}, \xs)} \bar{y}_{\x} 
= \bar{g}^{(\bar{z}, \xs)}(\cdot, \bar{y}) 
\quad \text{on } \dom
\end{equation*}
if all piecewise smooth parts $\bar{y}^j \in C^1(\bar{D}_j)$
of $\bar{y} \in \pdiff$ satisfy \eqref{problem_ybar_sl_component} for all $j \in [n]_0$.
	
Also, we drop indices $j$ from quantities related to the different sectors $\dom_j$
as abuse of notation, since it will implicitly be clear which $\dom_j$ is meant.
\end{remark}

\begin{remark}
	The \cref{lem:lemma_rescaledsolution_Dj} shows an (unsurprising) 
	drawback of the reference space. Scaling out the 
	unknown shock locations comes at the cost of the 
	underlying problem becoming non-local,
	since $\bar{A}^{(\bar{z}, \xs)}(t,\bar{x})$ and $\bar{g}^{(\bar{z}, \xs)}(t,\bar{x}, \bar{y}) $
	do not only depend on $\bar{z}(t,\bar{x})$, but on 
	values of $\bar{z}$ up until time $t$.    
\end{remark}

\begin{remark}
	It follows from \eqref{A_bar_z} that the eigenvalues
	$\bar{\lambda}^{(\bar{z}, \xs)}_i(t,\bar{x})$ of 
	the matrix $\bar{A}^{(\bar{z}, \xs)}(t,\bar{x})$ from \eqref{A_bar_z} are
	given for all $i \in [n]$ by
	\begin{equation} \label{lambda_bar}
		\bar{\lambda}^{(\bar{z}, \xs)}_i(t,\bar{x}) = 
		\big( \lambda_i(\bar{z}(t, \bar{x})) - x_t^{(\bar{z}, \xs)}(t,\bar{x}) \big) 
		\big( x_{\x}^{(\bar{z}, \xs)}(t,\bar{x}) \big)^{-1} .
	\end{equation}	
    Using \cref{lema:spacetrafo_properties}, it is easy to see that the 
    minimal-angle property also holds in the reference space,
    since \eqref{lambda_bar} can be rewritten as
    \begin{equation} \label{lambda_bar_rewritten}
	    \bar{\lambda}^{(\bar{z}, \xs)}_i(t,\bar{x}) = \bar{x}_t^{(\bar{z}, \xs)}(t,x) 
	    + \bar{x}_x^{(\bar{z}, \xs)}(t,x) \lambda_i(\bar{z}(t,\bar{x}))
	    \quad \text{with } x = x^{(\bar{z}, \xs)}(t,\x),
    \end{equation}
    and thus it holds for any $i \neq j$ 
    and independent of $\bar{z}$, $(t, \x)$ due to \eqref{spacetrafo_xderivative_bound} that
    \begin{equation} \label{minimal_angle_refspace}
	    | \bar{\lambda}^{(\bar{z}, \xs)}_i(t,\bar{x}) 
	    - \bar{\lambda}^{(\bar{z}, \xs)}_j(t,\bar{x}) |
	    \geq \bar{ \eta }_{\min} \coloneqq \eta_{\min}/2 > 0 .
    \end{equation}
    The entropy condition \eqref{entropy_cond_speeds}
    requires that on all shocks $\Sigma_j$ it holds
    for all $i < j < k$
    \begin{equation} \label{entropy_cond_Abar}
	    \bar{\lambda}^{(\bar{z}^{j-1}, \xs)}_j \geq s_j \geq 
	    \bar{\lambda}^{(\bar{z}^{j}, \xs)}_j, 
	    \quad \bar{\lambda}^{(\bar{z}^{j-1}, \xs)}_k, 
	    \bar{\lambda}^{(\bar{z}^{j}, \xs)}_k > s_j >
	    \bar{\lambda}^{(\bar{z}^{j-1}, \xs)}_i, 
	    \bar{\lambda}^{(\bar{z}^{j}, \xs)}_i.
    \end{equation}          
\end{remark}

Analogously to \eqref{char_curve} of 
characteristic curves in physical coordinates,
we define characteristic curves for the transformed
problem in the reference space.
\begin{definition} \label{def:characteristiccurve_refspace}
Let $\bar{z} \in \pdiff$ 
$j \in [n]_0$, and $(t, \x) \in \dom_j$. 
The solution of
\begin{equation*}
\bar{x}_i'(s) = \bar{\lambda}^{(\bar{z}, \xs)}_i(s,\bar{x}_i(s)), 
\quad \bar{x}_i(t) = \bar{x}
\end{equation*}
is the $i$-\emph{characteristic curve}
w.r.t.\ $\bar{A}^{(\bar{z}, \xs)}$
and denoted by $\bar{x}_i(\cdot ; t, \bar{x},\bar{z})$.
\end{definition}

The characteristic curves satisfy the following stability properties.
\begin{lemma} \label{lem:char_stability}
Let $\bar{z} \in \pdiff$ with values in $\Yval$ 
with $\| \bar{z} \|_{PC^1(\dom)} \leq C$ for some constant $C>0$.
Moreover, let $i \in [n], \ j \in [n]_0$, and $(t_1, x_1),(t_2,x_2) \in \dom_j$ be arbitrary.
\begin{enumerate}
\item[(a)] There is $L_{\bar{x}}>0$ only depending on $C$
such that for 
$\bar{x}^{1}(t) = \bar{x}_i(t; t_1,x_1,\bar{z})$
and $\bar{x}^{2}(t) = \bar{x}_i(t; t_2,x_2,\bar{z})$
it holds for all $t \in [0,\max(t_1, t_2)]$ 
for which both $\bar{x}^1$ and $\bar{x}^2$ stay inside of $\bar{D}_j$ that
\begin{equation*} 
| \bar{x}^1(t) - \bar{x}^2(t) | 
\leq L_{\bar{x}} ( |x_1 - x_2| + |t_1 - t_2| ).
\end{equation*}
\item[(b)] Assume that \eqref{minimal_angle_refspace}, \eqref{entropy_cond_Abar} hold for $\bar{z}$.
Let $j \in [n-1]$, $i>j$. The case $i < j$ is analogous. 
Then $\bar{x}^1, \bar{x}^2$ have unique
intersection times $t^j_i(t_1, x_1) , t^j_i(t_2, x_2)$
with $\Sigma_j$ satisfying for a constant $L_{\bar{t}} > 0$ only depending on $C$ that
\begin{equation} \label{tbar_lipschitz}
| t^j_i(t_1, x_1) - t^j_i(t_2, x_2) | 
\leq L_{\bar{t}} ( |x_1 - x_2| + |t_1 - t_2| ).
\end{equation}
\item[(c)] Let $(t_1,x_1) \in \mathrm{int}(\dom_j)$ and
$\bar{x}(t)=\bar{x}_i(t; t_1,x_1,\bar{z})$.
Assume $t \in [0,t_1]$ is such that $(t, \bar{x}(t)) \in \mathrm{int}(\dom_j)$.
Then $\bar{x}(t)$ and the intersection time $t^j_i(t_1,x_1)$ from part (b)
are continuously differentiable w.r.t.\ $(t_1,x_1)$.
\end{enumerate}
\end{lemma}

\begin{proof}
First, we show that the eigenvalues $\bar{\lambda}^{(\bar{z}, \xs),j}$
of $\bar{A}^{(\bar{z}, \xs),j}$ from \eqref{lambda_bar}
are Lipschitz continuous w.r.t.\ $\x$.
We prove the case $j \in [n-1]$, since $j \in \{ 0, n \}$ is analogous.
From \eqref{x_trafo_middle_dx} we see that
$x_{\x}^{(\bar{z}, \xs),j}(t, \x)$
is independent of $\x$. Using that the space transformations are 
twice continuously differentiable by \cref{lema:spacetrafo_properties}, it holds
\begin{equation} \label{trafo_mixed_derivative}
	x_{t, \x}^{(\bar{z}, \xs)}(t,\x)
	= x_{\x, t}^{(\bar{z}, \xs)}(t,\x) 
	= \frac{ ( \dot{\xi}_{j+1}^{(\bar{z}, \xs)}(t) - \dot{\xi}_{j}^{(\bar{z}, \xs)}(t) ) t	
		- ( \xi_{j+1}^{(\bar{z}, \xs)}(t) - \xi_{j}^{(\bar{z}, \xs)}(t) )}
	{ (s_{j+1}-s_j) t^2 } . \!\!
\end{equation}
Since $\xi_{j}^{(\bar{z}, \xs)} \in C^2([0,T])$ with
$ \| \ddot{\xi}_{j}^{(\bar{z}, \xs)} \| _{ C^0([0,T])} \leq \| \nabla \lambda_j \| C (1 + \lambdamax)$, this yields
\begin{equation*}
	\big| \xi_{j}^{(\bar{z}, \xs)}(t) - \dot{\xi}_{j}^{(\bar{z}, \xs)}(t) t - \xs \big|
	\leq t^2 \| \ddot{\xi}_{j}^{(\bar{z}, \xs)} \| _{ C^0([0,T])} /2
	\leq t^2 \| \nabla \lambda_j \| C (1 + \lambdamax) /2 
\end{equation*}
and similarly for $\xi_{j+1}^{(\bar{z}, \xs)}$.
Note that $\xs$ cancels in the difference appearing in the numerator of \eqref{trafo_mixed_derivative}.
With \eqref{trafo_mixed_derivative}, this implies
$ \| \frac{\intd}{\intd \x} x_{t}^{(\bar{z}, \xs),j} \| 
\leq \frac{\| \nabla \lambda \| C (1 + \lambdamax)}{s_{j+1}-s_j}$.
This proves the Lipschitz continuity of $\bar{\lambda}^{(\bar{z}, \xs)}$
w.r.t.\ $\x$ with a Lipschitz constant $L_{\bar{\lambda}}(C)$.	

For (a), assume w.l.o.g. that $t_1 \leq t_2$ and let $t \leq t_2$ with
$(t, \bar{x}^1(t)), (t, \bar{x}^2(t)) \in \bar{D}_j$ be arbitrary.
If $t \in [t_1, t_2]$, it follows with the maximal characteristic speed $\lambdamax$
that $| \bar{x}^1(t) - \bar{x}^2(t) | \leq |x_1 - x_2| + \lambdamax |t_2 - t_1|$.
If $t < t_1$, it holds that
\begin{align*}
| \bar{x}^1(t) - \bar{x}^2(t) | 
& \leq |x_1 - x_2|  +  \lambdamax |t_1 - t_2| 
+  L_{\bar{\lambda}} \int_t^{t_1} |\bar{x}^1(s) - \bar{x}^2(s)| \ds.
\end{align*}
Assertion (a) then follows by Gronwall's inequality.

To prove (b), assume that $t^j_i(t_1, x_1) \leq t^j_i(t_2, x_2)$.
The minimal angle condition \eqref{minimal_angle_refspace} implies that 
$ | \bar{x}^1( t^j_i(t_2, x_2) ) - \bar{x}^2( t^j_i(t_2, x_2) ) |
\geq \bar{\eta}_{\min} | t^j_i(t_1, x_1) - t^j_i(t_2, x_2) | $. 
Using part (a) yields
$ | t^j_i(t_1, x_1) - t^j_i(t_2, x_2) |
\leq \bar{\eta}_{\min}^{-1} L_{\bar{x}} ( |x_1 - x_2|  +  |t_1 - t_2| )$. 

The differentiable dependence of solutions of ODEs w.r.t.\ the initial condition as in (c)
is a standard result, see, e.g., \cite{Hartmann2002}. 
Using \eqref{minimal_angle_refspace} and the Implicit Function Theorem also
proves the asserted differentiability of the intersection times.	
\end{proof}

\subsection{The Semilinear Problem} \label{sec:slproblems}
Throughout this section, we fix $\bar{z} \in \pdiff$ with values in $\Yval$ and
$\xs \in (-\varepsilon, \varepsilon)$ which
satisfy \eqref{minimal_angle_refspace} and the entropy condition \eqref{entropy_cond_Abar}.
We prove the unique existence of a function $\bar {y} \in \pcont$ which satisfies the
integral characteristic equations associated to the semilinear problem
\begin{equation} \label{sl_problem_compwise}
\bar{y}^j_t + \bar{A}^{(\bar{z}, \xs),j} \bar{y}^j_{\x} = \bar{h}^j(\bar{y}) 
\quad \text{on } \dom_j \quad \forall j \in [n]_0
\end{equation}
with sector-dependent source terms 
$\bar{h}^j : C^0(\dom_j; \R^n) \to C^0(\dom_j; \R^n)$
given in functional form. 
The initial state is 
$(u_l,u_r) \in \mathcal{W} \coloneqq C^0([-\el,0]; \R^n) \times C^0([0, \el]; \R^n)$.
The piecewise continuous parts of $\bar {y}$ are coupled by interior 
boundary conditions imposed on the transformed shocks with speeds $s_j$.
For this, let 
\begin{equation*} 
	F_i^j : [0,T] \times \R \times \R^n \times \R^n \to \R 
	\qquad \forall i, j \in [n], \ i \neq j
\end{equation*}
be such that it holds with the bases from \cref{def:fixed_shock_basis} and components
\begin{equation*}
	\bar{y}^{k}_{*,i} \coloneqq ( l^{*,j}_i )^\top \bar{y}^{k} , \qquad k \in \{ j-1, j \} ,
\end{equation*}
where $j \in [n]$ refers to the discontinuity $\Sigma_j$
from \eqref{Sigma_i} under consideration, that
\begin{equation} \label{F_ij_coupling}
	\begin{aligned}
		\bar{y}^{j}_{*,i} (t, s_j t) &= F_i^j \big( t, \bar{y}^{j-1}_{*,i} (t, s_j t),
		\bar{y}^{j-1} (t, s_j t), \bar{y}^{j} (t, s_j t)  \big) && i>j, \\
		\bar{y}^{j-1}_{*,i} (t, s_j t) &= F_i^j \big( t, \bar{y}^{j}_{*,i} (t, s_j t),
		\bar{y}^{j-1} (t, s_j t), \bar{y}^{j} (t, s_j t) \big) && i<j,
	\end{aligned}
\end{equation}
for all $t \in [0,T]$. Note that \eqref{F_ij_coupling} prescribes
the initial value transported along the $i$-characteristic curve 
starting from the $j$-shock curve. This explains the case distinction 
between faster characteristics $i>j$ and slower ones for $i<j$.

\begin{remark}
	It is necessary to consider source terms in functional form
	in order to handle the transformation \eqref{h_bar_z} of the space variable.
	Note that \eqref{sl_problem_compwise} also covers the local case
	$\bar{h}^j(\bar{y})(t,\x) = \bar{g}^j(t, \x, \y(t,\x))$.
	Also note that the interior 
	boundary conditions \eqref{F_ij_coupling} are not explicit,
	but depend on the full left and right states $\bar{y}^{j-1}, \bar{y}^{j}$.
	This is necessary since the jump condition couples the characteristic families.
\end{remark}

Instead of \eqref{sl_problem_compwise}, we simply use the notation
\begin{equation} \label{problem_ref_sl} 
\bar{y}_t + \bar{A}^{(\bar{z}, \xs)} \bar{y}_{\x} = \bar{h}(\bar{y}) 
\quad \text{on } \dom, \quad \bar{y}^0(0, \cdot) = u_l, 
\quad \bar{y}^n(0, \cdot) = u_r.
\end{equation}
Clearly, $\y$ cannot be expected to be a piecewise classical solution.
Instead, we only demand the associated characteristic equations to be satisfied.
This is known as \emph{broad solutions} \cite{Bressan2005}.
Before formally defining broad solutions, we introduce the notation 
\begin{equation} \label{hi_bar_broadsol}
\bar{h}_i(\bar{y}) = l_i(\bar{z})^\top \bar{h}(\bar{y}) 
+ \big( \partial_t l_i(\bar{z}) 
+ \bar{\lambda}^{( \bar{z}, \xs)} _i \partial_{\bar{x}} l_i(\bar{z}) \big)^\top \bar{y} .
\end{equation}

\begin{definition} \label{def:broadsol_sl}
A function $\bar{y} \in \pcont$ is called a broad solution of the semilinear
system \eqref{problem_ref_sl} equipped with the
interior boundary conditions \eqref{F_ij_coupling} if 
the following integral conditions are satisfied.
We use the abbreviation $\bar{x}_i(s) = \bar{x}_i(s ; t, \bar{x},\bar{z})$
for the characteristic curves. For any $(t,\x) \in \dom_0$, 
it should hold for all $i \in [n]$ that 
\begin{equation} \label{semilinear_characteristiceq_0n}
\bar{y}_i^0 (t, \bar{x}) = u_{l,i}(\bar{x}_i(0) ) 
+ \int_0^t \bar{h}_i(\bar{y}) \left(s,\bar{x}_i(s) \right) \ds,
\end{equation}
where $u_{l,i}(x) = l_i(\bar{z}^0(0,x)) ^\top u_l(x)$.
This is analogous on $\dom_n$.

If $j \in [n-1]$ and $i > j$, let $t^{j}_i = t^{j}_i(t, \bar{x}, \bar{z})$ denote the time 
of intersection of the $i$-characteristic $\bar{x}_i(\cdot ; t, \bar{x},\bar{z})$ 
with $\Sigma_j$. Then it should hold for any $(t, \x) \in \dom_j$ with the abbreviation
$F_i^j(\bar{y}) (t) \coloneqq F_i^j \big( t, \bar{y}^{j-1}_{*,i} (t, s_j t),
\bar{y}^{j-1} (t, s_j t), \bar{y}^{j} (t, s_j t) \big)$ for $i>j$, and with
$\bar{y}^{j-1}_{*,i}$ replaced by $\bar{y}^{j}_{*,i}$ in the second argument
for $i<j$, that
\begin{equation} \label{semilinear_characteristiceq_j}
	\bar{y}_i^j (t, \bar{x}) = F_i^j(\bar{y}) ( t^{j}_i ) 
	+ \big( l_i(\bar{z}^j) - l^{*,j}_i \big)^\top \bar{y}^{j} \, ( t^{j}_i , s_j t^{j}_i )
	+ \int_{t_i^{j}}^t \bar{h}_i(\bar{y}) \left(s,\bar{x}_i(s) \right) \ds.
\end{equation}	
If $i \leq j$, replace the first two terms above by
\begin{equation*}
	F_i^{j+1}(\bar{y}) ( t^{j+1}_i )
	+ ( l_i(\bar{z}^j) - l^{*,j+1}_i )^\top \bar{y}^{j} \, ( t^{j+1}_i , s_{j+1} t^{j+1}_i ) ,
\end{equation*}
with the time of intersection $t^{j+1}_i$ of $\bar{x}_i(\cdot ; t, \bar{x},\bar{z})$
with $\Sigma_{j+1}$.
\end{definition}

Next, we prove the unique existence of a broad solution of 
\eqref{problem_ref_sl} under suitable Lipschitz assumptions.
For this, let $\dom^t = \dom \cap ([0,t] \times \R)$
and $\dom_j^t = \dom_j \cap ([0,t] \times \R)$, and define
$PC^0(\dom^t)$ and $\| \cdot \|_{PC^0(\dom^t)}$ analogously to \eqref{def_X}
with $\dom_j$ replaced by $\dom_j^t$.

\begin{theorem} \label{thm:existence_broadsol_sl}
Assume that the source term $\bar{h} : \pcont \to \pcont$
admits a Lipschitz constant $L_{\bar{h}} > 0$ such that
for all $\bar{v}, \bar{w} \in \pcont$ it holds that
\begin{equation} \label{lipschitz_sl_sourceterm}
\|  \bar{h}(\bar{v}) - \bar{h}(\bar{w}) \| _{PC^0(\dom^t)}
\leq L_{\bar{h}} \| \bar{v} - \bar{w} \| _{PC^0(\dom^t)} \quad \forall t \in [0,T].
\end{equation}
Suppose that the interior boundary conditions \eqref{F_ij_coupling}
are continuous w.r.t.\ time and
Lipschitz continuous w.r.t.\ the state, i.e., there are $L_{F_1}, L_{F_2} > 0$
such that for all $v, w \in \R$, $v^1, v^2, w^1, w^2 \in \R^n$,
$t \in [0,T]$, and all $i \neq j \in [n]$ it holds that
\begin{equation} \label{lipschitz_Fij}
	| F_i^j (t, v, v^1, v^2) -  F_i^j (t, w, w^1, w^2) |
	\leq L_{F_1} | v - w | + L_{F_2} ( |v^1 - w^1| + |v^2 - w^2| ) .
\end{equation}
Moreover, let $\theta > 0$ be such that the bases from \cref{def:fixed_shock_basis} satisfy
\begin{equation} \label{def_theta_basisdeviation}
	\big\| l_i(\bar{z}^k) - l^{*,j}_i \big\|_{C^0(\Sigma_j)} \leq \theta
	\qquad \forall i \neq j \in [n], \ k \in \{ j-1, j \} ,
\end{equation}
and assume that the Lipschitz constants $L_{F_1}, L_{F_2}$ of $F_i^j$ 
and $\theta$ from \eqref{def_theta_basisdeviation} satisfy
\begin{equation} \label{smallness_LF2}
	L_{\mathcal{F}_2} \coloneqq L_{F_2} + ( 1 + L_{F_1} ) \, \theta
	\leq \frac{1}{4n} \Big( \sum_{k=0}^{n} (L_{F_1}) ^k \Big)^{-1} .
\end{equation}
Then for all $(u_l,u_r) \in \mathcal{W}$
there exists a unique broad solution $\bar{y} \in \pcont$ 
of \eqref{problem_ref_sl} satisfying \eqref{F_ij_coupling}.
\end{theorem}

\begin{proof}
We define an operator $\mathcal{T} : \pcont \to \pcont$
whose fixed points are broad solutions.
We construct $\bar{y} = \mathcal{T} (\bar{v})$ component-wise by
$ \bar{y}^j  = \sum_{i=1}^n \bar{y}^j_i  r_i(\bar{z}^j)$
with $\bar{y}^j_i = l_i(\bar{z}^j)^\top \bar{y}^j$.
W.l.o.g. we only define $\bar{y}^j_i$ for $j \in [n-1]_0$ and $i > j$,
the cases $i \leq j$ are analogous.
Abbreviating $d_i^{j,k} \coloneqq l_i(\bar{z}^{k}) - l^{*,j}_i$ for $k \in \{ j-1, j \}$,
we introduce the basis converted interior boundary conditions
\begin{equation} \label{def_Fcal_ij}
	\mathcal{F}_i^j \big( t, w, v^1, v^2 \big) \coloneqq
	F_i^j \big( t, w - ( d_i^{j,j-1} )^\top v^1 , v^1, v^2 \big)
	+ ( d_i^{j,j} )^\top v^2 ,
\end{equation}
where $d_i^{j,k} = d_i^{j,k}(t)$ is evaluated at $(t, s_j t)$.

For $(t, \bar{x}) \in \dom_j$, we set with the intersection times $t_i^j$ from \eqref{tbar_lipschitz}
\begin{align} 
	\bar{y}_i^0 (t, \bar{x}) & = u_{l,i}(\bar{x}_i(0) ) 
	+ \int_0^t \bar{h}_i(\bar{v}) \left(s,\bar{x}_i(s) \right) \ds, \label{def_operator_case1} \\
	\bar{y}_i^j (t, \bar{x}) & = \mathcal{F}_i^j \big( t_i^{j}, \bar{y}^{j-1}_i,
	\bar{v}^{j-1}, \bar{v}^{j} \big)  
	+ \int_{t_i^{j}}^t \bar{h}_i(\bar{v}) \left(s,\bar{x}_i(s) \right) \ds , \label{def_operator_case3}
\end{align}
for all $i \geq j+1$, 
where the arguments $\bar{y}^{j-1}_i, \bar{v}^{j-1}, \bar{v}^{j}$
of $\mathcal{F}_i^j$ are evaluated at $(t_i^{j}, s_j t_i^{j})$.
Note that \eqref{def_operator_case3} is in fact well-defined, since it can be evaluated
in an order for which the input $\bar{y}^{j-1}_i$ into $\mathcal{F}_i^j$ was already computed before.
At a fixed point of $\mathcal{T}$, the two terms involving $d_i^{j,k}$ in \eqref{def_Fcal_ij}
turn $\bar{y}^{j-1}_i$ into $\bar{y}^{j-1}_{*,i}$ and the value prescribed by
\eqref{F_ij_coupling} into $\bar{y}^{j}_i$, so that \eqref{semilinear_characteristiceq_j} holds.
Moreover, \eqref{lipschitz_Fij} and \eqref{def_theta_basisdeviation} imply with
$L_{\mathcal{F}_2}$ from \eqref{smallness_LF2}
for all $w, \tilde{w} \in \R$, $v^1, v^2, \tilde{v}^1, \tilde{v}^2 \in \R^n$,
$t \in [0,T]$ and all $i \neq j \in [n]$ that
\begin{equation} \label{lipschitz_Fcal_ij}
	\big| \mathcal{F}_i^j ( t, w, v^1, v^2 ) - \mathcal{F}_i^j ( t, \tilde{w}, \tilde{v}^1, \tilde{v}^2 ) \big|
	\leq L_{F_1} | w - \tilde{w} | 
	+ L_{\mathcal{F}_2}	\big( | v^1 - \tilde{v}^1 | + | v^2 - \tilde{v}^2 | \big) .
\end{equation}
Next, observe that the Lipschitz continuity \eqref{lipschitz_sl_sourceterm} of
$\bar h$ also transfers to $\bar{h}_i$ from \eqref{hi_bar_broadsol},
since it holds with
$C_l = \| l \| L_{\bar{h}} + \max_{i \in [n]} 
\| \partial_t l_i(\bar{z}) + \bar{\lambda}^{(\bar{z}, \xs)} _i \partial_{\bar{x}} l_i(\bar{z}) \|_{\pcont}$ that 
\begin{equation*} 
\|  \bar{h}_i(\bar{v}) - \bar{h}_i(\bar{w}) \| _{PC^0(\dom^t)}
\leq C_l \| \bar{v} - \bar{w} \| _{PC^0(\dom^t)}
\qquad \forall \bar{v}, \bar{w} \in \pcont, \ \forall t \in [0,T].
\end{equation*}
Note that $C_l$ depends on $\| \bar{z} \| _{PC^1(\dom)}$.	
We will now show that $\mathcal{T}$ has a unique fixed point in $\pcont$
using Banach's fixed-point theorem. 
We use the norm
\begin{equation}
\| \bar{y} \|_* = \max_{j \in [n]_0} \max_{(t,\bar{x}) \in \dom_j} 
e^{-Ct} \sum_{i=1}^n | l_i(\bar{z}^j(t, \x))^ \top \bar{y}^j(t,\x) |  ,
\quad C = 4 n C_l \sum_{k=0}^{n} (L_{F_1}) ^k 
\label{rescalednorm1}
\end{equation}
on $\pcont$.
Clearly, $(\pcont, \| \cdot \|_*)$ is a Banach space. 

Let $\bar{v}^{[1]},\bar{v}^{[2]} \in \pcont$ be arbitrary and let 
$\bar{y}^{[k]} = \mathcal{T}(\bar{v}^{[k]})$ for $k \in \{1, 2 \}$. 
For $j=0$, arbitrary $i \in [n]$ and $(t,\x) \in \dom_0$,
it follows by \eqref{def_operator_case1} that
\begin{equation} \label{x11}
\begin{aligned}
e^{-Ct} & \left| \bar{y}^{[2],0}_i (t,\x) - \bar{y}^{[1],0}_i (t,\x) \right| 
\leq e^{-Ct} \int_0^t C_l \| \bar{v}^{[2]} - \bar{v}^{[1]} \| _{PC^0(\dom^s)} \ds  \\ 
& \leq C_l \|\bar{v}^{[2]} - \bar{v}^{[1]} \|_* \int_0^t e^{-C(t-s)} \ds 
\leq \frac{C_l}{C} \|\bar{v}^{[2]} - \bar{v}^{[1]} \|_*.
\end{aligned}
\end{equation}
Here we used that $| \bar{f}^j | \leq \sum_{i=1}^n | l_i(\bar{z}^j)^\top \bar{f}^j |$
by $\bar{f}^j = \sum_{i} \bar{f}^j_i r_i(\bar{z}^j)$ and the normalization
$|r_i| \equiv 1$, so that $e^{-C \tau} | \bar{f}^j (\tau, x) | \leq \| \bar{f} \|_*$
and $\| \bar{f} \|_{PC^0(\dom^\tau)} \leq e^{C \tau} \| \bar{f} \|_*$ for $\bar{f} \in \pcont$.

If $j \in [n-1]$, $i>j$ and $(t, \x) \in \dom_j$, it holds by \eqref{lipschitz_Fcal_ij}
with all arguments of $\mathcal{F}_i^j$ evaluated at $(t_i^{j}, s_j t_i^{j})$ that
\begin{equation*}
	\begin{aligned}
		& e^{-Ct} \big| \mathcal{F}_i^j \big( \bar{y}^{[2],j-1}_i, \bar{v}^{[2], j-1}, \bar{v}^{[2], j} \big) 
		- \mathcal{F}_i^j \big( \bar{y}^{[1],j-1}_i, \bar{v}^{[1], j-1}, \bar{v}^{[1], j} \big) \big| \\
		\leq & \, e^{-Ct} \big( L_{F_1} | \bar{y}^{[2],j-1}_i - \bar{y}^{[1],j-1}_i |  
		+ L_{\mathcal{F}_2}	( | \bar{v}^{[2], j-1} - \bar{v}^{[1], j-1} | 
		+ | \bar{v}^{[2], j} - \bar{v}^{[1], j} | ) \big) \\
		\leq & \, e^{-C t_i^{j}} L_{F_1} \big| \bar{y}^{[2],j-1}_i - \bar{y}^{[1],j-1}_i \big| 
		(t_i^{j}, s_j t_i^{j}) 
		+ 2 L_{\mathcal{F}_2} \|\bar{v}^{[2]} - \bar{v}^{[1]} \|_* .
	\end{aligned}
\end{equation*}
With \eqref{def_operator_case3}, it then follows that
\begin{equation} \label{x12}
	\begin{aligned}
		& e^{-Ct} \big| \bar{y}^{[2],j}_i(t,\x) - \bar{y}^{[1],j}_i(t,\x) \big| \\
		\leq & e^{-C t_i^{j}} L_{F_1} \big| \bar{y}^{[2],j-1}_i  
		- \bar{y}^{[1],j-1}_i \big| (t_i^{j}, s_j t_i^{j})
		+ ( C_l/C +  2 L_{\mathcal{F}_2} ) 
		\|\bar{v}^{[2]} - \bar{v}^{[1]} \|_* .
	\end{aligned}
\end{equation}
Repeatedly applying \eqref{x12} and then using \eqref{x11} yields for 
any $j \in [n-1]$, $i>j$ that
\begin{equation*}
	e^{-Ct} \Big|  \bar{y}^{[2],j}_i(t,\x) - \bar{y}^{[1],j}_i(t,\x) \Big| 
	\leq \Big( \Big( \frac{C_l}{C} + 2 L_{\mathcal{F}_2} \Big) \sum_{k=0}^{n} (L_{F_1}) ^k \Big)
	\|\bar{v}^{[2]} - \bar{v}^{[1]} \|_* .
\end{equation*}
This holds analogously for the case $j \in [n]$ and $i \leq j$.
Summing over $i \in [n]$ and inserting the assumed
bound \eqref{smallness_LF2} for $L_{\mathcal{F}_2}$
and the constant $C$ from \eqref{rescalednorm1} results in
\begin{equation*}
	\|\bar{y}^{[2]} - \bar{y}^{[1]} \|_* \leq  \tfrac34 \|\bar{v}^{[2]} - \bar{v}^{[1]} \|_*  .
\end{equation*}
To apply Banach's fixed-point theorem, it remains to prove that
$\bar{y} = \mathcal{T}(\bar{v}) \in \pcont$ for $\bar{v} \in \pcont$.
This is the case if $\bar{y}_i^j \in C^0(\dom_j)$ for all $i,j$. 
To verify this, first observe that $\bar{h}_i$ from \eqref{hi_bar_broadsol} 
is continuous on all sectors $\dom_j$, since $\bar{h}$ is assumed to be piecewise continuous. 
Using the stability properties of the characteristic curves from \cref{lem:char_stability}
implies the continuity of all $\bar{y}_i^0$ by \eqref{def_operator_case1}.
To show the continuity of the middle states $\bar{y}_i^j$ for $j \in [n-1]$ and $i>j$,
observe that the operators $\mathcal{F}_i^j$ from \eqref{def_Fcal_ij} are continuous
w.r.t.\ all arguments, since this holds for $F_i^j$ by assumption and since
$t \mapsto d_i^{j,k}(t)$ is continuous by $\bar{z} \in \pcont$,
the traces $ \bar{y}_i^{j-1} | _{\Sigma_j} $ are continuous,
and the intersection times $t_i^j$ in \eqref{def_operator_case3} also depend continuously on $(t, \x)$
by \cref{lem:char_stability}. This implies $\bar{y}_i^j \in C^0(\dom_j)$ and thus $\bar{y} \in \pcont$.
	
The unique existence of a fixed point of $\mathcal{T}$ follows by Banach's fixed-point theorem.
Clearly, it is a broad solution to the semilinear system by construction.
\end{proof}

We now prove a $PC^0$-bound for broad solutions of the semilinear system.
\begin{theorem} \label{thm:sl_pc0_bound}
Let the assumptions of \cref{thm:existence_broadsol_sl} hold and
$\bar{y} \in \pcont$ be the broad solution of \eqref{problem_ref_sl}.
There is $C = C(T, A, \| \bar{z} \|_{PC^1(\dom)}, L_{F_1}, L_{\bar{h}} ) > 0$ with
\begin{equation} \label{explicit_sl_pc0_bound}
\| \bar{y} \|_{PC^0(\dom)} \leq C ( \| \bar{h}(0) \|_{PC^0(\dom)}
+ \max_{i,j} \| F_i^j (\cdot, 0, 0, 0) \|_{C^0([0,T])} +  \| u \|_\infty ) ,
\end{equation}
where $\| u \|_\infty \coloneqq \max \{ \|u_l \|_{C^0([-\el,0]; \R^n)}, \|u_r \|_{C^0([0,\el]; \R^n)} \}$.
\end{theorem}
\begin{proof}	
Let $\bar{y} \in \pcont$ be the broad solution of \eqref{problem_ref_sl} and define
\begin{equation*} 
Y(\tau) = \| \bar{y} \|_{PC^0(\dom^\tau)} = \max \big\{ |\bar{y}^j(t,\x)| \ | \ 
(t, \x) \in \dom_j, \ t \leq \tau, \ j \in [n]_0 \big\}
\end{equation*}
for $\tau \in [0,T]$. Due to the piecewise continuity of $\y$, $Y$ is also continuous.

First, we bound $\bar{h}_i$ from \eqref{hi_bar_broadsol}.
Let $C_{\partial l} = C_{\partial l} ( \| \bar{z} \|_{PC^1(\dom)} )$ be a bound for
\begin{equation*}
| \partial_t l_i(\bar{z}) + \bar{\lambda}^{( \bar{z}, \xs)} _i \partial_{\bar{x}} l_i(\bar{z}) | 
\leq C_{\partial l}
\quad \forall i \in [n]
\end{equation*}
on all $\dom_j$. 
With the Lipschitz continuity property \eqref{lipschitz_sl_sourceterm}
of $\bar{h}$, this yields for suitable $t' \leq t$ 
and any curve $\gamma$ (e.g., a characteristic curve) that
\begin{equation*}
\int_{t'}^t | \bar{h}_i(\y)(s, \gamma(s)) | \ds 
\leq T \| l \| \| \bar{h}(0) \|_{PC^0(\dom)}
+ \left( \| l \| L_{\bar{h}} + C_{\partial l} \right) \int_{t'}^t Y(s) \ds .
\end{equation*}
By \eqref{semilinear_characteristiceq_0n} this yields 
for $(t,\x) \in \dom_j$ with $j \in \{ 0, n\}$ and $i \in [n]$ that
\begin{equation} \label{bound_sl_y0n_i}
| \bar{y}^{j}_i(t,\x) | \leq \| l \| \| u \|_\infty 
+  T \| l \| \| \bar{h}(0) \|_{PC^0(\dom)}
+ \left( \| l \| L_{\bar{h}} + C_{\partial l} \right) \int_{0}^t Y(s) \ds .
\end{equation}
For the middle states $j \in [n-1]$ and $i>j$
(the other case $i \leq j$ is again analogous),
first observe that with $\mathcal{F}_i^j$ from \eqref{def_Fcal_ij} and the Lipschitz
estimate \eqref{lipschitz_Fcal_ij} it holds that
\begin{equation*}
	\begin{aligned}
		& \big| \mathcal{F}_i^j \big( t, \bar{y}^{j-1}_i (t, s_j t),
		\bar{y}^{j-1} (t, s_j t), \bar{y}^{j} (t, s_j t)  \big) \big| \\
		\leq & \max_{i,j} \| F_i^j (\cdot, 0, 0, 0) \|_{C^0([0,T])} 
		+ L_{F_1} \|  \bar{y}^{j-1}_i  \|_{C^0(\dom_{j-1}^t)} 
		+ 2 L_{\mathcal{F}_2} \| \bar{y} \|_{PC^0(\dom^t)} 
	\end{aligned}
\end{equation*}
where we used $\mathcal{F}_i^j(\cdot,0,0,0) = F_i^j(\cdot,0,0,0)$, since the terms
involving $d_i^{j,k}$ in \eqref{def_Fcal_ij} are linear in $v^1, v^2$.
Thus we obtain by \eqref{semilinear_characteristiceq_j}
for $j \in [n-1]$, $i>j$ and $(t,\x) \in \dom_j$ that
\begin{equation*}
	\begin{aligned}
		| \bar{y}^{j}_i(t,\x) | & \leq T \| l \| \| \bar{h}(0) \|_{PC^0(\dom)}
		+ \left( \| l \| L_{\bar{h}} + C_{\partial l} \right) \int_{t_i^{j}}^t Y(s) \ds \\ 
		& + \| F (\cdot, 0) \| 
		+ L_{F_1} \|  \bar{y}^{j-1}_i  \|_{C^0(\dom_{j-1}^t)}
		+ 2 L_{\mathcal{F}_2} \| \bar{y} \|_{PC^0(\dom^t)}
	\end{aligned}
\end{equation*}
with $ \| F (\cdot, 0) \| = \max_{i,j} \| F_i^j (\cdot, 0, 0, 0) \|_{C^0([0,T])}$ as abuse of notation.
Repeatedly applying the above estimate and then using \eqref{bound_sl_y0n_i} results in
\begin{equation*} 
	\begin{aligned}
		| \bar{y}^{j}_i(t,\x) | & \leq \big( 
		T \| l \| \| \bar{h}(0) \|_{PC^0(\dom)}
		+ \| F (\cdot, 0) \| + \| l \| \| u \|_\infty \big) \sum_{k=0}^{n} (L_{F_1}) ^k \\
		& + \big( ( \| l \| L_{\bar{h}} + C_{\partial l} ) \int_{0}^t Y(s) \ds
		+ 2 L_{\mathcal{F}_2} \| \bar{y} \|_{PC^0(\dom^t)} \big) \sum_{k=0}^{n} (L_{F_1}) ^k .
	\end{aligned}
\end{equation*} 
Recall that the assumption \eqref{smallness_LF2} in \cref{thm:existence_broadsol_sl}
implies $ 2 n L_{\mathcal{F}_2} \sum_{k=0}^{n} (L_{F_1})^k \leq \frac12 $.
With this and $ \bar{y}^j  = \sum_{i=1}^n \bar{y}^j_i  r_i(\bar{z}^j)$,
we obtain after rearranging the inequality that
\begin{equation*}
	Y(t) \leq 2n \sum_{k=0}^{n} (L_{F_1}) ^k \Big( c^0_{\mathrm{data}} 
	+ c^1_{\mathrm{data}} \int_{0}^t Y(s) \, \ds \Big)
	\quad \forall t \in [0,T]
\end{equation*}
with the constants depending on the data given by
\begin{equation*} 
	c^0_{\mathrm{data}} = 
	T \| l \| \| \bar{h}(0) \|_{PC^0(\dom)}
	+ \| F (\cdot, 0) \| + \| l \| \| u \|_\infty , \qquad
	c^1_{\mathrm{data}} =  \| l \| L_{\bar{h}} + C_{\partial l} .
\end{equation*}
Applying Gronwall's inequality yields
\eqref{explicit_sl_pc0_bound}.
\end{proof}

The next result shows that the broad solution depends continuously on the data of the problem.
For a single initial value problem, this result is proved in \cite{Bressan2005}.

\begin{theorem} \label{thm:sl_continuous_dep_data}
	Let $(\bar{z}^{[\nu]})_{\nu \in \N}$ be a bounded sequence in $\pdiff$
	with values in $\Yval$
	and $\bar{z} \in \pdiff$ with values in $\Yval$. 
	Moreover, let $\xs, \, \xs^{\nu} \in (-\varepsilon, \varepsilon)$ for all $\nu \in \N$.
	Assume all $\bar{A}^{(\bar{z}^{[\nu]}, \xs^{\nu})}$
	and $\bar{A}^{(\bar{z}, \xs)}$ satisfy \eqref{minimal_angle_refspace} and \eqref{entropy_cond_Abar}.
	Let $\bar{y} \in \pcont$ be the unique broad solution of \eqref{problem_ref_sl}, 
	and $\bar{y}^{[\nu]} \in \pcont$ the unique broad solution of
	\begin{equation*} 
		\bar{y}^{[\nu]}_t + \bar{A}^{(\bar{z}^{[\nu]}, \xs^{\nu})} \bar{y}^{[\nu]}_{\x} 
		= \bar{h}^\nu(\bar{y}^{[\nu]}) 
		\quad \text{on } \dom, \quad \bar{y}^{0,[\nu]}(0, \cdot) = u_l^\nu, 
		\quad \bar{y}^{n,[\nu]}(0, \cdot) = u_r ^\nu , 
	\end{equation*}
	with the interior boundary conditions 
	$F_i^{j,\nu} : [0,T] \times \R \times \R^n \times \R^n \to \R $ for all $\nu \in \N$.
	Assume that the source terms $\bar{h}^\nu, \bar{h}$ and the interior
	boundary conditions $F_i^{j,\nu}, F_i^{j}$ have the regularity 
	assumed in \cref{thm:existence_broadsol_sl} and 
	the Lipschitz assumptions \eqref{lipschitz_sl_sourceterm} and \eqref{lipschitz_Fij}
	as well as \eqref{def_theta_basisdeviation} for $\bar{z}^{[\nu]}$
	hold uniformly for all $\nu \in \N$, with $L_{\mathcal{F}_2}$ from \eqref{smallness_LF2}
	sufficiently small in the sense of \cref{thm:existence_broadsol_sl}. 
	Further, assume that
	\begin{align*} 
		\xs^{\nu} & \to \xs , \\
		\bar{z}^{[\nu]} & \to \bar{z} && \text{in } PC^0(\dom) , \\ 
		\bar{h}^{\nu}(\bar{w}) & \to \bar{h}(\bar{w}) && \text{in } PC^0(\dom) \ \forall \bar{w} \in \pcont , \\
		F_i^{j,\nu} & \to F_i^{j} && \text{locally uniformly on } 
		[0,T] \times \R \times \R^n \times \R^n \ \forall i \neq j \in [n] , \\ 
		u_l^\nu & \to u_l && \text{uniformly on } [-\el,0] , \\ 
		u_r^\nu & \to u_r && \text{uniformly on } [0,\el] ,
	\end{align*}
	for $\nu \to \infty$.
	Then $\lim_{\nu \to \infty} \| \bar{y}^{[\nu]} - \bar{y} \|_{PC^0(\dom)} = 0$.
\end{theorem}
\begin{proof}
	The proof is analogous to the proof of \cite[Theorem 3.5]{Bressan2005}.
	It can be adapted to handle interior boundary conditions
	and a source term given in functional form.
	Note that \eqref{A_bar_z} implies
	$\lim_{\nu \to \infty} \| \bar{A}^{(\bar{z}^{[\nu]}, \xs^{\nu})} 
	- \bar{A}^{(\bar{z}, \xs)} \|_{PC^0(\dom; \R^{n \times n})} = 0$.
	Moreover, $\bar{z}^{[\nu]} \to \bar{z}$ in $PC^0(\dom)$ and the continuity of $l_i$
	imply the uniform convergence on $[0,T]$ of the coefficients $d_i^{j,k}$
	in \eqref{def_Fcal_ij} associated with $\bar{z}^{[\nu]}$, so that the basis converted
	interior boundary conditions $\mathcal{F}_i^{j,\nu}$ converge to $\mathcal{F}_i^{j}$
	locally uniformly.
\end{proof}

Next, we prove that the broad solution from \cref{thm:existence_broadsol_sl}
is actually a classical solution if the data is suitably smooth.
We only state the result for the space derivative.
The corresponding result for the time derivative is analogous.
The result is restricted to the case of a local source term 
instead of the functional form \eqref{problem_ref_sl},
since this generality is not needed here.
First, we need the following auxiliary results.
\begin{lemma} \label{lem:differentiability_Itauxi}
Let $f \in C^1(\R^2)$ and $l, u \in C^1(\R^3; \R^n)$.
Then $I : \R^2 \to \R$ with
$I(t, \x) = \int_{f(t, \x)}^{t}  l_s(s, t, \x)^ \top u(s, t, \x) \intd s$
is continuously differentiable, and
\begin{equation*}
\begin{aligned}
I_{\x} (t, \x) = &  \int_{f(t, \x)}^{t} 
( l_s ^\top u_{\x} - l_{\x} ^\top u_s ) (s, t, \x) \intd s 
- l_{s} (f(t,\x) , t, \x)^\top \ u(f(t,\x) , t, \x) \ f_{\x}(t, \x) \\ 
& + l_{\x} (t, t, \x)^\top \ u(t, t, \x) 
- l_{\x} (f(t,\x) , t, \x)^\top \ u(f(t, \x) , t, \x) . \\ 
\end{aligned} 
\end{equation*}
\end{lemma}
\begin{proof}
The proof for $f \equiv 0$ can be found in \cite[Lemma 3.2]{Bressan2005}.
It can be easily extended to any $f \in C^1(\R^2)$
by differentiating w.r.t.\ the lower integration limit.
\end{proof}

\begin{lemma} \label{lem:YnuZnu}
Let $\alpha, \beta > 0$, $m \in (\frac12, 1]$,
and $Y_{\nu} \in C^0([0,T])$ for $\nu \in \N_0$. 
If 
\begin{equation*}
Y_0  \leq \alpha, \qquad
Y_{\nu + 1}(t) \leq \alpha m^{\nu} + \beta \int_0^t Y_{\nu }(s) \intd s 
+ \frac12 Y_{\nu }(t)
\quad \forall t \in [0,T], \ \forall \nu \in \N_0,
\end{equation*}
then $Y_{\nu}(t) \leq m^{\nu} K e^{\gamma t}$
with $K = \frac{2 \alpha}{2m - 1}$ and $\gamma = \frac{2 \beta}{2m - 1}$
for all $t \in [0,T]$ and $\nu \in \N$.
\end{lemma}
\begin{proof}
The result extends \cite[Lemma 3.3]{Bressan2005} and the proof works by induction.
\end{proof}

\begin{theorem} \label{thm:c1_sol_sl}
	Let $\bar{f} \in PC^1(\dom \times \R^n; \R^n)$ such that there is $L_{\bar{f}} > 0$ with
	\begin{equation*}
		\| \bar{f}^j(\cdot, v^1) - \bar{f}^j(\cdot, v^2) \|_{C^0(\dom_j)} 
		\leq L_{\bar{f}} | v^1 - v^2 | \qquad \forall v^1, v^2 \in \R^n, \ \forall j \in [n]_0.
	\end{equation*} 
	Suppose that the interior boundary conditions $F_i^j$ satisfy the assumptions of
	\cref{thm:existence_broadsol_sl}.
	We abbreviate on $\Sigma_j$ for $k \in \{ j-1, j \}$ and $i \neq j$
	\begin{equation} \label{def_Mjk_cij}
		M^{j,k} \coloneqq s_j I - \bar{A}^{(\bar{z}^k, \xs)} , \quad
		c_i^{j,k} \coloneqq ( l^{*,j}_i )^\top M^{j,k} r^{*,j}_i , \quad
		\Pi^j_i w \coloneqq w - \big( ( l^{*,j}_i )^\top w \big) r^{*,j}_i,
	\end{equation}
	with the bases from \cref{def:fixed_shock_basis}.
	Assume that $\theta$ from \eqref{def_theta_basisdeviation} is small enough that
	\begin{equation} \label{cij_nondeg}
		| c_i^{j,k} | \geq \bar{\eta}_{\min} / 2 
		\qquad \forall i \neq j \in [n], \ k \in \{ j-1, j \}
	\end{equation}
	with $\bar{\eta}_{\min}$ from \eqref{minimal_angle_refspace},
	and that $L_{\mathcal{F}_2}$ from \eqref{smallness_LF2} is sufficiently small such that
	\begin{equation} \label{smallness_LF2_pc1}
		\tilde{C}_1 \, L_{\mathcal{F}_2} \sum_{k=0}^{n-1} (L_{F_1})^k \leq \tfrac12
	\end{equation}
	with a constant $\tilde{C}_1$ only depending on
	$T$, $A$, $\| \bar{z} \|_{PC^1(\dom)}$, $\bar{\eta}_{\min}$ and $L_{F_1}$.
	If all $F_i^j$ and $(u_l, u_r)$ are continuously differentiable on their respective domains,
	the solution of 
	\begin{equation} \label{problem_ref_sl_local} 
		\bar{y}_t + \bar{A}^{(\bar{z}, \xs)} \bar{y}_{\x} = \bar{f}(\cdot, \bar{y}) 
		\quad \text{on } \dom, \quad \bar{y}^0(0, \cdot) = u_l, 
		\quad \bar{y}^n(0, \cdot) = u_r,
	\end{equation}
	with interior boundary conditions $F_i^j$
	is a classical solution $\bar{y} \in \pdiff$.
	Moreover, $\bar{v} \coloneqq \bar{y}_{\x} \in \pcont$
	is a broad solution of the semilinear system
	\begin{equation} \label{pde_sl_yx}
		\begin{aligned}
			\bar{v}_t + \bar{A}^{(\bar{z}, \xs)} \bar{v}_{\x} & 
			= \bar{f}_{\x}(\cdot, \bar{y}) + \bar{f}_{\y} (\cdot, \bar{y}) \bar{v}
			- ( \partial_{\x} \bar{A}^{(\bar{z}, \xs)} ) \bar{v}  \\ 
			\bar{v}^0(0, \cdot) &= u_l', \quad \bar{v}^n(0, \cdot) = u_r' ,
		\end{aligned}
	\end{equation}
	with interior boundary conditions for all $i, j \in [n]$ with $i \neq j$
	and $t \in [0,T]$ given by
	\begin{equation} \label{K_ij_coupling}
		\begin{aligned}
			\bar{v}^{j}_{*,i} (t, s_j t) &= K_i^j \big( t, \bar{v}^{j-1}_{*,i} (t, s_j t),
			\bar{v}^{j-1} (t, s_j t), \bar{v}^{j} (t, s_j t)  \big) && i>j, \\
			\bar{v}^{j-1}_{*,i} (t, s_j t) &= K_i^j \big( t, \bar{v}^{j}_{*,i} (t, s_j t),
			\bar{v}^{j-1} (t, s_j t), \bar{v}^{j} (t, s_j t) \big) && i<j,
		\end{aligned}
	\end{equation}
	where $\bar{v}^{k}_{*,i} = ( l^{*,j}_i )^\top \bar{v}^{k}$ for $k \in \{ j-1, j \}$
	and $K_i^j$ is given with \eqref{def_Mjk_cij} for $i>j$ by
	\begin{equation} \label{sl_yx_nodecond}
		\begin{aligned}
			K_i^j(t, v, & w_1, w_2) = \frac{1}{c_i^{j,j}} \Big[
			(F_i^j)_t - ( l^{*,j}_i )^\top \bar{f}^j(\cdot, \bar{y}^j) 
			- ( l^{*,j}_i )^\top M^{j,j} \Pi^j_i w_2 \\
			& + \partial_2(F_i^j) \big( ( l^{*,j}_i )^\top \bar{f}^{j-1}(\cdot, \bar{y}^{j-1}) 
			+ c_i^{j,j-1} v + ( l^{*,j}_i )^\top M^{j,j-1} \Pi^j_i w_1 \big) \\
			& + \partial_3(F_i^j) \big( \bar{f}^{j-1}(\cdot, \bar{y}^{j-1}) + M^{j,j-1} w_1 \big)
			+ \partial_4(F_i^j) \big( \bar{f}^{j}(\cdot, \bar{y}^{j}) + M^{j,j} w_2 \big) \Big] ,
		\end{aligned}
	\end{equation}
	where all quantities are evaluated in $(t, s_j t)$, $\partial_k(F_i^j)$ denotes the 
	partial derivative of $F_i^j$ w.r.t.\ the $k$-th argument
	for $k \in \{2,3,4\}$. The case $i<j$ is analogous.
\end{theorem}
\begin{proof}
The proof in a smooth regime without shock curves can be found in \cite[Theorem 3.6]{Bressan2005}.
We extend the proof to the setting of this paper.

By $\mathcal{T}$ we again denote the operator from the proof of \cref{thm:existence_broadsol_sl},
see \eqref{def_operator_case1} and \eqref{def_operator_case3}.
Recall that the proof of \cref{thm:existence_broadsol_sl} 
implies $\bar{y} = \lim_{\nu \to \infty} \bar{y}^{[\nu]}$
in $\pcont$ for a sequence generated by $\bar{y}^{[0]} \in \pcont$ and 
$\bar{y}^{[\nu+1]} = \mathcal{T}(\bar{y}^{[\nu]})$ for $\nu \in \N_0$.
Now suppose $\bar{y}^{[0]} \in \pdiff$. 
Due to \cref{lem:differentiability_Itauxi},
the auxiliary source term $\bar{h}_i$ from \eqref{hi_bar_broadsol}
is continuously differentiable w.r.t.\ $(t, \x)$.
Note that $t \mapsto d_i^{j,k}(t)$ from \eqref{def_Fcal_ij} is 
$C^1$ due to $\bar{z} \in \pdiff$, 
so that $\mathcal{F}_i^j$ is $C^1$ w.r.t.\ all arguments.
Together with the differentiability properties of the characteristic curves
and the intersection times $t_i^j$ (see \cref{lem:char_stability}),
this implies $\bar{y}^{[\nu+1]} \in \pdiff$
if $\bar{y}^{[\nu]} \in \pdiff$
and thus $( \bar{y}^{[\nu]} )_{\nu \in \N} \subset \pdiff$.

We now prove that also $\y \in \pdiff$. 
For this, fix $i \in [n]$ and $(t, \x) \in \dom_j$ for some $j \in [n]_0$. 
By definition of $\mathcal{T}$, it holds that
\begin{equation} \label{iter_generalform}
	\y_i^{[\nu + 1]}(t, \x) = \varphi(t, \x) 
	+ \int_{t_i(t,\x)}^t \Big[ l_i(\bar{z})^\top \bar{f}(\y^{[\nu]}) 
	+ \big( \frac{\mathrm d}{\mathrm{ds}} l_i(\bar{z}) \big)^ \top \y^{[\nu]} \Big] \intd s ,
\end{equation}
where the integrand is evaluated in $(s, \x_i(s;t,\x))$ and, depending on $i$ and $j$,
\begin{equation*}
	\begin{aligned}
		\varphi(t,\x) \in \big\{ & u_{l,i}\big( \x_i(0;t,\x) \big), u_{r,i}\big( \x_i(0;t,\x) \big), \\
		& \mathcal{F}_i^{j} \big( t_i^{j}, \y^{[\nu+1],j-1}_i, \y^{[\nu],j-1}, \y^{[\nu],j} \big), 
		\mathcal{F}_i^{j+1} \big( t_i^{j+1}, \y^{[\nu+1],j+1}_i, \y^{[\nu],j}, \y^{[\nu],j+1} \big) \big\}
	\end{aligned}
\end{equation*}
with $t_i(t,\x) = 0$ in the first two cases, corresponding to $j \in \{ 0, n \}$,
and $t_i(t,\x) = t_i^{j}$ for $i>j$ and $t_i(t,\x) = t_i^{j+1}$ for $i \leq j$
in the latter two, and all arguments of $\mathcal{F}_i^j$, $\mathcal{F}_i^{j+1}$
from \eqref{def_Fcal_ij} evaluated at
$(t_i^{j}, s_j t_i^{j})$ and $(t_i^{j+1}, s_{j+1} t_i^{j+1})$, respectively.
Since
$ (\y^{[\nu + 1]}_i)_t + \bar{\lambda}^{(\bar{z}, \xs)}_i (\y^{[\nu + 1]}_i)_{\x}
= l_i^\top \bar{f}(\y^{[\nu]}) + (l_{i,t} + \bar{\lambda}^{(\bar{z}, \xs)}_i l_{i,\x})
	^ \top \y^{[\nu ]}$,
it can be checked that
\begin{equation} \label{pde_during_iter}
\y^{[\nu + 1]}_t + \bar{A}^{(\bar{z}, \xs)} \y^{[\nu + 1]}_{\x}
= \bar{f}(\y^{[\nu]}) + \sum_{i=1}^n 
\big( (l_{i,t} + \bar{\lambda}^{(\bar{z}, \xs)}_i l_{i,\x}) ^\top (\y^{[\nu]} - \y^{[\nu + 1]}) \big) r_i
\end{equation}
holds in the classical sense on all $\dom_j$
with the abbreviations $l_i = l_i(\bar{z})$ and $r_i = r_i(\bar{z})$.

Recall that the proof of \cref{thm:existence_broadsol_sl} implies
\begin{equation} \label{ynu_geometric_sum}
	\| \y^{[\nu+1]} - \y^{[\nu]}\|_{\pcont} \leq C q^{\nu}
\end{equation}
for some fixed $q \in (\frac34, 1)$ and a generic constant $C>0$ used in the sequel.
In particular, this shows that $\| \y^{[\nu]}\|_{\pcont}$ is uniformly bounded for all $\nu \in \N$. 
Rearranging \eqref{pde_during_iter} shows that
$\| \y^{[\nu]}_t \|_{\pcont} \leq C (1 + \| \y^{[\nu]}_{\x} \|_{\pcont} )$. 
Define for $t \in [0,T]$ and all $\nu \in \N$
\begin{equation} \label{def_ZnuYnu}
Y_{\nu}(t) = \max_{i \in [n]} \| (\y^{[\nu]}_{i})_{\x} \|_{PC^0(\dom^t)}, \quad
Z_{\nu}(t) = \max_{i \in [n]} 
\| (\y^{[\nu + 1]}_{i})_{\x} - (\y^{[\nu]}_{i})_{\x} \|_{PC^0(\dom^t)} .
\end{equation}
Note that $\| \y^{[\nu]}_{\x} \|_{PC^0(\dom^t)} \leq C(1 + Y_{\nu}(t) )$.
Choosing $l(s,t, \x) = l_i \left( s, \x_i(s; t, \x) \right)$ and
$u(s,t, \x) = \y^{[\nu]} \left( s, \x_i(s; t, \x) \right)$
in \cref{lem:differentiability_Itauxi} and using that
\begin{align*}
	l_s ^\top u_{\x} - l_{\x} ^\top u_s 
	& = \Big( \frac{\partial l_i}{\partial t} 
		+ \frac{\partial l_i}{\partial \x} \bar{\lambda}^{(\bar{z}, \xs)}_i \Big)
	^\top \y^{[\nu]}_{\x} \frac{\partial \x_i}{\partial \x}
	- \frac{\partial l_i ^\top}{\partial \x} \frac{\partial \x_i}{\partial \x} 
	\Big( \y^{[\nu]}_{t} + \y^{[\nu]}_{\x} \bar{\lambda}^{(\bar{z}, \xs)}_i \Big) \\
	& = \frac{\partial l_i ^\top}{\partial t}  \y^{[\nu]}_{\x} \frac{\partial \x_i}{\partial \x}
	- \frac{\partial l_i ^\top}{\partial \x} \y^{[\nu]}_{t} \frac{\partial \x_i}{\partial \x} ,
\end{align*}
it follows from differentiating \eqref{iter_generalform} w.r.t.\ $\x$ that
\begin{equation} \label{Iter_xderivative}
\begin{aligned}
& \y^{[\nu + 1]}_{i, \x} 
= \varphi_{\x}  
+ \int_{t_i}^t \Big[ \frac{\partial l_i}{\partial \x} ^\top \frac{\partial \x_i}{\partial \x} \bar{f}(\y^{[\nu]})
+ l_i ^\top \Big( \bar{f}_{\x}(\y^{[\nu]}) \frac{\partial \x_i}{\partial \x} 
+ \bar{f}_{\y}(\y^{[\nu]}) \y^{[\nu]}_{\x} \frac{\partial \x_i}{\partial \x} \Big) \Big] \intd s \\
& + \int_{t_i}^t \! \Big[ \frac{\partial l_i}{\partial t}^\top \!\! \y^{[\nu]}_{\x} \frac{\partial \x_i}{\partial \x}
- \frac{\partial l_i ^\top}{\partial \x} \y^{[\nu]}_{t} \frac{\partial \x_i}{\partial \x} \Big] \intd s 
- \big( \big( \frac{\partial l_i}{\partial t} 
+ \bar{\lambda}^{(\bar{z}, \xs)}_i \frac{\partial l_i}{\partial \x} \big)^\top \! 
\y^{[\nu]} \big) (t_i, \x_i(t_i)) \frac{\partial t_i}{\partial \x}  \!\!\! \\
& - \big( l_i^\top \bar{f}(\y^{[\nu]}) \big) (t_i, \x_i(t_i)) \frac{\partial t_i}{\partial \x}
+ \big( \frac{\partial l_i}{\partial \x} ^\top \!\! \y^{[\nu]} \big)(t, \x) 
- \big( \frac{\partial l_i}{\partial \x} ^\top \!\! \y^{[\nu]} \big)(t_i, \x_i(t_i)) 
\frac{\partial \x_i(t_i; t,\x)}{\partial \x} \!\!\! \!\!
\end{aligned}
\end{equation}
on $\dom_j$, with the abbreviations $t_i = t_i(t,\x)$, $\x_i = \x_i(\cdot; t,\x)$. This implies
\begin{equation} \label{Ynu_i_t}
| \y^{j, [\nu + 1]}_{i, \x} (t, \x) | \leq
\| \varphi_{\x} \| + C + C \int_{t_i}^t Y_{\nu}(s) \intd s.
\end{equation}
If $j \in \{0, n \}$, then $t_i = 0$ and 
$\varphi(t, \x) \in \{ u_{l,i}(\x_i(0;t,\x)), u_{r,i}(\x_i(0;t,\x))\}$. Hence,
$ | \varphi_{\x} | \leq C ( \| u_l \|_{C^1} + \| u_r \|_{C^1} )$.
If $(t, \x) \in \dom_j$ for $j \in [n-1]$,
we assume w.l.o.g.\ that $i > j$. Then
$\varphi(t, \x) = \mathcal{F}_i^j \big( t_i, \y^{j - 1, [\nu + 1]}_{i}(t_i, s_j t_i),
\y^{j-1, [\nu]}(t_i, s_j t_i), \y^{j, [\nu]}(t_i, s_j t_i) \big)$ and
\begin{equation*}
	|\varphi_{\x}(t, \x)| \leq C \Big| \frac{\partial t_i}{\partial \x} \Big|  
	\big( \| (F_i^j)_t \| + C + L_{F_1} | \y^{j-1, [\nu + 1]}_{i, \x}(t_i, \x_i)|
	+ L_{\mathcal{F}_2} Y_{\nu}(t_i) \big) 
\end{equation*}
with $L_{F_1}, L_{\mathcal{F}_2}$ from \eqref{lipschitz_Fij} and \eqref{smallness_LF2},
where we differentiated \eqref{def_Fcal_ij} w.r.t.\ $t$ to obtain
\begin{equation*}
	( \mathcal{F}_i^j )_t ( t, w, v^1, v^2 ) = (F_i^j)_t 
	- \partial_2 F_i^j \, \big( \tfrac{\mathrm{d}}{\dt} d_i^{j,j-1} \big)^\top v^1
	+ \big( \tfrac{\mathrm{d}}{\dt} d_i^{j,j} \big)^\top v^2 .
\end{equation*}
Evaluated at the traces of $\y^{[\nu]}$, this yields
\begin{equation*}
	\big| ( \mathcal{F}_i^j )_t \big| \leq \| (F_i^j)_t \| 
	+ ( 1 + L_{F_1} ) \max_{k} \big\| \tfrac{\mathrm{d}}{\dt} d_i^{j,k} \big\|_{C^0([0,T])} 
	\, \| \y^{[\nu]} \|_{\pcont} \leq \| (F_i^j)_t \| + C ,
\end{equation*}
where $l^{*,j}_i$ is constant, hence
$\tfrac{\mathrm{d}}{\dt} d_i^{j,k} = \nabla l_i(\bar{z}^{k}) ( \bar{z}^{k}_t + s_j \bar{z}^{k}_{\x} )$
on $\Sigma_j$ is bounded in terms of $\| \nabla l \|$, $s_j$ and $\| \bar{z} \|_{PC^1(\dom)}$,
and $\| \y^{[\nu]} \|_{\pcont}$ is uniformly bounded w.r.t.\ $\nu$ by \eqref{ynu_geometric_sum}.
Since $Y_{\nu}$ is monotone non-decreasing,
$Y_{\nu}(t_i) \leq Y_{\nu}(t)$ holds.
Applying \eqref{Ynu_i_t} repeatedly to replace $| \y^{j-1, [\nu + 1]}_{i, \x}(t_i, \x_i)|$
in the previous estimate until we arrive on $\dom_0$ yields
\begin{equation} \label{pc1_Ynu_estimate}
	Y_{\nu + 1}(t) \leq C + C \int_0^t Y_{\nu}(s) \intd s 
	+ \tilde{C}_1 L_{\mathcal{F}_2} \sum_{k=0}^{n-1} (L_{F_1})^k \, Y_{\nu}(t)
	\quad \forall t \in [0,T], \ \forall \nu \in \N.
\end{equation}
By \eqref{smallness_LF2_pc1}, there holds
$C_1 = \tilde{C}_1 L_{\mathcal{F}_2} \sum_{k=0}^{n-1} (L_{F_1})^k \leq \frac12$.
Choosing $m = 1$ in \cref{lem:YnuZnu} implies $\sup_{\nu \in \N} Y_{\nu}(T) < \infty$,
showing that $(\y^{[\nu]}_{\x})_{\nu \in \N}$ is bounded in $\pcont$.
Rearranging \eqref{pde_during_iter} yields the same for $(\y^{[\nu]}_{t})_{\nu \in \N}$.
Taking the difference of \eqref{Iter_xderivative} for $\nu + 1$ and $\nu$
implies with the same argument as in \eqref{pc1_Ynu_estimate} 
and the estimate \eqref{ynu_geometric_sum} that
\begin{equation*} 
	Z_{\nu + 1}(t) \leq C q ^{\nu} + C \int_0^t Z_{\nu}(s) \intd s 
	+ \frac12 Z_{\nu}(t)
	\quad \forall t \in [0,T], \ \forall \nu \in \N.
\end{equation*}
Now applying \cref{lem:YnuZnu} with $m = q < 1$ implies that 
$Z_{\nu}(T) \leq C q^{\nu}$ for all $\nu \in \N$. 
By definition \eqref{def_ZnuYnu} of $Z_{\nu}$, this proves that
$( \bar{y}^{[\nu]}_{\x} )_{\nu \in \N} \subset \pcont$ is a Cauchy sequence.
Rearranging \eqref{pde_during_iter} implies the same for $( \bar{y}^{[\nu]}_{t} )_{\nu \in \N}$.
Therefore, $( \bar{y}^{[\nu]})_{\nu \in \N} \subset \pdiff$ is convergent and $\y \in \pdiff$.
From the fact that $\y$ satisfies the characteristic equations and is continuously differentiable,
it can be easily deduced that $\y$ solves \eqref{problem_ref_sl_local} in the classical sense,
see again the proof of \cite[Theorem 3.6]{Bressan2005}. 

To prove that $\bar{y}_{\x}$ is a broad solution of \eqref{pde_sl_yx},
the associated integral equations must be satisfied. 
This is analogous to \cite[Theorem 3.6]{Bressan2005} or \cite{Douglis1952} and omitted here.
To derive the interior boundary condition for $\bar{y}_{\x}$, consider the
$j$-shock and a field $i > j$. Since $l^{*,j}_i$ is constant, it holds
for $k \in \{ j-1, j \}$ with 
$( \bar{y}^k_{\x} )_{*,i} \coloneqq ( l^{*,j}_i )^\top \bar{y}^k_{\x}$ that
\begin{equation*}
	\frac{\mathrm{d}}{\dt} \bar{y}^k_{*,i} (t, s_j t) = 
	( l^{*,j}_i )^\top \bar{y}^k_t + s_j ( \bar{y}^k_{\x} )_{*,i} .
\end{equation*}
Differentiating $\bar{y}^{j}_{*,i}(t, s_j t) = F_i^j ( t, \bar{y}^{j-1}_{*,i} (t, s_j t),
\bar{y}^{j-1} (t, s_j t), \bar{y}^{j} (t, s_j t) ) $
w.r.t.\ time (and neglecting the argument $(t, s_j t)$) results in
\begin{equation*}
	\begin{aligned}
		( l^{*,j}_i )^\top \big( \bar{y}^j_t + s_j \bar{y}^j_{\x} \big) 
		& = (F_i^j)_t 
		+ \partial_2(F_i^j) \, ( l^{*,j}_i )^\top \big( \bar{y}^{j-1}_t + s_j \bar{y}^{j-1}_{\x} \big) \\
		& + \partial_3(F_i^j) \big( \bar{y}^{j-1}_t + s_j \bar{y}^{j-1}_{\x} \big)
		+ \partial_4(F_i^j) \big( \bar{y}^{j}_t + s_j \bar{y}^{j}_{\x} \big) .
	\end{aligned}
\end{equation*}
Next, we insert
$\bar{y}^k_t = \bar{f}^k(\cdot, \bar{y}^k) - \bar{A}^{(\bar{z}^k, \xs)} \bar{y}^{k}_{\x}$
and use that $\bar{y}^k_t + s_j \bar{y}^k_{\x} 
= \bar{f}^k(\cdot, \bar{y}^k) + M^{j,k} \bar{y}^{k}_{\x}$
with the abbreviations \eqref{def_Mjk_cij} and
\begin{equation} \label{split_Mjk}
	( l^{*,j}_i )^\top M^{j,k} w 
	= c_i^{j,k} \, ( l^{*,j}_i )^\top w + ( l^{*,j}_i )^\top M^{j,k} \Pi^j_i w
	\qquad \forall w \in \R^n .
\end{equation}
With all quantities evaluated in $(t, s_j t)$, this yields
\begin{equation*}
	\begin{aligned}
		& ( l^{*,j}_i )^\top \bar{f}^j(\cdot, \bar{y}^j) 
		+ c_i^{j,j} ( \bar{y}^j_{\x} )_{*,i} 
		+ ( l^{*,j}_i )^\top M^{j,j} \Pi^j_i \bar{y}^{j}_{\x} \\
		& = (F_i^j)_t + \partial_2(F_i^j) \big( ( l^{*,j}_i )^\top \bar{f}^{j-1}(\cdot, \bar{y}^{j-1}) 
		+ c_i^{j,j-1} ( \bar{y}^{j-1}_{\x} )_{*,i} 
		+ ( l^{*,j}_i )^\top M^{j,j-1} \Pi^j_i \bar{y}^{j-1}_{\x} \big) \\
		& + \partial_3(F_i^j) \big( \bar{f}^{j-1}(\cdot, \bar{y}^{j-1}) + M^{j,j-1} \bar{y}^{{j-1}}_{\x} \big)
		+ \partial_4(F_i^j) \big( \bar{f}^{j}(\cdot, \bar{y}^{j}) + M^{j,j} \bar{y}^{j}_{\x} \big).
	\end{aligned}
\end{equation*}
Finally, rearranging the above identity proves \eqref{sl_yx_nodecond}.
For $i<j$ one solves for $( \bar{y}^{j-1}_{\x} )_{*,i}$ instead, dividing by $c_i^{j,j-1}$.
\end{proof}

\begin{remark} \label{rem:nodecond_yx_slassumptions}
	With $d_i^{j,k}$ from \eqref{def_Fcal_ij}, it holds
	$( l^{*,j}_i )^\top M^{j,k} = ( s_j - \bar{\lambda}^{(\bar{z}^k, \xs)}_i ) l_i(\bar{z}^k)^\top
	- ( d_i^{j,k} )^\top M^{j,k}$, which shows
	$c_i^{j,k} = s_j - \bar{\lambda}^{(\bar{z}^k, \xs)}_i + \mathcal{O}(\theta)$ with $\theta$
	from \eqref{def_theta_basisdeviation}. Since additionally
	$( l^{*,j}_i )^\top \Pi^j_i = 0$, the second summand in \eqref{split_Mjk} carries
	a factor $d_i^{j,k}$ and is thus also of order $\theta$.
	This shows that \eqref{cij_nondeg} holds for 
	sufficiently small $\theta > 0$ by \eqref{minimal_angle_refspace}.
	Reading off the coefficients of $v, w_1, w_2$ in \eqref{sl_yx_nodecond} shows
	that $K_i^j$ satisfies \eqref{lipschitz_Fij} with
	\begin{equation*} 
		L_{K_1} \leq L_{F_1} \max_{i \neq j} 
		\max \big\{ | c_i^{j,j-1} | \, | c_i^{j,j} |^{-1} , 
		| c_i^{j,j} | \, | c_i^{j,j-1} |^{-1} \big\} , \quad
		L_{K_2} \leq C_K \bar{\eta}_{\min}^{-1} \, L_{\mathcal{F}_2} , \!\!
	\end{equation*}
	where $C_K$ only depends on $\| l \|$ from \eqref{notation_norm_rllambda},
	the speeds $s_j$ and an a priori bound for $\bar{A}^{(\bar{z}, \xs)}$.
	In particular, \eqref{pde_sl_yx}, \eqref{K_ij_coupling} satisfy the assumptions of
	\cref{thm:existence_broadsol_sl} if $L_{\mathcal{F}_2}$ is sufficiently small,
	so that \cref{thm:sl_pc0_bound} and \cref{thm:sl_continuous_dep_data} apply.
\end{remark}

\subsection{The Quasilinear Problem} \label{sec:qlproblems}
In this section, we analyze 
the quasilinear problem in the reference space
which is given by \eqref{problem_ref_sl} with $\y = \bar{z}$. 
Using the results from \cref{sec:slproblems},
we prove various stability properties for piecewise
classical solutions of the quasilinear problem w.r.t.\ the initial state.


\begin{theorem}[\cite{LiYu1985}] \label{thm:grp_sol_physcoord}
Assume the jump $|u_L - u_R|$ and the constants $M_0, \, \varepsilon > 0$
in the definition \eqref{def_Uad} of $\Vad$ are sufficiently small.  
Then there is a time $T>0$ such that \eqref{GRP} has a unique 
piecewise $C^1$ entropy solution $y \in BV(\D; \R^n)$ with values in $\Yval$
for all initial states $u_0$ as in \eqref{initial_data_GRP} with $(u_l, u_r, \xs) \in \Uad$.
There are curves $\xi_k \in C^2([0,T])$ for all $k \in [n]$ with 
$\xi_k(0)=\xs$ and $(t, \xi_k(t)) \in \D$ for all $t \in [0,T]$ such that the
entropy condition \eqref{entropy_cond_speeds} is satisfied.
Moreover, there is a uniform $C^1$-bound of the smooth parts of the solutions
for all $(u_l, u_r, \xs) \in \Uad$ with $\Uad$ from \eqref{Uad_physicalcoord}.

If $y$ solves \eqref{GRP} with the initial state $(u_l, u_r, 0) \in \Uad$, 
then $\tilde{y}(x/t)$ with the principal part 
$\tilde{y}(\xi) = \lim_{t \to 0+} y(t, \xi t)$ solves the 
associated Riemann problem \eqref{RP} with the piecewise constant initial state 
$(u_l(0 -), u_r(0 +))$.
\end{theorem}
\begin{proof}
If $|u_L - u_R|$ and $M_0, \, \varepsilon > 0$ in \eqref{def_Uad} are sufficiently small, the jumps of 
all initial states $(u_l, u_r, \xs) \in \Uad$ are small, since then
\begin{equation*}
	|u_l(\xs) - u_r(\xs)| \leq |u_L - u_R| + 2 M_0 + 2 M_1 \varepsilon.
\end{equation*}
If even $u_l , u_r \in C^{2,1}$ and higher regularity of the data is assumed,
the unique existence of a piecewise $C^{2,1}$ solution
of \eqref{GRP} follows from \cite[Chapter 6, Theorem 5.1]{LiYu1985}.
By construction of $\Vad$, all associated Riemann Problems do not contain
rarefaction waves, see \cref{ass:Uad1}. 
With this, \cite[Chapter 6, Remark 3.1]{LiYu1985} shows
that the result remains valid if only $f \in C^2$. 
Checking \cite[Chapter 4, Theorem 4.1]{LiYu1985}
shows that it is sufficient if the source term and the piecewise smooth
initial states are only $C^1$. Then, the solution of \eqref{GRP} is only piecewise $C^1$. 
A uniform existence time $T>0$ follows by \cite[Chapter 2, Remark 4.3]{LiYu1985} 
and a uniform $C^1$-bound of the smooth parts of the solutions
for all controls $(u_l, u_r, \xs) \in \Uad$ follows by 
\cite[Chapter 2, Theorem 2.2]{LiYu1985} 
and \cite[Chapter 1, Theorem 4.1]{LiYu1985}. 
The asserted property of the principal part is shown in \cite[Chapter 6, Lemma 1.1]{LiYu1985}.
\end{proof}

By \cref{lem:lemma_rescaledsolution_Dj}, the GRP can be equivalently 
solved in the reference space.
In the sequel, we construct the associated problem in the reference space.
First, observe with \cref{lem:GRP_jump_equivalent}, applied with the bases
$l^{*,j}_k, r^{*,j}_k$ from \cref{def:fixed_shock_basis},
that the interior boundary condition is given for all $t \in [0,T]$ by
\begin{equation} \label{G_ij_coupling}
	\begin{aligned}
		\bar{y}^{j}_{*,i} (t, s_j t) &= G_i^j \big( t, \bar{y}^{j-1}_{*,i} (t, s_j t),
		\bar{y}^{j-1} (t, s_j t), \bar{y}^{j} (t, s_j t)  \big) && i>j, \\
		\bar{y}^{j-1}_{*,i} (t, s_j t) &= G_i^j \big( t, \bar{y}^{j}_{*,i} (t, s_j t),
		\bar{y}^{j-1} (t, s_j t), \bar{y}^{j} (t, s_j t) \big) && i<j,
	\end{aligned}
\end{equation}
with $\bar{y}^{k}_{*,i} = ( l^{*,j}_i )^\top \bar{y}^{k}$
for $k \in \{ j-1, j \}$ as in \eqref{F_ij_coupling}.
For $i > j$, $G_i^j$ is given by
\begin{equation} \label{intbdary_cond_GRP}
	G_i^j(t, v, v^1, v^2) = v
	- \sum_{k \neq i} \frac{ l_i ( v^1 , v^2 ) ^\top r^{*,j}_k}
	{l_i ( v^1 , v^2 ) ^\top r^{*,j}_i} ( l^{*,j}_k )^\top (v^2 - v^1) .
\end{equation}
The case $i < j$ is analogous.

The initial state 
$(\urefinit_l, \urefinit_r) \in \Vref \coloneqq C^1 \left( [-\el, 0]; \R^n \right) 
\times C^1 \left( [0, \el]; \R^n \right)$ in the reference space
associated to a control $u = (u_l, u_r, \xs) \in \Uad$ in physical coordinates is
\begin{equation} \label{trafo_initstate}
\urefinit_l(\x) = u_l \Big( \x + \frac{\x+\el}{\el} \xs \Big)
\ \ \forall \x \in [-\el, 0], \ \ \
\urefinit_r(\x) = u_r \Big( \x - \frac{\x-\el}{\el} \xs  \Big)
\ \ \forall \x \in [0,\el].
\end{equation}
This dilation operation is denoted by
\begin{equation} \label{def_dilation}
\begin{aligned}
\mathbf{D}_l: (u_l, \xs) \in C^0([-\el, \varepsilon]; \R^n) \times (-\varepsilon, \varepsilon)
& \mapsto \urefinit_l \in C^0([-\el, 0]; \R^n), \\ 
\mathbf{D}_r: (u_r, \xs) \in C^0([-\varepsilon, \el]; \R^n) \times (-\varepsilon, \varepsilon)
& \mapsto \urefinit_r \in C^0([0, \el]; \R^n).
\end{aligned}
\end{equation}

\begin{corollary} \label{cor:sol_grp_pc1_refspace}
Let the assumptions of \cref{thm:grp_sol_physcoord} be satisfied.
Then there is $T>0$ such that for all $(u_l, u_r, \xs) \in \Uad$
there is a unique solution $\bar{y} \in \pdiff$ of 
\begin{equation} \label{problem_ref_ql} 
\bar{y}_t + \bar{A}^{(\bar{y}, \xs)} \bar{y}_{\x} = \bar{g}^{(\bar{y}, \xs)}
\quad \text{on } \dom, \quad \bar{y}^0(0, \cdot) = \mathbf{D}_l(u_l, \xs), 
\ \bar{y}^n(0, \cdot) = \mathbf{D}_r(u_r, \xs),
\end{equation}
with $\bar{A}^{(\bar{y}, \xs)}$ from \eqref{A_bar_z},
and the source term in functional form given by
\begin{equation} \label{sourceterm_GRP_functionalform}
\begin{aligned}
(\bar{y}, \xs) \in \pdiff \times (-\varepsilon, \varepsilon)
& \mapsto \bar{g}^{(\bar{y},\xs)} \in \pdiff \\ 
\bar{g}^{(\bar{y},\xs)} (t, \x)  & \coloneqq
g(t, x^{(\bar{y}, \xs)}(t,\x), \bar{y}(t, \x) )  
\end{aligned}
\end{equation}
with $x^{(\bar{y}, \xs)}$ from \eqref{x_trafo_middle}, \eqref{x_trafo_left}.
The interior boundary condition is \eqref{G_ij_coupling}, \eqref{intbdary_cond_GRP}.
\end{corollary}

We denote the solution operator of the GRP in the reference space by
\begin{equation*}
\Sref : (u_l, u_r, \xs) \in \Uad \mapsto \y \in \pdiff 
\quad \text{subject to } \y \text{ solves \eqref{problem_ref_ql}, \eqref{sourceterm_GRP_functionalform}}.
\end{equation*}
Let $u = (u_l, u_r, \xs) \in \Uad$ and $\y = \Sref(u)$.
By \cref{cor:sol_grp_pc1_refspace} and \cref{lem:lemma_rescaledsolution_Dj},
the solution $y$ of \eqref{GRP} is obtained from $\y$ by the space transformation
associated with the curves $\xi_j^{(\y, \xs)}$ from \eqref{xij_curve_shock}.
Across these curves, $y$ satisfies the Rankine--Hugoniot condition \eqref{RH_cond},
since their speeds are given by \eqref{GRP_jump_speed}
and the correct jump property
follows from \eqref{G_ij_coupling} by \cref{lem:GRP_jump_equivalent}.
Therefore, they are the shock curves $\xi_j$ of $y$ from \cref{thm:grp_sol_physcoord},
and we write $\xi_j(u) \coloneqq \xi_j^{(\y, \xs)} = \xi_j^{(\Sref(u), \xs)}$ 
for all $j \in [n]$.

Using the uniform $C^1$-bound of the piecewise smooth parts of the solution 
in physical coordinates stated in \cref{thm:grp_sol_physcoord} and 
by possibly reducing $T>0$, the space transformations have uniform bounds in $C^1$
by \cref{lema:spacetrafo_properties} and \eqref{spacetrafo_xderivative_bound}.
This implies that there exists a constant $C_{\Sref} > 0$ depending on $\Uad$
such that it holds
\begin{equation} \label{ref_ql_pc1_apriori_bound}
	\| \Sref(u) \|_{\pdiff} \leq C_{\Sref} \quad
	\forall u \in \Uad.
\end{equation}
Note that $\Sref(u)$ takes values in $\Yval$ for all $u \in \Uad$ by
\cref{thm:grp_sol_physcoord}, since the space transformation does not change the values of the solution.
Moreover, by reducing the time $T$ and the constants $M_0, \, \varepsilon > 0$
in \eqref{def_Uad}, we may assume that the smallness assumptions
\eqref{smallness_ass_welldef_spacetrafo} hold for all solutions of
\eqref{problem_ref_ql} and all controls $(u_l, u_r, \xs) \in \Uad$:
This follows from the fact that the principal part of a solution of the GRP
coincides with the corresponding solution of the associated Riemann Problem
by \cref{thm:grp_sol_physcoord}. 
Noting \cref{rem:RP_soloperator_cont}, the piecewise constant parts of the
solution of a Riemann Problem depend continuously on its piecewise constant initial state.
Therefore, for fixed $M_1 > 0$, a reduction of $M_0, \varepsilon > 0$ makes
$|u_l(\xs) - u_L|, |u_r(\xs) - u_R| \leq M_0 + M_1 \varepsilon$ arbitrarily
small, and hence also the constant $c_{\max} > 0$ from the assumptions of
\cref{lem:welldefinedness_spacetrafo}.
A reduction of $T > 0$ then shows \eqref{smallness_ass_welldef_spacetrafo}.

To apply the results of the previous \cref{sec:slproblems}, 
the functions $G_i^j$ from \eqref{intbdary_cond_GRP}
need sufficiently small Lipschitz constants w.r.t.\ $(v^1, v^2)$
in the sense of \eqref{smallness_LF2}.
This can be achieved if $|u_L-u_R|, \, M_0, \, \varepsilon > 0$ in \eqref{def_Uad}
and $T > 0$ are chosen suitably small.
To make this rigorous, we use for $\varrho > 0$ and $j \in [n]$ the convex sets
\begin{equation} \label{def_ball_rho}
	\mathcal{B}^j_{\varrho} \coloneqq \big\{ (v^1, v^2) \in \Yval \times \Yval \, : \,
	|v^1 - \hat{y}_{j-1}| \leq \varrho, \ |v^2 - \hat{y}_j| \leq \varrho \big\}
\end{equation}
and abbreviate $\hat{\sigma} \coloneqq \max_{j \in [n]} | \hat{y}_{j-1} - \hat{y}_j |$
for the maximal jump of the entropy solution \eqref{sol_RP} of the Riemann Problem
\eqref{RP} associated with $(u_L, u_R)$ from \cref{ass:Uad1},
with the states $\hat{y}_j$ from \cref{def:fixed_shock_basis}.

\begin{lemma} \label{lem:GRP_bdarycondition_small}	
	There are $\varrho_0, C_G > 0$, depending only on $f$ and $\Yval$, with the
	following properties. If $\hat{\sigma} \leq \varrho_0$, then for all
	$\varrho \leq \varrho_0$, $j \in [n]$ and $i \neq j$ the function $G_i^j$
	from \eqref{intbdary_cond_GRP} is well defined on $\mathcal{B}^j_{\varrho}$
	with Lipschitz continuous first derivatives, and, with
	$D G_i^j$ denoting the derivative w.r.t.\ $(v, v^1, v^2)$, it holds that
	\begin{equation} \label{Gij_derivative_bounds}
		\partial_2 G_i^j \equiv 1, \qquad
		\| \partial_3 G_i^j \| + \| \partial_4 G_i^j \| \leq C_G ( \varrho + \hat{\sigma} ) ,
		\qquad \mathrm{Lip} ( D G_i^j ) \leq C_G .
	\end{equation}
	Moreover, if $T$ and $|u_L-u_R|, \, M_0, \, \varepsilon > 0$ in \eqref{def_Uad} and
	\eqref{Uad_physicalcoord} are sufficiently small, then $\hat{\sigma} \leq \varrho_0$
	and, with $\varrho_* \coloneqq c_{\max} + (1+\lambdamax) C_{\Sref} T \leq \varrho_0$
	and $\bar y = \Sref(u)$, it holds that
	\begin{equation} \label{traces_in_ball}
		\big( \bar y^{j-1}(t, s_j t) , \bar y^{j}(t, s_j t) \big) \in \mathcal{B}^j_{\varrho_*}
		\qquad \forall j \in [n], \ t \in [0,T], \ u \in \Uad .
	\end{equation}
	Since $\varrho_* + \hat{\sigma}$ becomes arbitrarily small
	in this way, we may in addition assume that
	\begin{equation} \label{rho_sigma_threshold}
		\varrho_* + \hat{\sigma} \leq ( 12 \, n (n+1) \, C_G )^{-1} .
	\end{equation}
	Then \eqref{lipschitz_Fij} and \eqref{def_theta_basisdeviation} hold for
	$F_i^j = G_i^j$ and $\bar{z} = \bar{y}$ with
	\begin{equation} \label{LF2_verified}
		L_{F_1} = 1, \qquad
		L_{F_2}, \, \theta \leq C_G ( \varrho_* + \hat{\sigma} ) , \qquad
		L_{\mathcal{F}_2} \leq 3 \, C_G ( \varrho_* + \hat{\sigma} )
		\leq ( 4n(n+1) )^{-1} ,
	\end{equation}
	i.e., the smallness assumption \eqref{smallness_LF2} is satisfied by $G_i^j$.
\end{lemma}
\begin{proof}
	Throughout the proof, $C_G > 0$ denotes a constant depending only on $f$ and
	$\Yval$, which is increased finitely often, and $\varrho_0 > 0$ is reduced
	finitely often accordingly.
	On $\mathcal{B}^j_{\varrho}$ it holds that $|v^1 - v^2| \leq 2\varrho + \hat{\sigma}
	\leq 3 \varrho_0$, so that $A(v^1,v^2)$ from \eqref{Aavg} is strictly hyperbolic
	and $l_i(v^1,v^2)$, $r_i(v^1,v^2)$ are well defined for $\varrho_0$ small.
	Since $l_i(\hat{y}_{j-1}, \hat{y}_j) = l^{*,j}_i$ and
	$( l^{*,j}_i )^\top r^{*,j}_k = \delta_{ik}$ by \cref{def:fixed_shock_basis},
	the Lipschitz continuity of $(y_1,y_2) \mapsto l_i(y_1,y_2)$ implies for
	$(v^1,v^2) \in \mathcal{B}^j_{\varrho}$ that
	\begin{equation} \label{aij_expansion}
		l_i ( v^1 , v^2 ) ^\top r^{*,j}_i = 1 + \mathcal{O}(\varrho) , \qquad
		l_i ( v^1 , v^2 ) ^\top r^{*,j}_k = \mathcal{O}(\varrho) \quad \forall k \neq i,
	\end{equation}
	uniformly in $i,j,k$.
	In particular, the denominators in \eqref{intbdary_cond_GRP} are bounded away
	from zero after possibly reducing $\varrho_0$, which together with the
	$C^{1,1}$-regularity of $(y_1,y_2) \mapsto l_i(y_1,y_2)$, following from
	\eqref{Aavg} and the assumptions on $A$, shows that $G_i^j$ is continuously
	differentiable on $\mathcal{B}^j_{\varrho}$ with $\mathrm{Lip} ( D G_i^j ) \leq C_G$.
	
	Writing the sum in \eqref{intbdary_cond_GRP} as
	$\sum_{k \neq i} q_{ik}(v^1,v^2) \, ( l^{*,j}_k )^\top (v^2 - v^1)$ with
	the abbreviations
	$q_{ik} \coloneqq l_i ( v^1 , v^2 ) ^\top r^{*,j}_k
	\big( l_i ( v^1 , v^2 ) ^\top r^{*,j}_i \big)^{-1}$, the product rule gives
	\begin{equation*}
		\partial_m \big( \sum_{k \neq i} q_{ik} \, ( l^{*,j}_k )^\top (v^2 - v^1) \big)
		= \sum_{k \neq i} \big( ( \partial_m q_{ik} ) \, ( l^{*,j}_k )^\top (v^2 - v^1)
		\mp q_{ik} \, l^{*,j}_k \big) , \quad m \in \{3,4\} ,
	\end{equation*}
	where the upper sign corresponds to $m=3$.
	It holds on $\mathcal{B}^j_{\varrho}$ that
	$| ( l^{*,j}_k )^\top (v^2 - v^1) | \leq C_G ( \varrho + \hat{\sigma} )$, while
	$|q_{ik}| \leq C_G \varrho$ for $k \neq i$ by \eqref{aij_expansion} and since the
	denominator in $q_{ik}$ is bounded away from zero, and $\partial_m q_{ik}$ is
	bounded. This proves \eqref{Gij_derivative_bounds}.
	
	It remains to verify \eqref{traces_in_ball} and that $\varrho_* + \hat{\sigma}$
	becomes arbitrarily small by reducing $T$ and $|u_L-u_R|, \, M_0, \, \varepsilon > 0$.
	For $\bar{y} = \Sref(u)$ it holds for all $j \in [n]$ and $t \in [0,T]$ that
	\begin{equation*}
		| \bar y^{j}(t, s_j t) - \hat y_j | \leq | \bar y^{j}(0,0) - \hat y_j |
		+ | \bar y^{j}(t,s_jt) - \bar y^{j}(0,0) |
		\leq c_{\max} + (1+\lambdamax) C_{\Sref} T = \varrho_*
	\end{equation*}
	with $c_{\max}$ from \cref{lem:welldefinedness_spacetrafo} and $C_{\Sref}$ from
	\eqref{ref_ql_pc1_apriori_bound}, and analogously for $\bar y^{j-1}$ and
	$\hat{y}_{j-1}$, which proves \eqref{traces_in_ball}.
	As shown above, $c_{\max}$ is arbitrarily small for $M_0, \, \varepsilon > 0$
	sufficiently small, and $C_{\Sref}$ does not increase when $T$ is reduced, so that
	$\varrho_*$ is arbitrarily small for $M_0, \, \varepsilon, \, T > 0$ sufficiently small.
	Since $\hat{y}_0 = u_L$, $\hat{y}_n = u_R$ and the intermediate states $\hat{y}_j$
	depend continuously on $(u_L,u_R)$ by \cref{rem:RP_soloperator_cont}, 
	$\hat{\sigma}$ is arbitrarily small for $|u_L - u_R|$ small. In particular,
	$\varrho_*, \hat{\sigma} \leq \varrho_0$ and \eqref{rho_sigma_threshold} can be achieved.
	
	We now show \eqref{LF2_verified} for $G_i^j$. By \eqref{Gij_derivative_bounds}
	applied with $\varrho = \varrho_*$, it holds that
	$L_{F_1} = 1$ and $L_{F_2} \leq C_G ( \varrho_* + \hat{\sigma} )$. For $\theta$,
	note that $l_i(\bar y^k) = l_i(\bar y^k, \bar y^k)$ and
	$l^{*,j}_i = l_i(\hat y_{j-1}, \hat y_j)$, and that \eqref{traces_in_ball} gives
	$| \bar y^k - \hat y_{j-1} | + | \bar y^k - \hat y_{j} | \leq 2 \varrho_* + \hat{\sigma}$
	on $\Sigma_j$ for $k \in \{j-1,j\}$. The Lipschitz continuity of
	$(y_1,y_2) \mapsto l_i(y_1,y_2)$ therefore yields
	$\theta \leq C_G ( \varrho_* + \hat{\sigma} )$.
	Since $L_{F_1} = 1$, it holds that $\sum_{k=0}^{n} (L_{F_1})^k = n+1$ and
	$L_{\mathcal{F}_2} = L_{F_2} + 2 \theta \leq 3 \, C_G ( \varrho_* + \hat{\sigma} )$,
	so that \eqref{rho_sigma_threshold} implies
	$L_{\mathcal{F}_2} \leq \frac{1}{4n(n+1)}
	= \frac{1}{4n} \big( \sum_{k=0}^{n} (L_{F_1})^k \big)^{-1}$.
	Finally, note that only the traces \eqref{traces_in_ball} are inserted into $G_i^j$,
	so that we may w.l.o.g.\ assume \eqref{lipschitz_Fij} to hold with these constants
	on all of $\R^n \times \R^n$.
\end{proof}

\begin{assumption} \label{ass:constants_small}
	We further reduce $T$ and $|u_L-u_R|, \, M_0, \, \varepsilon > 0$ in
	\eqref{def_Uad}, \eqref{Uad_physicalcoord} such that \eqref{traces_in_ball} and
	\eqref{rho_sigma_threshold} hold, and that \eqref{cij_nondeg} and
	\eqref{smallness_LF2_pc1} are satisfied for $F_i^j = G_i^j$ and
	$\bar{z} = \Sref(u)$ for all $u \in \Uad$. 
	This is possible by \cref{lem:GRP_bdarycondition_small} and \eqref{LF2_verified}.
	No further reduction of these constants is required in the sequel.	
\end{assumption}

For our further considerations, we set
\begin{equation} \label{def_set_calM}
	\mathcal{M} \coloneqq \{ \y \in \pcont \, : \, \mathrm{img}(\y) \subset \Yval \} ,
\end{equation}
so that $\Sref(u) \in \mathcal{M}$ for all $u \in \Uad$.
The source term $\bar{g}^{(\bar{y}, \xs)}$ and the matrix $\bar{A}^{(\bar{y}, \xs)}$
in \eqref{problem_ref_ql} satisfy the following differentiability properties.

\begin{lemma} \label{lem:GRP_sourceterm_diff}	
The source term $\bar{g}^{(\bar{y}, \xs)}$ from \eqref{sourceterm_GRP_functionalform} is
continuously Fr\'echet differentiable as a map 
$(\bar{y}, \xs) \in \mathcal{M} \times (-\varepsilon, \varepsilon)
\mapsto \bar{g}^{(\bar{y}, \xs)} \in \pcont$.
Moreover, the derivative $ d_{(\bar{y}, \xs)} \bar{g}^{(\bar{y}, \xs)}
\in \mathcal{L}(\pcont \times \R, \pcont)$
is uniformly bounded on $\mathcal{M} \times (-\varepsilon, \varepsilon)$.
\end{lemma}
\begin{proof}
Let $(\delta \y, \delta \xs) \in \pcont \times \R$.
For $j \in [n]_0$ and $(t, \x) \in \dom_j$ it holds that
\begin{equation*} 
\big( d_{(\y, \xs)} \bar{g}^{(\bar{y}, \xs)} \cdot (\delta \y, \delta \xs) \big)(t, \x)
= g_x \cdot \big( d_{(\y, \xs)} x^{(\bar{y}, \xs),j} \cdot (\delta \y, \delta \xs) \big)(t, \x)  + g_y \cdot \delta \y(t, \x)
\end{equation*}
where $g_x$ and $g_y$ are evaluated in $(t, x^{(\bar{y}, \xs),j}(t,\x), \bar{y}^j(t, \x))$.
With the differentiability properties \eqref{trafo_Fderivative}
of the space transformation,
it is easy to verify the pointwise estimate
\begin{equation*}
\big| \bar{g}^{(\bar{y} + \delta \y,\xs + \delta \xs)} 
- \bar{g}^{(\bar{y},\xs)} 
- \big( d_{(\y, \xs)} \bar{g}^{(\bar{y}, \xs)} \cdot (\delta \y, \delta \xs) \big) \big| (t, \x) 
= o( \| \delta \y \|_{\pcont} + | \delta \xs| ).
\end{equation*}
Since $g \in C^{1,1}(\D \times \Yval; \R^n)$,
the above estimate also holds uniformly w.r.t.\ $j \in [n]_0$ and all $(t, \x) \in \dom_j$. 
This proves the differentiability in the topology of $\pcont$.

Finally, $ d_{(\bar{y}, \xs)} \bar{g}^{(\bar{y}, \xs)} \in 
\mathcal{L}(\pcont \times \R, \pcont)$ is bounded 
uniformly for all $(\bar{y}, \xs) \in \mathcal{M} \times (-\varepsilon, \varepsilon)$,
since $g_x, g_y$ are bounded on
$\D \times \Yval$ and $d_{(\y, \xs)} x^{(\bar{y}, \xs),j}$
is bounded uniformly w.r.t.\ $(\bar{y}, \xs) \in \mathcal{M} \times (-\varepsilon, \varepsilon)$
due to \eqref{x_trafo_middle_Fdiff_normbound}.
\end{proof}

\begin{lemma} \label{lem:GRP_matrix_diff}
The map $(\bar{y}, \xs) \in \mathcal{M} \times (-\varepsilon, \varepsilon)
\mapsto \bar{A}^{(\bar{y}, \xs)} \in PC^0(\dom; \R^{n \times n})$
is continuously Fr\'echet differentiable and
$ d_{(\bar{y}, \xs)} \bar{A}^{(\bar{y}, \xs)} \in 
\mathcal{L}(\pcont \times \R, PC^0(\dom; \R^{n \times n}) )$
is uniformly bounded on $\mathcal{M} \times (-\varepsilon, \varepsilon)$.
\end{lemma}
\begin{proof}
The assertion follows analogously to the previous proof by using the definition
\eqref{A_bar_z} of $\bar{A}^{(\bar{y}, \xs)}$, the bound \eqref{spacetrafo_xderivative_bound},
and observing that \eqref{x_trafo_middle_Fdiff_normbound} is uniform
for all $(\bar{y}, \xs) \in \mathcal{M} \times (-\varepsilon, \varepsilon)$.
\end{proof}

\begin{theorem} \label{thm:GRP_lipschitz_pcont}	
There is a constant $L_{\Sref} > 0$ such that for all
$u=(u_l, u_r, \xs), \allowbreak 
\hat{u}=(\hat{u}_l, \hat{u}_r, \hat{x}_0) \in \Uad$
with $\Uad$ from \eqref{Uad_physicalcoord} it holds that
\begin{equation} \label{ql_lipschitz}
\| \Sref(\hat{u}) - \Sref(u) \|_{\pcont} \leq L_{\Sref} 
( \|u_l - \hat{u}_l \|_{C^0([-\el,\varepsilon])}
+ \|u_r - \hat{u}_r \|_{C^0([-\varepsilon,\el])}
+ |\xs - \hat{x}_0| ).
\end{equation}
\end{theorem}
\begin{proof}
Denote $\y = \Sref(u), \ \tilde{y} = \Sref(\hat{u})$. 
Then $\bar{v} \coloneqq \tilde{y} - \y \in \pdiff$ solves
\begin{equation} \label{ql_difference}
\bar{v}_t + \bar{A}^{(\bar{y}, \xs)} \bar{v}_{\x} = \bar{h}(\bar{v}) 
\quad \text{on } \dom, \quad \bar{v}^0(0, \cdot) = \bar{v}_l, 
\quad \bar{v}^n(0, \cdot) =  \bar{v}_r,
\end{equation}
with $\bar{v}_l = \mathbf{D}_l(\hat{u}_l, \hat{x}_0) - \mathbf{D}_l(u_l, \xs)$,
$\bar{v}_r = \mathbf{D}_r(\hat{u}_r, \hat{x}_0) - \mathbf{D}_r(u_r, \xs)$. 
By \eqref{def_Uad} and \eqref{trafo_initstate},
\begin{equation} \label{init_state_refspace_c0lipschitz}
|\bar{v}_l(\x)| \leq \| u_l - \hat{u}_l \|_{C^0([-\el,\varepsilon])} 
+ M_1|\xs - \hat{x}_0|,
\end{equation}
and analogously for $\bar{v}_r$.
The source term in functional form is given by 
\begin{equation} \label{barh_sourceterm}
	\begin{aligned}
		\bar{h}(\bar{v})  & =  \int_0^1 d_{(\bar{y}, \xs)} 
		\bar{g}^{\left( s (\bar{y}, \xs) + (1-s) (\tilde{y}, \hat{x}_0) \right)} 
		\ds \cdot (\bar{v}, \hat{x}_0 - \xs) \\
		& - \Big( \int_0^1 d_{(\bar{y}, \xs)} 
		\bar{A}^{\left( s (\bar{y}, \xs) + (1-s) (\tilde{y}, \hat{x}_0) \right)} 
		\ds \cdot (\bar{v}, \hat{x}_0 - \xs) \Big) \tilde{y}_{\x} ,
	\end{aligned}
\end{equation}
with the Fr\'echet derivatives $d_{(\y, \xs)} \bar{g}^{(\bar{y}, \xs)}$
from \cref{lem:GRP_sourceterm_diff} and 
$d_{(\y, \xs)} \bar{A}^{(\bar{y}, \xs)}$ from \cref{lem:GRP_matrix_diff}.
Note that the two Bochner integrals above are in fact well-defined due to the continuous
differentiability of $\bar{g}^{(\bar{y}, \xs)}$ and $\bar{A}^{(\bar{y}, \xs)}$ w.r.t.\ $\y$ and $\xs$.
To bound the Lipschitz constant of $\bar{h}(\bar{v})$
w.r.t.\ $\bar{v} \in \pcont$, we exploit that
$\y, \tilde{y} \in \mathcal{M}$ with the set $\mathcal{M}$ from \eqref{def_set_calM}.
By \cref{lem:GRP_sourceterm_diff}, we have 
$ d_{(\bar{y}, \xs)} \bar{g}^{(\bar{y}, \xs)}
\in \mathcal{L}(\pcont \times \R, \pcont)$
is uniformly bounded for all $(\bar{y}, \xs) \in \mathcal{M} \times (-\varepsilon, \varepsilon)$.
With the same argument and using \cref{lem:GRP_matrix_diff},
we can also find an a-priori bound for $d_{(\y, \xs)} \bar{A}^{(\bar{y}, \xs)}$.
Together, this shows that $\bar{h}(\bar{v})$ is Lipschitz continuous
w.r.t.\ $\bar{v} \in \pcont$ in the sense of \eqref{lipschitz_sl_sourceterm},
where the Lipschitz constant can be a-priori bounded.

The interior boundary constraints \eqref{G_ij_coupling}
imply on $\Sigma_j$ for all $j \in [n]$ and $i > j$
\begin{equation} \label{Hij_explicit}
	\begin{aligned}
		\bar{v}^j_{*,i} & = ( l^{*,j}_i ) ^\top \tilde{y}^j - ( l^{*,j}_i ) ^\top \bar{y}^j
		= G_i^j (t, \tilde{y}^{j-1}_{*,i}, \tilde{y}^{j-1}, \tilde{y}^{j} ) 
		- G_i^j (t, \bar{y}^{j-1}_{*,i}, \bar{y}^{j-1}, \bar{y}^{j} ) \\
		& = \int_0^1 D G_i^j \ds 
		\cdot ( \bar{v}^{j-1}_{*,i}, \bar{v}^{j-1}, \bar{v}^{j} ) 
		\eqqcolon H_i^j(t, \bar{v}^{j-1}_{*,i}, \bar{v}^{j-1}, \bar{v}^{j}),
	\end{aligned}
\end{equation}
where $G_i^j = G_i^j \big(t, s \tilde{y}^{j-1}_{*,i} + (1-s) \bar{y}^{j-1}_{*,i}, 
s \tilde{y}^{j-1} + (1-s) \bar{y}^{j-1}, 
s \tilde{y}^{j} + (1-s) \bar{y}^{j} \big)$
and $D G_i^j$ is the derivative w.r.t.\ $(v, v^1, v^2)$ as in \eqref{Gij_derivative_bounds}.
The derivation \eqref{Hij_explicit} is analogous for $i < j$.
The arguments of $G_i^j$ in \eqref{Hij_explicit} are convex combinations of the
traces of $\y$ and $\tilde{y}$ on $\Sigma_j$, which lie in the convex set
$\mathcal{B}^j_{\varrho_*}$ from \eqref{def_ball_rho} by \eqref{traces_in_ball}.
Hence \eqref{Gij_derivative_bounds} applies and shows that $H_i^j$ satisfies
\eqref{lipschitz_Fij} with the same constants \eqref{LF2_verified} as $G_i^j$,
which satisfy \eqref{smallness_LF2} by \cref{ass:constants_small}.

Since $\bar{h}(\bar{0})$ is given by \eqref{barh_sourceterm}
evaluated at $(\bar{v}, \hat{x}_0 - \xs) = (0, \hat{x}_0 - \xs)$, 
the bounds from \cref{lem:GRP_sourceterm_diff}
and \cref{lem:GRP_matrix_diff} and the a-priori bound
\eqref{ref_ql_pc1_apriori_bound} for $\tilde{y}_{\x}$ show that
\begin{equation} \label{hbar_zero_bound}
	\| \bar{h}(\bar{0}) \|_{\pcont} \leq C | \xs - \hat{x}_0 |
\end{equation}
with a constant $C>0$ independent of $u, \hat{u} \in \Uad$.
We now apply \cref{thm:sl_pc0_bound} to \eqref{ql_difference}
with $H_i^j(\cdot, 0,0,0) = 0$, 
the bound \eqref{hbar_zero_bound}
and \eqref{init_state_refspace_c0lipschitz}.
This yields \eqref{ql_lipschitz}.
\end{proof}

\begin{theorem} \label{thm:GRP_cont_pdiff}	
The solution operator $\Sref : \Uad \to \pdiff$ is continuous.
\end{theorem}
\begin{proof}
Let $u = (u_l, u_r, \xs) \in \Uad$ and $\y = \Sref(u)$.
Further, let $(u^\nu)_\nu \subset \Uad$ with $u^\nu = (u_l^\nu, u_r^\nu, \xs^\nu)$
be a sequence with $u^\nu \to u$ in $\Uad$
for $\nu \to \infty$. Denote $\y^{[\nu]} = \Sref(u^\nu)$.
By \cref{ass:constants_small}, the assumptions of \cref{thm:c1_sol_sl} are
satisfied for $F_i^j = G_i^j$ and $\bar{z} = \y$, so that $\bar{y}_{\x}$
is a broad solution of
\begin{equation} \label{yx_pde_a1}
(\bar{y}_{\x})_{t} + \bar{A}^{(\bar{y}, \xs)} (\bar{y}_{\x})_{\bar{x}}
= \partial_{\x} \bar{g}^{(\bar{y}, \xs)} - (\partial_{\x} \bar{A}^{(\bar{y}, \xs)}) \bar{y}_{\x} 
\end{equation}
with interior boundary conditions \eqref{sl_yx_nodecond} with $\bar{z} = \y$.
Using \eqref{A_bar_z} and \eqref{h_bar_z} yields
\begin{equation} \label{yx_sourcterm_derivatives}
\begin{aligned}
\partial_{\x} \bar{g}^{(\bar{y}, \xs)} = g_x x^{(\bar{y}, \xs)}_{\x} + g_y \bar{y}_{\x}, \quad
\partial_{\x} \bar{A}^{(\bar{y}, \xs)} = \big( A'(\y) \y_{\x} - x_{t \x}^{(\y, \xs)} \big) 
\big( x_{\x}^{(\y, \xs)} \big)^{-1},
\end{aligned}
\end{equation}
since $x_{\x}^{(\y, \xs)}$ is independent of $\x$, see \eqref{x_trafo_middle_dx}. 
Differentiating \eqref{x_trafo_middle_dx} w.r.t.\ $t$ yields
\begin{equation} \label{spacetrafo_mixed_second_D}
x_{t \x}^{(\bar{y}, \xs)} (t, \x) = 
\big( t ( \dot{\xi}^{(\bar{y}, \xs)}_{j+1} - \dot{\xi}^{(\bar{y}, \xs)}_{j} )
- ( \xi^{(\bar{y}, \xs)}_{j+1} - \xi^{(\bar{y}, \xs)}_{j} ) \big)
t^{-2} (s_{j+1}-s_j)^{-1}
\end{equation}
for all $j \in [n-1]$ and $(t, \x) \in \dom_j$. 
Using $\xi^{(\bar{y}, \xs)}_{j} (t) - t \dot{\xi}^{(\bar{y}, \xs)}_{j}(t) 
= \xs - \int_0^t \ddot{\xi}^{(\bar{y}, \xs)}_{j}(\tau) \tau \intd \tau$ 
and \eqref{xij_curve_shock} yields 
\begin{equation} \label{sourceterm_yx_seconddiff_xij}
\xi_{j} (t) - t \dot{\xi}_{j}(t)
= \xs - \int_0^t \nabla \lambda_j(\y^{j-1}, \y^j)  
\big( (s_j - \bar{A}^{(\bar{y}^k, \xs)}) \y^k_{\x} + \bar{g}^{(\bar{y}^k, \xs)} \big)_{k=j-1,j}
(\tau, s_j \tau) \tau \intd \tau
\end{equation}
which depends on the values of $\y^{j-1}_{\x}$ and $\y^{j}_{\x}$ on $\Sigma_j$ on $[0,t]$.
Note that $\xs$ cancels in the difference appearing in the numerator of
\eqref{spacetrafo_mixed_second_D}.
Combining this with \eqref{yx_sourcterm_derivatives} implies
that the source term in \eqref{yx_pde_a1} can be written as
\begin{equation*}
\partial_{\x} \bar{g}^{(\bar{y}, \xs)} 
- (\partial_{\x} \bar{A}^{(\bar{y}, \xs)}) \bar{y}_{\x} =
\bar{d}( \bar{y}_{\x} )
\end{equation*}
with $\bar{d} : \pcont \to \pcont$ given
with $g_x, g_y$ evaluated at $( t, x^{(\bar{y}, \xs)}(t,\x), \y(t,\x) )$ by
\begin{equation} \label{sourceterm_bard}
	\begin{aligned}
		\bar{d}( \bar{v} )^j (t, \x) 
		& = g_x \, x^{(\bar{y}, \xs),j}_{\x}(t) + g_y \, \bar{v}^j(t, \x) \\
		& - \big( A'( \y^j(t,\x) ) \bar{v}^j(t,\x) 
		- x^{(\bar{y}, \xs),j}_{t \x}[\bar{v}](t) \big)
		\big( x^{(\bar{y}, \xs),j}_{\x}(t) \big)^{-1} \bar{v}^j(t,\x),
	\end{aligned}
\end{equation}
where we have in view of \eqref{spacetrafo_mixed_second_D} 
and \eqref{sourceterm_yx_seconddiff_xij} that
\begin{equation} \label{sourceterm_bard_notation}
	\begin{aligned}
		x^{(\bar{y}, \xs),j}_{t \x}[\bar{v}](t) & = 
		\frac{1}{(s_{j+1}-s_j) t^{2}} \int_0^t 
		\big( \ddot{\xi}_{j+1}[\bar{v}] - \ddot{\xi}_{j}[\bar{v}] \big)(\tau) \, \tau \intd \tau ,\\
		\ddot{\xi}_{j}[\bar{v}](\tau) & = 
		\nabla \lambda_j( \y^{j-1}, \y^{j} ) \cdot
		\big( ( s_j - \bar{A}^{(\bar{y}^k, \xs)} ) \bar{v}^{k} 
		+ \bar{g}^{(\bar{y}^k, \xs)} \big)_{k=j-1,j} (\tau, s_j \tau) .
	\end{aligned}
\end{equation}
Note that $\bar{d}$ on $\dom^t$ only depends on $\bar v$ on $\dom^t$ due to
\eqref{sourceterm_bard_notation}.
Although $\bar{d}$ is quadratic w.r.t.\ $\bar{v}$ and thus only locally Lipschitz
continuous, we may replace $\bar{v}$ in \eqref{sourceterm_bard},
\eqref{sourceterm_bard_notation} by $\rho \circ \bar{v}$, where
$\rho(w) = w$ for $|w| \leq C_{\Sref}$ and $\rho(w) = C_{\Sref} w / |w|$ otherwise
with $C_{\Sref}$ from \eqref{ref_ql_pc1_apriori_bound}.
Since $\rho$ is Lipschitz continuous and acts only pointwise, this yields
\eqref{lipschitz_sl_sourceterm} without changing the solutions.
Moreover, noting \cref{rem:nodecond_yx_slassumptions} and \cref{ass:constants_small},
the interior boundary conditions of $\y_{\x}$ satisfy the assumptions of
\cref{thm:existence_broadsol_sl}.
Thus, $\y_{\x}$ is the unique broad solution of
\begin{equation} \label{yx_pde_a2}
\bar{v}_{t} + \bar{A}^{(\bar{y}, \xs)} \bar{v}_{\bar{x}}
= \bar{d}(\bar{v}) \quad \text{on } \dom, 
\quad \bar{v}^0(0, \cdot) = \mathbf{D}_l(u_l, \xs)', 
\ \bar{v}^n(0, \cdot) = \mathbf{D}_r(u_r, \xs)'.
\end{equation}
Replacing $(\y, \xs)$ by $(\y^{[\nu]}, \xs^\nu)$ in \eqref{sourceterm_bard},
\eqref{sourceterm_bard_notation} defines $\bar{d}^\nu$, 
and $\y^{[\nu]}_{\x}$ is the unique broad solution of
\begin{equation} \label{yx_pde_a3}
\bar{v}_{t} + \bar{A}^{(\y^{[\nu]}, \xs^\nu)} \bar{v}_{\bar{x}}
= \bar{d}^{\nu}(\bar{v}) \text{ on } \dom, 
\quad \bar{v}^0(0, \cdot) = \mathbf{D}_l(u_l^\nu, \xs^\nu)', 
\ \bar{v}^n(0, \cdot) = \mathbf{D}_r(u_r^\nu, \xs^\nu)' \!\!\!
\end{equation}
and interior boundary condition \eqref{sl_yx_nodecond}
with $\y$ and $\bar{z}$ replaced by $\y^{[\nu]}$.

Now observe that $u^\nu \to u$ in $\Uad$ 
implies with the definition \eqref{trafo_initstate} 
of $\mathbf{D}_l, \, \mathbf{D}_r$
that $\mathbf{D}_l(u_l^\nu, \xs^\nu)' \to \mathbf{D}_l(u_l, \xs)'$ and 
$\mathbf{D}_r(u_r^\nu, \xs^\nu)' \to \mathbf{D}_r(u_r, \xs)'$ uniformly on
$[-\el, 0]$ and $[0, \el]$ for $\nu \to \infty$,
since $u_l'$ and $u_r'$ are uniformly continuous.

\Cref{thm:GRP_lipschitz_pcont} yields $\y^{[\nu]} \to \y$ in $\pcont$
and thus $\bar{d}^{\nu}(\bar{w}) \to \bar{d}(\bar{w})$ in $\pcont$ for
all $\bar{w} \in \pcont$ for $\nu \to \infty$ by
\eqref{sourceterm_bard} and \eqref{sourceterm_bard_notation}.
Since \eqref{sl_yx_nodecond} depends on $\bar{z} = \y$ only through
$\bar{A}^{(\bar{y}, \xs)}$ and the values of $\y$ on $\Sigma_j$, but not on
derivatives of $\y$, the interior boundary conditions of \eqref{yx_pde_a2},
\eqref{yx_pde_a3} converge locally uniformly by \cref{lem:GRP_matrix_diff}
and \cref{thm:GRP_lipschitz_pcont}.
Together with \eqref{ref_ql_pc1_apriori_bound}, which bounds $(\y^{[\nu]})_\nu$
in $\pdiff$ uniformly in $\nu$, the assumptions of
\cref{thm:sl_continuous_dep_data} are satisfied, where
\eqref{def_theta_basisdeviation} holds uniformly in $\nu$ by \eqref{LF2_verified}
and \cref{ass:constants_small}.
Then applying \cref{thm:sl_continuous_dep_data} yields 
$\y^{[\nu]}_{\x} \to \y_{\x}$ in $\pcont$ for $\nu \to \infty$.
By rearranging the PDE, this also holds for the time derivatives
$\y^{[\nu]}_{t} \to \y_{t}$ in $\pcont$ for $\nu \to \infty$.
\end{proof}

\section{Differentiability Properties} \label{sec:diff}
In this section, we derive differentiability properties of the solution operator
of the GRP both in the reference space and in physical coordinates
w.r.t.\ the initial state \eqref{initial_data_GRP}
with $u = (u_l, u_r, \xs) \in \Uad$ and $\Uad$ from \eqref{Uad_physicalcoord}.
This will also imply the differentiability
of the objective functional \eqref{objective_fcntal}.

\subsection{Differentiability of the Solution Operator in the Reference Space}
\label{sec:diffref}
We prove the differentiability of the solution of the GRP in the reference space
w.r.t.\ the piecewise $C^1$ initial states and the position of the discontinuity.
Let
\begin{equation*}
\U = C^1([-\el,\varepsilon]; \R^n) \times C^1([-\varepsilon,\el]; \R^n) \times \R
\end{equation*}
be equipped with the induced norm denoted by $\| \cdot \|_{\U}$.

\begin{lemma} \label{lem:deltay_properties}
	Let $u \in \Uad$ and $\y = \Sref(u)$. For all $\delta u = (\delta u_l, \delta u_r, \delta \xs) \in \U$,
	there exists a unique broad solution $\delta \y \in \pcont$ of
	\begin{align} 
		\delta \y_t + \bar{A}^{(\bar{y}, \xs)} \delta \y_{\x}
		& = \bar{a}(\delta \y) \coloneqq d_{(\bar{y}, \xs)} \bar{g}^{(\bar{y},\xs)} \cdot (\delta \y , \delta \xs)
		- (d_{(\bar{y}, \xs)} \bar{A}^{(\bar{y}, \xs)} \cdot (\delta \y , \delta \xs) ) \y_{\x} , \notag \\
		\delta \y^0(0, \cdot) &= \delta \bar{u}_l , \quad \delta \bar{u}_l(\x) 
		= \mathbf{D}_l(\delta u_l, \xs)(\x)+\mathbf{D}_l(u_l', \xs)(\x) 
		\frac{\x+\el}{\el} \delta \xs \label{pde_dy} , \\
		\quad \delta \y^n(0, \cdot) &= \delta \bar{u}_r , \quad \delta \bar{u}_r(\x) 
		= \mathbf{D}_r(\delta u_r, \xs)(\x)-\mathbf{D}_r(u_r', \xs)(\x) \frac{\x-\el}{\el} \delta \xs \notag ,
	\end{align}
	with interior boundary conditions for all $j \in [n]$, $i \neq j$, and all $t \in [0,T]$
	defined by
	\begin{equation} \label{B_ij_coupling}
		\begin{aligned}
			\delta \bar{y}^{j}_{*,i} (t, s_j t) &= B_i^j \big( t, \delta \bar{y}^{j-1}_{*,i} (t, s_j t),
			\delta \bar{y}^{j-1} (t, s_j t), \delta \bar{y}^{j} (t, s_j t)  \big) && i>j, \\
			\delta \bar{y}^{j-1}_{*,i} (t, s_j t) &= B_i^j \big( t, \delta \bar{y}^{j}_{*,i} (t, s_j t),
			\delta \bar{y}^{j-1} (t, s_j t), \delta \bar{y}^{j} (t, s_j t) \big) && i<j.
		\end{aligned}
	\end{equation}
	Here, $\delta \bar{y}^{k}_{*,i} = ( l^{*,j}_i ) ^\top \delta \y^{k}$ for $k \in \{ j-1, j \}$,
	and $B_i^j$ are given with $G_i^j$ from \eqref{intbdary_cond_GRP} by 
	\begin{equation} \label{sl_dy_nodecond}
		B_i^j(t, v, v^1, v^2) = D G_i^j (t, \y_{*,i}^{j-1}, \y^{j-1}, \y^{j})
		\cdot ( v, v^1, v^2 )
	\end{equation}
	for $i>j$, where $D G_i^j$ is as in \eqref{Gij_derivative_bounds} and all
	quantities are evaluated in $(t, s_j t)$. The case $i<j$ is analogous.
	Moreover, \eqref{pde_dy} and \eqref{sl_dy_nodecond} define a bounded linear operator 
	$ (\delta u \mapsto \delta \y) \in \mathcal{L}(\U, \pcont)$
	which is uniformly bounded for all $u \in \Uad$, i.e.,
	there is $C > 0$ such that $\| \delta \y \|_{\pcont} \leq C \| \delta u \|_{\U}$
	holds independently of $u \in \Uad$.
\end{lemma}
\begin{proof}
	First note that $B_i^j$ from \eqref{sl_dy_nodecond} is linear and, since the traces
	of $\y$ on $\Sigma_j$ lie in $\mathcal{B}^j_{\varrho_*}$ by \eqref{traces_in_ball},
	satisfies \eqref{lipschitz_Fij} with the same constants \eqref{LF2_verified}
	as $G_i^j$. In particular, no further reduction of the constants beyond
	\cref{ass:constants_small} is needed, and \eqref{smallness_LF2} holds.
	Hence \eqref{pde_dy}, \eqref{sl_dy_nodecond} satisfy the assumptions of
	\cref{thm:existence_broadsol_sl},
	implying the unique existence of a solution $\delta \y \in \pcont$.
	Further, $\delta \y$ is linear w.r.t.\ $\delta u$, since
	the initial states $\delta \bar{u}_l, \delta \bar{u}_r$
	and the source term $\bar{a}$ in \eqref{pde_dy} are linear w.r.t.\ $\delta u$
	and \eqref{sl_dy_nodecond} is linear w.r.t.\ $(v,v^1,v^2)$.
	Since $\| \delta \bar{u}_l \|_{C^0([-\el, 0]; \R^n)} \leq \| 
	\delta u_l \| _{C^0([-\el, \varepsilon]; \R^n)} + M_1 | \delta \xs |$ and 
	analogously for $\delta \bar{u}_r$,
	\cref{thm:sl_pc0_bound} implies $\| \delta \y \|_{\pcont} \leq C \| \delta u \|_{\U}$
	and thus \eqref{pde_dy}, \eqref{sl_dy_nodecond} define 
	a bounded linear operator 
	$\delta u = (\delta u_l, \delta u_r, \delta \xs) \in \U \mapsto \delta \y \in \pcont$.
	Moreover, noting the uniform boundedness of $d_{(\bar{y}, \xs)} \bar{g}^{(\y,\xs)}$ 
	and $d_{(\bar{y}, \xs)} \bar{A}^{(\y, \xs)}$
	from \cref{lem:GRP_sourceterm_diff} and \cref{lem:GRP_matrix_diff},
	the constant $C$ can be chosen independently of $u \in \Uad$.
\end{proof}

\begin{theorem} \label{thm:differentiability_refspace}
The operator $\Sref$ is continuously Fr\'echet differentiable as a map
\begin{equation} \label{Sref_diffability}
\Sref: \Uad \to (\pdiff, \| \cdot \|_{\pcont}) .
\end{equation}
For any $\delta u = (\delta u_l, \delta u_r, \delta \xs) \in \U$,
the sensitivity $\delta \y = d_u \Sref(u) \cdot \delta u \in \pcont$
is the unique broad solution of \eqref{pde_dy}, \eqref{sl_dy_nodecond}.
\end{theorem}
\begin{proof}
Fix $u \in \Uad$ and $\y = \Sref(u)$. 
Let $\delta u \in \U$ be sufficiently small such that
$\hat{u} \coloneqq u + \delta u \in \Uad$
and denote $\tilde{y} = \Sref(\hat{u})$. 
Moreover, let $\delta \y \in \pcont$ be the solution of \eqref{pde_dy}, \eqref{sl_dy_nodecond}.
Then $\w \coloneqq \tilde{y} - \y - \delta \y \in \pcont$
is the broad solution of
\begin{equation} \label{pde_difference_dy}
\w_t + \bar{A}^{(\bar{y}, \xs)} \w_{\x}
= \bar{q}(\w) \quad \text{on } \dom, 
\quad \w^0(0, \cdot) = \w_l, 
\quad \w^n(0, \cdot) = \w_r,
\end{equation}
with the initial states 
$ \w_l = \mathbf{D}_l(u_l + \delta u_l, \xs + \delta \xs) 
- \mathbf{D}_l(u_l, \xs) - \delta \bar{u}_l$
and analogously for $\w_r$.
The source term in functional form is given by
\begin{equation} \label{pde_difference_sourceterm}
\begin{aligned}
\bar{q}(\w) &= d_{\bar{y}} \bar{g}^{(\bar{y}, \xs)} \cdot \w
- ( d_{\bar{y}} \bar{A}^{(\bar{y}, \xs)} \cdot \w) \bar{y}_{\x} \\
& + \bar{g}^{(\tilde{y}, \xs + \delta \xs)} - \bar{g}^{(\bar{y}, \xs)}
- d_{(\bar{y}, \xs)} \bar{g}^{(\bar{y}, \xs)} \cdot (\tilde{y} - \bar{y}, \delta \xs) \\ 
& + (\bar{A}^{(\bar{y}, \xs)} - \bar{A}^{(\tilde{y}, \xs + \delta \xs)})
( \tilde{y}_{\x} - \bar{y}_{\x} ) \\
& + \big( \bar{A}^{(\bar{y}, \xs)} 
+ d_{(\bar{y}, \xs)} \bar{A}^{(\bar{y}, \xs)}
\cdot (\tilde{y} - \bar{y}, \delta \xs) 
- \bar{A}^{(\tilde{y}, \xs + \delta \xs)} \big) \bar{y}_{\x} .
\end{aligned}
\end{equation}
With $\bar{w}^{k}_{*,i} = ( l^{*,j}_i )^\top \bar{w}^{k}$ for $k \in \{ j-1, j \}$,
the interior boundary conditions read
\begin{equation} \label{Q_ij_coupling}
	\begin{aligned}
		\bar{w}^{j}_{*,i} (t, s_j t) &= Q_i^j ( t, \bar{w}^{j-1}_{*,i} (t, s_j t),
		\bar{w}^{j-1} (t, s_j t), \bar{w}^{j} (t, s_j t) ) && i>j, \\
		\bar{w}^{j-1}_{*,i} (t, s_j t) &= Q_i^j ( t, \bar{w}^{j}_{*,i} (t, s_j t),
		\bar{w}^{j-1} (t, s_j t), \bar{w}^{j} (t, s_j t) ) && i<j.
	\end{aligned}
\end{equation}
The functions $Q_i^j$ for $i > j$ (the case $i<j$ is analogous) are given by
\begin{equation} \label{pde_difference_nodecond}
	\begin{aligned}
		Q_i^j(t, v, v^1, v^2) & = B_i^j(t, v, v^1, v^2) + G_i^j(\tilde{y}) - G_i^j(\y) \\
		& - D G_i^j(\y) \cdot ( \tilde{y}^{j-1}_{*,i} - \y^{j-1}_{*,i} ,
		\tilde{y}^{j-1} - \y^{j-1} , \tilde{y}^{j} - \y^{j} )
	\end{aligned}
\end{equation}
with $B_i^j$ from \eqref{sl_dy_nodecond}, the abbreviation
$G_i^j(\y) = G_i^j(t, \y_{*,i}^{j-1}, \y^{j-1}, \y^{j})$ with $G_i^j$ in \eqref{intbdary_cond_GRP}
and similarly for
$G_i^j(\tilde{y})$, all quantities evaluated in $(t, s_j t)$.
In particular, $Q_i^j$ and $B_i^j$ have the same
Lipschitz constants w.r.t.\ $(v, v^1, v^2)$
and \eqref{smallness_LF2} holds for $Q_i^j$.
Next, we apply \cref{thm:sl_pc0_bound} to
\eqref{pde_difference_dy}, \eqref{Q_ij_coupling}.
For this, we observe with \eqref{problem_ref_ql} and \eqref{pde_dy} that
\begin{equation*}
\begin{aligned}
\w_l(\x) &= \mathbf{D}_l(u_l, \xs + \delta \xs)(\x) - \mathbf{D}_l(u_l , \xs) (\x)
- \mathbf{D}_l(u_l' , \xs)(\x) \frac{\x+\el}{\el} \delta \xs \\
& + \mathbf{D}_l(\delta u_l, \xs + \delta \xs)(\x) - \mathbf{D}_l(\delta u_l, \xs)(\x)
\end{aligned}
\end{equation*}
for all $\x \in [-\el,0]$, where we used the linearity of $\mathbf{D}_l$
w.r.t.\ its first argument, see its definition \eqref{def_dilation}.
The first line is a first-order Taylor remainder in
$\xs$ and is $o(|\delta \xs|)$ uniformly in $\x$ since $u_l'$ is uniformly
continuous on $[-\el, \varepsilon]$, and the second line is bounded by
$\| \delta u_l' \| \, |\delta \xs| \leq \| \delta u \|_{\U}^2$.
This implies $\| \w_l \|_{C^0([-\el,0];\R^n)} = o(\| \delta u \|_{\U})$.
Analogously, 
we obtain $\| \w_r \|_{C^0([0,\el];\R^n)} = o(\| \delta u \|_{\U})$.
By \eqref{pde_difference_nodecond}, $Q_i^j(\cdot, 0,0,0)$ is the first-order Taylor
remainder of $G_i^j$ at the traces of $\y$ on $\Sigma_j$. Since these and the traces
of $\tilde{y}$ lie in the convex set $\mathcal{B}^j_{\varrho_*}$ by
\eqref{traces_in_ball}, the bound $\mathrm{Lip} ( D G_i^j ) \leq C_G$ from
\eqref{Gij_derivative_bounds} and
$\| \tilde{y} - \y \|_{\pcont} = \mathcal{O}(\| \delta u \|_{\U})$ by
\cref{thm:GRP_lipschitz_pcont} yield
\begin{equation*}
	\| Q_i^j(\cdot, 0,0,0) \|_{C^0([0,T])} 
	= \mathcal{O} \big( \| \delta u \|_{\U}^2 \big) \quad \forall i \neq j \in [n].
\end{equation*}
To bound the source term \eqref{pde_difference_sourceterm}, we apply
\cref{lem:GRP_sourceterm_diff} and \cref{thm:GRP_lipschitz_pcont} to obtain
\begin{equation*}
\| \bar{g}^{(\tilde{y}, \xs + \delta \xs)} - \bar{g}^{(\bar{y}, \xs)}
- d_{(\bar{y}, \xs)} \bar{g}^{(\bar{y}, \xs)} \cdot (\tilde{y} - \bar{y}, \delta \xs) \|_{\pcont}
= o(\| \delta u \|_{\U}).
\end{equation*}
Applying \cref{lem:GRP_matrix_diff} and \cref{thm:GRP_cont_pdiff} shows that
\begin{equation*}
\begin{aligned}
\| (\bar{A}^{(\bar{y}, \xs)} - \bar{A}^{(\tilde{y}, \xs + \delta \xs)})
( \tilde{y}_{\x} - \bar{y}_{\x} ) \|_{\pcont} & = o(\| \delta u \|_{\U}), \\
\| \bar{A}^{(\bar{y}, \xs)} 
+ d_{(\bar{y}, \xs)} \bar{A}^{(\bar{y}, \xs)}
\cdot (\tilde{y} - \bar{y}, \delta \xs) 
- \bar{A}^{(\tilde{y}, \xs + \delta \xs)} \|_{\pcont} & = o(\| \delta u \|_{\U}).
\end{aligned}
\end{equation*}
The previous two bounds implies
$ \| \bar{q}(0) \|_{\pcont} = o(\| \delta u \|_{\U})$.
\Cref{thm:sl_pc0_bound} yields
\begin{equation*}
\| \bar{w} \|_{\pcont} = \| \tilde{y} - \y - \delta \y \|_{\pcont} = o(\| \delta u \|_{\U}).
\end{equation*}
This proves the Fr\'echet differentiability of $\Sref$ in $u \in \Uad$
as asserted in \eqref{Sref_diffability}.

It remains to prove that $\Sref$ is even continuously 
Fr\'echet differentiable, i.e., it holds 
\begin{equation} \label{cont_Fdiffability}
\sup_{\| \delta u \|_{\U} = 1} 
\| (d_u \Sref(\tilde{u}) - d_u \Sref(u) ) \cdot \delta u \|_{\pcont} \to 0 
\end{equation}
for $u = (u_l, u_r, \xs) \in \Uad$ and $\tilde{u} \to u$ in $\Uad \subset \U$.
Assume for the sake of contradiction that \eqref{cont_Fdiffability} does not hold.
Then there exists a sequence 
$( \delta u^{\nu} )_{\nu \in \N} \subset \U$
with $\| \delta u^ {\nu} \|_{\U} = 1 $ and
$\delta u^{\nu} = (\delta u_l^{\nu}, \delta u_r^{\nu}, \delta \xs^{\nu})$ for all $\nu \in \N$
and another sequence 
$( u^{\nu} )_{\nu \in \N} \subset \Uad$
with $u^{\nu} = (u_l^{\nu}, u_r^{\nu}, \xs^{\nu})\to u$
in $\U$ for $\nu \to \infty$ 
such that for some constant $\epsilon > 0$ it holds that
\begin{equation} \label{continuous_Frechet_contra}
	\| (d_u \Sref( u^{\nu} ) - d_u \Sref(u) ) \cdot \delta u^{\nu} \|_{\pcont}
	\geq \epsilon \quad \forall \nu \in \N.
\end{equation}
By the first part of the proof and \cref{lem:deltay_properties},
$\delta \y^{[\nu]} \coloneqq d_u \Sref(u) \cdot \delta u^{\nu}$
is the broad solution of \eqref{pde_dy}, \eqref{sl_dy_nodecond}
with $u$ fixed and $\delta u$ replaced by $\delta u^\nu$.
Moreover, $\delta \hat y ^{[\nu]} \coloneqq d_u \Sref(u^{\nu}) \cdot \delta u^{\nu}$
is the broad solution of the same problem with 
$(u, \delta u)$ replaced by $(u^\nu, \delta u^\nu)$,
where $\y^{[\nu]} \coloneqq \Sref(u^\nu)$.
We denote the corresponding source terms by $\bar a^\nu$, $\hat a^\nu$,
interior boundary conditions by $B_i^j$, $B_i^{j, [\nu]}$
and initial states by $\delta \bar u^\nu_l, \delta \bar u^\nu_r$ and
$\delta \hat u^\nu_l, \delta \hat u^\nu_r$, respectively.

Since $\| \delta u^ {\nu} \|_{\U} = 1 $ for all $\nu \in \N$,
it is by the compact imbedding no restriction to assume that there exist 
$\delta u_l \in C^0( [-\ell, \varepsilon])$ and
$\delta u_r \in C^0( [- \varepsilon, \ell])$
such that
\begin{equation*}
	\lim_{\nu \to \infty} \| \delta u_l^\nu - \delta u_l \|_{C^0( [-\ell, \varepsilon])} = 0,
	\quad
	\lim_{\nu \to \infty} \| \delta u_r^\nu - \delta u_r \|_{C^0( [- \varepsilon, \ell])} = 0	
\end{equation*}
by going to a subsequence.
Since $\delta \xs ^\nu \in [-1, 1]$ for all $\nu \in \N$,
it can also be assumed that $\lim_{\nu \to \infty} \delta \xs ^\nu  = \delta \xs$
for some $\delta \xs \in [-1,1]$.
Let $\delta \bar u_l \in C^0([-\ell, 0])$ be defined by
\begin{equation} \label{delta_u_limit_compact}
	\delta \bar u_l (\x) = 
	\mathbf{D}_l \big( \delta u_l, \xs \big)(\x) 
	+ \mathbf{D}_l \big( u_l', \xs \big)(\x) 
	\frac{\x+\el}{\el} \delta \xs 
	\quad \forall \x \in [-\ell, 0].
\end{equation}
By assumption, it holds that $\lim_{\nu \to \infty} \xs ^\nu  = \xs$ and
$\lim_{\nu \to \infty} \| (u_l')^\nu - u_l' \|_{C^0( [-\ell, \varepsilon])} = 0$.
Together with the convergence of $\delta u_l^\nu$ and $\delta \xs^\nu$ above,
\eqref{trafo_initstate} and \eqref{def_dilation} show
\begin{equation*}
	\delta \bar{u}_l^\nu, \, \delta \hat{u}_l^\nu
	\to \delta \bar u_l \quad
	\text{uniformly on } [-\ell, 0] \text{ for } \nu \to \infty.
\end{equation*}
The argument is analogous for $\delta \bar{u}_r^\nu, \, \delta \hat{u}_r^\nu$ on $[0, \ell]$.

Applying \cref{lem:GRP_sourceterm_diff} and \cref{lem:GRP_matrix_diff} to $\bar a^\nu$ 
and additionally \cref{thm:GRP_lipschitz_pcont} and \cref{thm:GRP_cont_pdiff} 
to $\hat a^\nu$ implies for all $\bar v \in \pcont$ that
\begin{equation*}
	\lim_{\nu \to \infty} 
	\| \bar a^ \nu ( \bar v)  - \bar a (\bar v) \|_{\pcont} = 0,
	\quad 
	\lim_{\nu \to \infty} 
	\| \hat a^ \nu ( \bar v)  - \bar a (\bar v) \|_{\pcont} = 0,
\end{equation*}
where $\bar a$ denotes the source term in \eqref{pde_dy} evaluated
for the limit $\delta \xs$.
For the interior boundary conditions we have that
$B_i^{j, [\nu]} \to B_i^j$ locally uniformly on 
$[0,T] \times \R \times \R^n \times \R^n$ for all $i \neq j$,
since the bases $l^{*,j}_k, r^{*,j}_k$ do not depend on $\nu$,
$D G_i^j$ is continuous by \cref{lem:GRP_bdarycondition_small}, and
$\y^{[\nu]} \to \y$ in $\pcont$ by \cref{thm:GRP_lipschitz_pcont}.

Let $\delta \y \in \pcont$ be the unique broad solution of \eqref{pde_dy},
\eqref{sl_dy_nodecond} with the initial state $\delta \y^0(0, \cdot) = \delta \bar u_l$ with
$\delta \bar u_l$ from \eqref{delta_u_limit_compact},
and $\delta \y^n(0, \cdot) = \delta \bar u_r$ with
$\delta \bar u_r$ defined analogously to \eqref{delta_u_limit_compact}.
Since the bases from \cref{def:fixed_shock_basis} do not depend on $\nu$,
\eqref{def_theta_basisdeviation} holds uniformly in $\nu$ by \eqref{LF2_verified}
and \cref{ass:constants_small}, so that \cref{thm:sl_continuous_dep_data} implies
\begin{equation*}
	\lim_{\nu \to \infty} 
	\| \delta \y^{[\nu]} - \delta \y \|_{\pcont} = 0,
	\quad 
	\lim_{\nu \to \infty} 
	\| \delta \hat y^{[\nu]} - \delta \y \|_{\pcont} = 0.
\end{equation*}
This is a contradiction to \eqref{continuous_Frechet_contra}
and thus \eqref{cont_Fdiffability} holds.
\end{proof}

\begin{remark}
The solution operator $\Sref$ is only differentiable in the topology of $\pcont$.
Remainder terms are not small in $\pdiff$. 
\end{remark}

\subsection{Differentiability Properties in Physical Coordinates}
\label{sec:diffphys}
We prove the differentiability of the shock
curves and the objective functional w.r.t.\ the control and
present differentiability properties of the solution operator of the GRP in physical coordinates.
We denote the solution operator in physical coordinates by
\begin{equation*}
\mathcal{S} : u \in \Uad \mapsto y \in L^{\infty}(\D; \R^n)  
\quad \text{subject to } y \text{ solves \eqref{GRP}}.
\end{equation*}

\begin{corollary} \label{cor:diff_shocks}
For all $j \in [n]$, the $j$-shock curve $\xi_j(u) \in C^2([0,T])$
is continuously Fr\'echet differentiable as a map
\begin{equation*}
u \in \Uad \mapsto \xi_j(u) \in C^1([0,T]). 
\end{equation*} 
For all $\delta u = (\delta u_l, \delta u_r, \delta \xs) \in \U$ 
it holds with $\delta \y = d_u \Sref(u) \cdot \delta u$
for all $t \in [0,T]$ that
\begin{equation*} 
\left( d_u \xi_j(u) \cdot \delta u \right)(t) = 
\int_0^t \big( \nabla \lambda_j ( \y^{j-1}, \y^j )
\cdot ( \delta \y^{j-1}, \delta \y^j ) \big) (\tau, s_j \tau) \intd \tau 
+ \delta \xs.
\end{equation*}
\end{corollary}
\begin{proof}
The assertion follows by combining 
\cref{thm:differentiability_refspace} and \eqref{xij_Fdiff}. 
\end{proof}

For $t \in (0,T]$, let $\It = [-\el + \lambdamax t, \el - \lambdamax t]$
and denote the solution operator of \eqref{GRP} at time $t$ by
\begin{equation*}
\St : u \in \Uad \mapsto \mathcal{S}(u)(t, \cdot) \in BV(\It; \R^n).
\end{equation*}
In regions between the shock curves, $\St$
satisfies the same stability and differentiability properties
as in \cref{thm:GRP_lipschitz_pcont}, \cref{thm:GRP_cont_pdiff},
and \cref{thm:differentiability_refspace}
in the reference space.

\begin{theorem} \label{thm:phys_smooth_region}
	Let $u \in \Uad$ and $t \in (0,T]$ be arbitrary.
	Then 
	\begin{equation} \label{def_deltay_phys}	
		\delta y = \delta \y(t, \bar{x}^{(\y, \xs)}(t, \cdot)) 
		+ \y_{\x} (t, \bar{x}^{(\y, \xs)}(t, \cdot)) 
		(d_{(\y, \xs)} \bar{x}^{(\y, \xs)} \cdot (\delta \y, \delta \xs) ) (t, \cdot) 
	\end{equation}
	with $\y = \Sref(u)$ and $\delta \y = d_u \Sref(u) \cdot \delta u$
	for $\delta u \in \U$
	defines a bounded linear operator
	\begin{equation} \label{T_deltay_operator}
		\delta u \in \U \mapsto T(u) \cdot \delta u \coloneqq \delta y
		\in L^{\infty}(\It). 
	\end{equation}	
	Let $\mathcal{K} = [\alpha, \beta] \subset \It$
	and $\U ' \subset \U$ be an open set with $u \in \U' \subset \Uad$.
	Assume that there exists no $j \in [n]$
	with $\xi_j(u')(t) \in \mathcal{K}$	for any $u' \in \U'$.
	Then the map
	\begin{equation} \label{smooth_region_regulartiy}
		u' \in \U' \mapsto Y(u') \coloneqq \St (u') | _{\mathcal{K}} \in 
		\big( C^1(\mathcal{K}), \| \cdot \|_{C^m(\mathcal{K})} \big)
	\end{equation}
	is continuously Fr\'echet differentiable for $m=0$ and continuous for $m=1$.
	Moreover, $\delta Y = d_u Y(u) \cdot \delta u \in C^0(\mathcal{K})$
	is given by
	\begin{equation*}
		\delta Y (x) = \delta y (x) \quad \forall x \in \mathcal{K}.
	\end{equation*}
	For all $u, \, \tilde u = u + \delta u \in \Uad$ with 	
	$x_j = \xi_j(u)(t)$ and $\tilde{x}_j = \xi_j(\tilde{u})(t)$,
	it holds on the set
	\begin{equation} \label{set_C_phys_coord}
		\mathcal{C}(u, \tilde u) 
		\coloneqq \It \setminus ( \cup_{j=1}^n I(x_j, \tilde{x}_j) )
	\end{equation}		
	with
	$I(x_j, \tilde{x}_j) \coloneqq [\min(x_j, \tilde x_j), 
		\max(x_j, \tilde x_j)] \subset \It$
	and $y = \St(u)$, $\hat{y} = \St(\tilde{u})$ that
	\begin{equation} \label{phys_diff_smoothparts}
		\|\hat{y} - y - \delta y \|_{C^0( \mathcal{C}(u, \tilde u))} 
		= o ( \| \delta u \|_{\U}) \quad \text{for }
		\| \delta u \|_{\U} \to 0. 
	\end{equation} 	
\end{theorem}
\begin{proof}
It immediately follows from \cref{thm:differentiability_refspace},
the boundedness of $\y \in \pdiff$, 
and \cref{lema:spacetrafo_inverse_properties}
that $T(u) \in \mathcal{L}(\U; L^{\infty}(\It) )$.
	
By \cref{cor:diff_shocks}, we can choose a constant $c > 0$ 
depending on $\mathcal{K}$ such that
$\xi_j(\tilde u)(t) \notin \mathcal{K}$ for all $\tilde u \in \Uad$ satisfying 
$\| \tilde u - u \|_{\U} \leq c$.
\Cref{lem:lemma_rescaledsolution_Dj} implies for $y = \St (u), \, \hat y = \St (\tilde u)$
with $\y = \Sref(u), \, \tilde{y} = \Sref(\tilde u)$ that 
for some $j \in [n]_0$ it holds that
\begin{equation} \label{phys_to_ref_trafo}
	y(x) = \bar{y}^j (t, \bar{x}^{(\bar{y}, \xs),j}(t,x)),
	\quad 
	\hat y(x) = \tilde{y}^j (t, \bar{x}^{(\tilde{y}, \tilde \xs),j}(t,x))
	\qquad \forall x \in \mathcal{K}.
\end{equation}

We first prove the continuity of \eqref{smooth_region_regulartiy} for $m=1$.
For this, let $\tilde u \in \Uad$ with $\| \tilde u - u \|_{\U} \leq c$ be arbitrary.
Combining \cref{lema:spacetrafo_inverse_properties} and the a-priori bound \eqref{ref_ql_pc1_apriori_bound},
the identity \eqref{phys_to_ref_trafo} shows that \eqref{smooth_region_regulartiy}
is continuous for $m=0$.
Moreover, differentiating \eqref{phys_to_ref_trafo} w.r.t.\ $x$ implies
for all $x \in \mathcal{K}$ that
\begin{equation*}  
	y_x(x) = \bar{y}_{\x} (t, \bar{x}^{(\bar{y}, \xs)}(t,x))
		\bar{x}_x^{(\bar{y}, \xs)}(t,x),
	\quad 
	\hat y_x(x) = \tilde{y}_{\x} (t, \bar{x}^{(\tilde{y}, \tilde \xs),j}(t,x))
		\bar{x}_x^{(\tilde{y}, \tilde \xs),j}(t,x).
\end{equation*}
This shows by using the identity \eqref{spacetrafo_identities},
then applying \cref{lema:spacetrafo_properties}, using \cref{thm:GRP_cont_pdiff},
and the fact that $\bar{y}_{\x}$ is uniformly continuous on the compact set 
$\mathcal{K}$, that $\| \hat y_x - y_x \|_{C^0(\mathcal{K})} \to 0$
for $\| \delta u \|_{\U} \to 0$.
To prove the continuous Fr\'echet differentiability
of \eqref{smooth_region_regulartiy} for $m=0$, 
it suffices to prove \eqref{phys_diff_smoothparts}, since
$\mathcal{K} \subset \mathcal{C}(u, \tilde u)$ holds for all
$\tilde{u} \in \Uad$ with $\| \tilde u - u \|_{\U} \leq c$.
To show \eqref{phys_diff_smoothparts}, let $\tilde u \in \Uad$
be arbitrary and $\delta u = \tilde u - u \in \U$.
By definition of the set $\mathcal{C} = \mathcal{C}(u, \tilde u)$
in \eqref{set_C_phys_coord}, both
$\hat y = \St(\tilde{u}), \, y = \St(u)$ do not have
discontinuities in $\mathcal{C}$. 
It follows from $\delta \y, \y_{\x} \in \pcont$
that $\delta y$ from \eqref{def_deltay_phys} is continuous
on $\mathcal{C}$.
We abbreviate
$\x = \x^{(\y, \xs)}(t,x)$, $\tilde{x} = \x^{(\tilde{y}, \tilde \xs)}(t,x)$,
and $ \delta \bar{x}^{(\y, \xs)}
= ( d_{(\y, \xs)} \bar{x}^{(\y, \xs)} \cdot (\delta \y, \delta \xs) )(t, x)$
for all $x \in \mathcal{C}$ and obtain
\begin{equation} \label{phys_diff_t2}
| \hat{y} - y - \delta y | (x) 
\leq | \tilde{y}(t, \bar{x}) - \y(t, \bar{x}) - \delta \y(t, \bar{x}) | 
+ | \tilde y(t, \tilde{x}) - \tilde y(t, \x) 
- \y_{\x} (t, \x) \delta \bar{x}^{(\y, \xs)} | .
\end{equation}
\Cref{thm:differentiability_refspace} implies 
$ | \tilde{y}(t, \bar{x}) - \y(t, \bar{x}) - \delta \y(t, \bar{x}) |
\leq \| \tilde y - \y - \delta \y \|_{\pcont} = o( \| \delta u \|_{\U})$.
Since $(t,\x) , \, (t, \tilde{x}) \in \dom_j$ are located in the same 
sector $\dom_j$ by construction of $\mathcal{C}$,
it holds with the abbreviation $\tilde{x}(s):=(1-s) \x+s \tilde{x}$ that
\begin{equation} \label{phys_diff_t1}
	\begin{aligned}
		& \tilde y(t, \tilde{x}) - \tilde y(t, \x) 
		- \y_{\x} (t, \x) \delta \bar{x}^{(\y, \xs)}(t, \x) 
		 = \y_{\x}(t, \x) 
                ( \tilde x - \bar x - \delta \bar{x}^{(\y, \xs)} )\\
		&+ \int_0^1 ( \tilde y_{\x}(t,\tilde{x}(s))
			- \y_{\x}(t, \tilde{x}(s)) ) \ds (\tilde x - \bar x)
		+ \int_0^1 ( \y_{\x}(t, \tilde{x}(s))
		- \y_{\x}(t, \x) ) \ds  (\tilde x - \bar x).
	\end{aligned}
\end{equation}
The Lipschitz property \eqref{xbar_trafo_lipschitz} together with \cref{thm:GRP_lipschitz_pcont} implies
$| \tilde x - \bar x | = \mathcal{O} ( \| \delta u \|_{\U})$, and the continuous
Fr\'echet differentiability from \cref{lema:spacetrafo_inverse_properties}
together with \cref{thm:differentiability_refspace} implies
$| \tilde x - \bar x - \delta \bar{x}^{(\y, \xs)} | = o ( \| \delta u \|_{\U})$,
both uniformly for all $x \in \mathcal{C}$.
Using \cref{thm:GRP_cont_pdiff}
and that $\y_{\x} \in \pcont$ is uniformly continuous
on $\dom_j$, \eqref{phys_diff_t1} shows
\begin{equation*}
	\tilde y(t, \tilde{x}) - \tilde y(t, \x) 
	- \y_{\x} (t, \x) \delta \bar{x}^{(\y, \xs)}(t, \x) 
	= o ( \| \delta u \|_{\U}) 
\end{equation*}
uniformly for all $x \in \mathcal{C}$.
Therefore, it follows from \eqref{phys_diff_t2}
that \eqref{phys_diff_smoothparts} holds.
\end{proof}

The previous results imply that $\mathcal{S}_t$
is shift-differentiable in the sense of \cite{Ulbrich2002}.

\begin{lemma} \label{lem:shiftdiff}
Fix $t \in (0,T]$ and $u = (u_l, u_r, \xs) \in \Uad$.
We define an operator 
\begin{equation} \label{shift_derivative_operator}
	T_s(u) : \U \to L^{r}(\It) \times \R^n 
\end{equation}
for any $r \in [1, \infty]$
which is defined for all
$\delta u \in \U$ by
$T_s(u) \cdot \delta u = (\delta y, \delta x)$,
where $\delta y = T(u) \cdot \delta u $ 
with $T(u)$ from \eqref{T_deltay_operator}, 
and $\delta x _j = ( d_u \xi_j(u) \cdot \delta u )(t)$ 
for all $j \in [n]$.
Then, \eqref{shift_derivative_operator} defines
a bounded linear operator
$T_s(u) \in \mathcal{L}(\U; L^{r}(\It) \times \R^n)$
for all $r \in [1, \infty]$. 
With the shock positions $x_j = \xi_j(u)(t)$ for all $j \in [n]$,
\begin{equation} \label{shift_operator}
\mathbf{S}^{x}_{\St(u)} = \delta y + 
\sum_{j=1}^n \llbracket \St(u)(x_j) \rrbracket
\mathrm{sgn}(\delta x_j) \mathbf{1}_{I(x_j, x_j + \delta x_j)}
\in L^1(\It; \R^n)
\end{equation}
and $I(\alpha, \beta) = [\min(\alpha, \beta), \max(\alpha, \beta)]$
for $\alpha, \beta \in \R$, it holds that
\begin{equation} \label{firstorder_shiftdiff}
\| \St(u + \delta u) - \St(u) - \mathbf{S}^{x}_{\St(u)}
\|_{L^1(\It;\R^n)} = o ( \| \delta u \|_{\U}) 
\quad \text{for } \U \ni \delta u \to 0.
\end{equation}
\end{lemma}
\begin{proof}
\Cref{thm:phys_smooth_region} 
and \cref{cor:diff_shocks} imply that
$T_s(u) \in \mathcal{L}(\U; L^{r}(\It) \times \R^n)$ for all $r \in [1, \infty]$.
To prove \eqref{firstorder_shiftdiff},
let $u,  \tilde{u} \coloneqq u + \delta u \in \Uad$.
For all $j \in [n]$, let $x_j = \xi_j(u)(t)$ and $\tilde{x}_j = \xi_j(\tilde{u})(t)$.
We abbreviate $y = \St(u)$, $\hat{y} = \St(\tilde{u})$,
$(\delta y, \delta x) = T_s(u) \cdot \delta u$,
and denote $ \mathcal{C} = \It \setminus ( \cup_{j=1}^n I(x_j, \tilde{x}_j) ) $.
Then we obtain
\begin{equation} \label{l1norm_decomp}
\begin{aligned}
\| \hat{y} - y - \mathbf{S}^{x}_{\St(u)} \|_{L^1(\It)}  
& \leq \| \hat{y} - y - \delta y \|_{L^1(\mathcal{C})} 
+ \sum_{j=1}^n  \| \delta y \|_{L^1( I(x_j, \tilde{x}_j) )} \\
& + \sum_{j=1}^n  \| \hat{y} - y - 
\mathrm{sgn}(\delta x_j) \llbracket y(x_j) \rrbracket \|_{L^1( I(x_j, \tilde{x}_j) )} \\
& + \sum_{j=1}^n \big| \llbracket y(x_j) \rrbracket \big| 
\| \mathbf{1}_{I(x_j, \tilde{x}_j)}
- \mathbf{1}_{I(x_j, x_j + \delta x_j)} \|_{L^1(\It)} .
\end{aligned}
\end{equation}
From \eqref{phys_diff_smoothparts} it immediately follows that
$\| \hat{y} - y - \delta y \|_{L^1(\mathcal{C})} = o (\| \delta u \|_{\U} )$.
For the second summand $\| \delta y \|_{L^1( I(x_j, \tilde{x}_j) )}$
in \eqref{l1norm_decomp}, we use that the shock positions
are Lipschitz continuous w.r.t.\ the control by \cref{cor:diff_shocks}.
Since $\| \delta y \|_{L^{\infty}(\It)} = \mathcal{O}(\| \delta u \|_{\U} )$
by the first part of the proof, it follows that
\[
	\| \delta y \|_{L^1( I(x_j, \tilde{x}_j) )} 
	\leq  \| \delta y \|_{L^{\infty}(\It)} |\tilde{x}_j - x_j| 
	= \mathcal{O}(\| \delta u \|^2_{\U} ) .
\]
For the last summand of \eqref{l1norm_decomp},
we use \eqref{ref_ql_pc1_apriori_bound} and \cref{cor:diff_shocks} to obtain
\[
\big| \llbracket y(x_j) \rrbracket \big| 
\| \mathbf{1}_{I(x_j, \tilde{x}_j)}
- \mathbf{1}_{I(x_j, x_j + \delta x_j)} \|_{L^1(\It)}
\leq 2 C_{\Sref} | \tilde{x}_j - x_j - \delta x_j| 
= o (\| \delta u \|_{\U} ) . 
\]
For the third summand of \eqref{l1norm_decomp}, we first consider
the case $\mathrm{sgn}(\delta x_j)\ne \mathrm{sgn}(\tilde{x}_j - x_j)$.
Then $|\tilde{x}_j - x_j|\le |\tilde{x}_j - x_j - \delta x_j| = o (\| \delta u \|_{\U} )$
by \cref{cor:diff_shocks} and thus by the boundedness of the solution operator $\St$ in $L^\infty$,
the third term is $o (\| \delta u \|_{\U} )$.
Now let $\mathrm{sgn}(\delta x_j) = \mathrm{sgn}(\tilde{x}_j - x_j)$.
Let the notation $x_j\pm$ and $x_j\mp$ mean $x_j+$ and $x_j-$ if
$\delta x_j>0$ and $x_j-$ and $x_j+$ else. 
It is easy to see that for all $x\in I(x_j,\tilde x_j)$
\[
 \hat y(x)-y(x)-\mathrm{sgn}(\delta x_j) \llbracket y(x_j) \rrbracket
 =\hat y(x)-\hat y(\tilde x_j\mp)+\hat y(\tilde
x_j\mp)-y(x_j\mp)+y(x_j\pm)-y(x).
\]
Since $\hat y,y$ are uniformly bounded in $C^1$ on $I(x_j,\tilde x_j)$
by \cref{thm:grp_sol_physcoord} and \cref{lema:spacetrafo_properties}
using \eqref{phys_to_ref_trafo}, we obtain
\[
 \|\hat y-\hat y(\tilde x_j\mp)+y(x_j\pm)-y\|_{L^1(I(x_j,\tilde x_j))}
 \le C |\tilde x_j-x_j|^2= \mathcal O (\| \delta u \|_{\U}^2 ).
\]
Finally, denote the corresponding states in the reference space by
$\y = \Sref(u)(t, \cdot)$ and $\tilde{y} = \Sref(\tilde{u})(t, \cdot)$. 
By \cref{thm:GRP_lipschitz_pcont},
$|\hat y(\tilde x_j\mp)-y(x_j\mp)|=|(\tilde{y}-\y)(s_j t\mp)|=\mathcal O(\| \delta u \|_{\U})$,
and thus $\|\hat y(\tilde x_j\mp)-y(x_j\mp)\|_{L^1(I(x_j,\tilde x_j))}
=\mathcal O(|\tilde x_j-x_j|\| \delta u \|_{\U})= \mathcal O (\| \delta u \|_{\U}^2 )$.
Hence, we have shown that the right-hand side of \eqref{l1norm_decomp}
is $o (\| \delta u \|_{\U} )$, which proves \eqref{firstorder_shiftdiff}.
\end{proof}

\begin{corollary} \label{cor:lipschitzL1}
For all $t \in (0,T]$,
$\St: \Uad \to L^1(\It; \R^n)$ is Lipschitz continuous.
\end{corollary}
\begin{proof}
This follows from \eqref{firstorder_shiftdiff}, since
$\| \mathbf{S}^{x}_{\St(u)} \|_{L^1(\It;\R^n)} = \mathcal{O} ( \| \delta u \|_{\U})$.
\end{proof}

\begin{theorem} \label{thm:diff_objective}
Let $t \in (0,T]$ and $[a,b] \subset \It$.
Moreover, let $\Phi \in C^{1,1}_{\mathrm{loc}}(\R^n \times \R^n)$
and $y_d \in BV([a,b]; \R^n)$.
Fix $\bar{u} \in \Uad$ and assume
$x_j \coloneqq \xi_j(\bar{u})(t) \in (a,b)$ for all $j \in [n]$.
\begin{enumerate}
	\item If $y_d$ is approximately continuous in all $x_j$, 
		the functional $J$ from \eqref{objective_fcntal}
		is Fr\'echet differentiable at $\bar{u}$.	
		For all $\delta u \in \U$ it holds with $y = \St(\bar u)$,
		$\delta y$ from \eqref{def_deltay_phys} and
		$\delta x_j = ( d_u \xi_j(\bar u) \cdot \delta u )(t)$ that
		\begin{equation} \label{duJ}
			\begin{aligned}
				d_u J(\bar u) \cdot \delta u
				& = \int_a^b \Phi_y \big( y(x), y_d(x) \big)^\top \delta y(x) \, \dx \\
				& + \sum_{j=1}^n \delta x_j \int_0^1
				\Phi_y \big( y(x_j+) + \tau \llbracket y(x_j) \rrbracket , y_d(x_j) \big)^\top 
				\intd \tau \llbracket y(x_j) \rrbracket ,
			\end{aligned}
		\end{equation}
		where $\Phi_y$ denotes the partial derivative of $\Phi$ w.r.t.\ its first argument.
	\item If $y_d$ is continuous in a neighborhood of all $x_j$,
	$J$ is even continuously Fr\'echet differentiable at $\bar{u}$.
	\item If at least one $x_j$ is an approximate discontinuity of $y_d$,
	$J$ is directionally differentiable at $\bar{u}$
	and the directional derivative is given by \eqref{duJ}
	with $y_d(x_j)$ replaced by	
	$y_d \big( x_j + 0 \cdot \mathrm{sgn}(\delta x_j) \big)$.	
	\end{enumerate}
\end{theorem}
\begin{proof}
The assertions directly follow from
\cref{lem:shiftdiff}, 
\cite[Lemma 2.3]{Ulbrich2002} and the observation that 
$T_s(u) \in \mathcal{L}(\U; L^{r}(\It) \times \R^n)$
is continuous w.r.t.\ $u$ for all $r \in (1, \infty)$.
\end{proof}

For the next result, we denote by $\mathcal{M}(\It)^n$ 
the space of $\R^n$-valued Radon measures on $\It$ with the dual pairing
$\langle f, \mu \rangle = \sum_{i=1}^n \int_{\It} f_i \intd \mu_i$
for all $\mu = (\mu_i)_{i \in [n]} \in \mathcal{M}(\It)^n$
and $f \in C^0(\It; \R^n)$ with components $f = (f_i)_{i \in [n]} \in C^0(\It)^n$.

\begin{theorem} 
	Fix $t \in (0,T]$ and $u \in \Uad$.
	Let $y = \St(u)$ and $x_j = \xi_j(u)(t)$ for all $j \in [n]$. 
	For $\delta u \in \U$, let $\delta \mu$ be given with 
	$\delta y$ from \eqref{def_deltay_phys} by
	\begin{equation} \label{delta_mu}
		\delta \mu = \delta y + 
		\sum_{j=1}^n \big( ( d_u \xi_j(u) \cdot \delta u )(t) \big) 
		\llbracket y(x_j) \rrbracket \delta(\cdot - x_j)
		\in \mathcal{M}(\It)^n,
	\end{equation}
	where $\delta(\cdot)$ denotes a Dirac measure at $0$.
	Then there is a constant $C \geq 0$ 
	such that it holds for all $f \in C^0(\It; \R^n)$
	and all $\delta u \in \U$ with $u + \delta u \in \Uad$ that
	\begin{equation} \label{diff_measure_remainder}
		\big| \big\langle f, \St(u + \delta u) - \St(u) - \delta \mu \big\rangle \big| \,
		\begin{cases}
			\leq C \| f \|_{C^0(\It;\R^n)} \| \delta u \|_{\U}, & \\
			= o ( \| \delta u \|_{\U} ) & \text{for } \| \delta u \|_{\U} \to 0 ,
		\end{cases}
	\end{equation}
	where $o ( \| \delta u \|_{\U} )$ is
	uniform on bounded and equicontinuous subsets of $C^0(\It; \R^n)$.		
\end{theorem}
\begin{proof}
	With \eqref{ref_ql_pc1_apriori_bound} and $T_s(u)$ from \cref{lem:shiftdiff},
	it holds that
	\begin{equation*}
		\begin{aligned}
			| \langle f, \delta \mu \rangle |
			& \leq \| f \|_{C^0} \| \delta y \|_{L^1(\It)} 
			+ 2 \sum_{j=1}^n | (d_u \xi_j(u) \cdot \delta u )(t) | \| \Sref(u) \|_{\pcont} |f(x_j)| \\
			& \leq \| f \|_{C^0} \| T_s(u) \|_{\mathcal{L}(\U; L^{1}(\It) \times \R^n)} \| \delta u \|_{\U} 
			(1 + 2 C_{\Sref}) 
		\end{aligned}
	\end{equation*}	
	for all $f \in C^0(\It; \R^n)$.
	Using the Lipschitz continuity of $\St: \Uad \to L^1(\It; \R^n)$
	by \cref{cor:lipschitzL1},
	this proves the first line of \eqref{diff_measure_remainder}.	
	To prove the second line of \eqref{diff_measure_remainder}, 
	let	
	$\delta x_j	= ( d_u \xi_j(u) \cdot \delta u )(t)$.
	With $\mathbf{S}^x_{\St(u)}$ from \eqref{shift_operator}, we rewrite
	\begin{equation} \label{mwstar_split}
		\langle f, \St(u + \delta u) - y - \delta \mu \rangle
		= \langle f, \St(u + \delta u) - y - \mathbf{S}^x_{\St(u)} \rangle
		+ \langle f, \mathbf{S}^x_{\St(u)} - \delta \mu \rangle .
	\end{equation}
	From \eqref{firstorder_shiftdiff} it follows that 
	\begin{equation} \label{mwstar_split_summand1}
		| \langle f, \St(u + \delta u) - y - \mathbf{S}^x_{\St(u)} \rangle | 
		\leq \| f \|_{C^0} \, o ( \| \delta u \|_{\U} )
		\quad \text{for } \| \delta u \|_{\U} \to 0.
	\end{equation}
	For the second summand on the right-hand side 
	of \eqref{mwstar_split} and $\| \delta u \|_{\U}$
	sufficiently small such that $x_j + \delta x_j \in \It$ for all $j \in [n]$,
	\eqref{shift_operator} and \eqref{delta_mu} imply
	\begin{equation} \label{mwstar_rectangle}
		\langle f, \mathbf{S}^x_{\St(u)} - \delta \mu \rangle
		= \sum_{j=1}^n \int_{x_j}^{x_j + \delta x_j}
		\big( f(x) - f(x_j) \big)^\top \llbracket y(x_j) \rrbracket \dx .
	\end{equation}
	Since $| \delta x_j | = \mathcal{O} ( \| \delta u \|_{\U} )$ by \cref{cor:diff_shocks},
	this shows for any bounded and equicontinuous set $\mathcal F \subset C^0(\It;\R^n)$ that
	$\sup_{f \in \mathcal F} \sup_{x \in I(x_j, x_j + \delta x_j)}
	|f(x) - f(x_j)| \to 0$ as $\| \delta u \|_{\U} \to 0$.
	With \eqref{mwstar_split_summand1} in \eqref{mwstar_split},
	this proves the second line of \eqref{diff_measure_remainder}.
\end{proof}

\section{Conclusion}
In this paper, we proved that the piecewise
$C^1$ entropy solution of the Generalized Riemann Problem is differentiable
w.r.t.\ the piecewise smooth parts
and the position of the discontinuity of the initial state
after transforming it into a reference space. 
An immediate consequence was the differentiability of the 
position of shock curves w.r.t.\ the controls.
By combining these results, the differentiability of a
large class of tracking-type functionals followed.

For general strictly hyperbolic systems, we had to assume that
the initial jump is sufficiently small, but for particular systems the
assumptions can also hold for large jumps. Moreover, we had to choose the
time horizon small enough such that no new shock formations occur. 
Our results also cover piecewise $C^1$ initial data with several jumps, as
long as no shock interactions occur. We think that under natural assumptions our differentiability results
can be extended beyond shock interactions. We did not consider rarefaction
waves, since it turns out that differentiability results for rarefaction waves
in a GRP with $C^1$ data are quite involved and have to be treated in a
separate paper.

The results of this paper can be used to derive an adjoint gradient representation for the objective
functional, which requires a careful definition of the adjoint state. Due to
the discontinuous state, the adjoint equation--a linear hyperbolic system
with discontinuous coefficients--has many solutions and the
correct adjoint state has to satisfy a particular interior boundary condition along
shocks. We refer to \cite{BreitkopfUlbrich2026} for results on existence, uniqueness and regularity of appropriate
reversible solutions of the adjoint equation, which build on results of this paper and provide
an adjoint gradient representation for the objective functionals.

\bibliographystyle{siamplain}
\bibliography{references}
\end{document}

%% file: ex_shared.tex
\usepackage{amsfonts}
\usepackage{graphicx}
\usepackage{epstopdf}
\usepackage{algorithmic}
\usepackage{mathtools}
\usepackage{amssymb}
\usepackage{amsmath}
\usepackage{commath}

\ifpdf
  \DeclareGraphicsExtensions{.eps,.pdf,.png,.jpg}
\else
  \DeclareGraphicsExtensions{.eps}
\fi

\newsiamremark{remark}{Remark}
\newsiamremark{hypothesis}{Hypothesis}
\newsiamremark{assumption}{Assumption}
\crefname{hypothesis}{Hypothesis}{Hypotheses}
\newsiamthm{claim}{Claim}

\headers{A Sensitivity Calculus for the GRP}{Jannik Breitkopf, Stefan Ulbrich}

\title{A Sensitivity Calculus for 
	Optimal Control of the Generalized Riemann Problem for 
	Hyperbolic Systems of Conservation Laws\thanks{\textbf{Funding:}
This work was supported by the German Research Foundation (DFG) within 
the collaborative research center TRR 154 ``Mathematical modeling, simulation and optimization using
the example of gas networks``, project A02.}}

\author{Jannik Breitkopf\thanks{Department of Mathematics, TU Darmstadt, Germany
		(\email{breitkopf@mathematik.tu-darmstadt.de}, \email{ulbrich@mathematik.tu-darmstadt.de}).}
\and Stefan Ulbrich\footnotemark[2]}

\usepackage{amsopn}

\makeatletter
\newcommand*{\addFileDependency}[1]{
  \typeout{(#1)}
  \@addtofilelist{#1}
  \IfFileExists{#1}{}{\typeout{No file #1.}}
}
\makeatother
